\documentclass[12pt, a4paper, parskip=half, abstracton, bibliography=totoc]{scrartcl}
\pdfoutput=1
\usepackage{array}
\usepackage{marginnote}
\usepackage{xcolor}
\usepackage{hyperref}
\usepackage{amscd,amssymb,amsfonts,amsmath,latexsym,amsthm}
\usepackage[capitalize]{cleveref}
\usepackage[all,cmtip]{xy}
\usepackage{mathrsfs}
\usepackage{graphicx}
\usepackage{bm} 
\usepackage{mathtools} 

\usepackage{etoolbox}

\let\counterwithout\relax

\usepackage{chngcntr} 

\usepackage{defs_pp1}
\usepackage{slashed}
\usepackage[utf8]{inputenc}
\usepackage{microtype}
\usepackage[english]{babel}

\title{Transfers in coarse homology}
\date{\today}
\author{
Ulrich Bunke\thanks{Fakult{\"a}t f{\"u}r Mathematik,
Universit{\"a}t Regensburg,
93040 Regensburg,
Germany\newline
ulrich.bunke@mathematik.uni-regensburg.de} 
\and Alexander Engel\thanks{Fakult{\"a}t f{\"u}r Mathematik,
Universit{\"a}t Regensburg,
93040 Regensburg,
Germany\newline
alexander.engel@mathematik.uni-regensburg.de
	}
\and
Daniel Kasprowski\thanks{
	Rheinische Friedrich-Wilhelms-Universit\"at Bonn, Mathematisches Institut, Endenicher Allee 60,\newline 53115 Bonn, Germany\newline
	kasprowski@uni-bonn.de
}
\and
Christoph Winges\thanks{
		Rheinische Friedrich-Wilhelms-Universit\"at Bonn, Mathematisches Institut, Endenicher Allee 60,\newline 53115 Bonn, Germany\newline
	winges@math.uni-bonn.de}
}

\numberwithin{equation}{section}
\counterwithout{footnote}{section}

\newtheorem{theorem}{Theorem}[section] 
\newtheorem{prop}[theorem]{Proposition}

\newtheorem{lem}[theorem]{Lemma}
\newtheorem{kor}[theorem]{Corollary}
\theoremstyle{remark}
\theoremstyle{definition}
\newtheorem{ddd-alt}[theorem]{Definition}
\newtheorem{ex-alt}[theorem]{Example}
\newtheorem{rem-alt}[theorem]{Remark}

\newenvironment{ddd}    
{%
	\pushQED{\qed}\begin{ddd-alt}}
	{\popQED\end{ddd-alt}}

\newenvironment{ex}    
{%
	\pushQED{\qed}\begin{ex-alt}}
	{\popQED\end{ex-alt}}

\newenvironment{rem}    
{%
	\pushQED{\qed}\begin{rem-alt}}
	{\popQED\end{rem-alt}}

\newcommand{\Ho}{\mathbf{Ho}}

\newcommand{\free}{\mathrm{free}}

\newcommand{\Prr}{\mathbf{Pr}}
\newcommand{\bd}{\mathrm{bd}}

\newcommand{\Tw}{\mathbf{Tw}}
 \newcommand{\Add}{\mathbf{Add}}

\newcommand{\pt}{\mathrm{pt}}

\DeclareMathOperator{\Lan}{Lan}
\DeclareMathOperator{\Nat}{Nat}

\DeclareMathOperator{\Mor}{Mor}
\DeclareMathOperator{\Yo}{Yo}

\DeclareMathOperator{\yo}{yo}

\DeclareMathOperator{\Res}{Res}

\newcommand{\Orb}{\mathbf{Orb}}

\newcommand{\BC}{\mathbf{BornCoarse}}

\newcommand{\Fin}{\mathbf{Fin}}

\newcommand{\Ob}{\mathrm{Ob}}

\newcommand{\wt}{\widetilde}
\newcommand{\hlg}{\mathrm{hlg}}
\DeclareMathOperator{\Cofib}{Cofib}

\DeclareMathOperator{\Fib}{Fib}

\newcommand{\cP}{\mathcal{P}}

\newcommand{\sCat}{{\mathbf{sCat}}}

\newcommand{\cL}{{\mathcal{L}}}
\newcommand{\cW}{{\mathcal{W}}}
\newcommand{\PSh}{{\mathbf{PSh}}}

\newcommand{\bA}{{\mathbf{A}}}

\newcommand{\bK}{{\mathbf{K}}}

\newcommand{\cO}{{\mathcal{O}}}
\newcommand{\cU}{{\mathcal{U}}}
\newcommand{\cY}{{\mathcal{Y}}}

 \newcommand{\Cat}{{\mathbf{Cat}}}

\newcommand{\eff}{\mathrm{eff}}
\newcommand{\Mack}{\mathbf{Mack}}

\newcommand{\cE}{{\mathcal{E}}}

\newcommand{\Sol}{\mathbf{Sol}}

\newcommand{\Coarse}{\mathbf{Coarse}}
\newcommand{\Spc}{\mathbf{Spc}}

\begin{document}
\maketitle

\begin{abstract}
We enlarge the category of bornological coarse spaces by   adding transfer morphisms and introduce the notion of an
equivariant coarse homology theory with transfers. We then 
   show that  equivariant coarse algebraic $K$-homology and equivariant coarse ordinary homology can be extended to 
   equivariant coarse homology theories with transfers. 
   In the case of a finite group we observe that equivariant coarse homology theories with transfers provide Mackey functors.
We express standard constructions with Mackey functors in  terms of coarse geometry, and we demonstrate the usage of transfers
in order to prove injectivity results about assembly maps.
\end{abstract}

\tableofcontents
 
\section{Introduction}
 {In order to  capture the large scale and the local finiteness behavior of metric spaces, groups, and other geometric objects,    the category $\BC$ of bornological coarse spaces with proper controlled maps as morphisms was introduced in \cite{buen}. 
In the present paper we will work in the equivariant situation. So
let $G$ be a group and $G\BC$ denote the category of $G$-bornological coarse spaces  \cite[Sec.~2]{equicoarse}. Let further $\bC$ be some cocomplete stable $\infty$-category. Following \cite[Sec.~3]{equicoarse} an equivariant $\bC$-valued coarse homology theory is a functor \[E\colon G\BC\to \bC\]  which satisfies four axioms:  {\begin{enumerate} \item coarse invariance,  \item  excision,  \item  vanishing on flasques,  \item  and $u$-continuity. \end{enumerate}} In \cite[Def. 4.9]{equicoarse} we construct a universal $G$-equivariant  coarse homology theory \[\Yo^{s}\colon G\BC\to G\Sp\cX\] whose target is the presentable stable $\infty$-category of equivariant coarse motivic spectra.
Any other equivariant coarse homology theory factorizes in an essentially unique way over the universal example $\Yo^{s}$. More precisely, 
precomposition with $\Yo^{s}$ induces  an equivalence from the $\infty$-category  
\[\Fun^{\colim}(G\Sp\cX,\bC)\]  of colimit-preserving functors 
to  the $\infty$-category of $\bC$-valued equivariant coarse homology theories  \cite[Cor.~4.10]{equicoarse}. 

The main goal of the present paper is to add transfers  as a new type of morphisms between bornological coarse spaces and to show that important examples of coarse homology theories extend to transfers.  We will furthermore construct the universal  equivariant coarse homology theory with transfers.
  To this end we enlarge the category $G\BC$ of $G$-bornological coarse spaces to the  category $G\BC_{tr}$ of $G$-bornological coarse spaces with transfers (see {Section~\ref{olergregregegerg}}). 

Given a $G$-set $I$ and a $G$-bornological coarse space we can form the $G$-bornological coarse space
$I_{min,min}\otimes X$ (see \cite[Ex.~2.17]{equicoarse}), or equivalently, the bounded union $\coprod_{i\in I}^{\bd}X$ of~$I$ copies of $X$ (\cref{ddd_bounded_unions}). If $i$ is a $G$-fixed point in $I$, then
   $j_{i}\colon X\to  \coprod_{i\in I}^{\bd}X$ denotes the inclusion  of the component with index $i$ which is  a morphism in $G\BC$. In general, if $i$ is not fixed by $G$, then we can consider this morphism after forgetting the $G$-action.

 By design (see \cref{ergioer434t43t34t}) $G\BC_{tr}$  contains a transfer morphism 
 \[\tr_{X,I}\colon X\to  I_{min,min}\otimes X\] which morally
is the sum $\sum_{i\in I} j_{i}$ of the inclusion morphisms. It will actually turn out that $G\BC_{tr}$ is semi-additive and therefore enriched in commutative monoids. 
 If~$I$ is finite and has the trivial $G$-action, then \[\tr_{X,I}=\sum_{i\in I}  j_{i}\] is a literally  true identity in $G\BC_{tr}$. But transfers are most interesting in the case of infinite sets $I$.

 If $E$ is an equivariant coarse homology theory, then the construction of an extension of $E$ to  $G\BC_{tr}$ should be guided by the idea that the morphism 
 \[E(\tr_{X,I})\colon E(X)\to E( I_{min,min}\otimes X)\] places   identical copies of a cycle for $E(X)$ on each component of the bounded union.

The projection $I_{min,min}\otimes X\to X$   is a controlled and bornological map, but not a proper map and therefore not a morphism of bornological coarse spaces. It is an example  of
a bounded covering, a notion  which we will introduce in the present paper. The category $G\BC_{tr}$ will be defined by  adding wrong-way maps for all bounded coverings. 

 {On the technical level we use spans to construct $G\BC_{tr}$ as a quasicategory (see \cref{olergregregegerg}).}
 {We further construct an embedding
\[\iota \colon G\BC\to G\BC_{\tr}\]
(see Definitions~\ref{voiehjoriverververv} and \ref{thorh5t4554t}).}

Let $\bC$ be a stable cocomplete $\infty$-category.
\begin{ddd}
A $\bC$-valued \emph{equivariant  {coarse} homology theory with transfers} is a functor \[E\colon  G\BC_{tr}\to\bC\] such that
\[E\circ \iota\colon G\BC\to \bC\] is a $\bC$-valued  equivariant  coarse homology theory.
\end{ddd}
By excison an equivariant coarse homology theory  with transfers 
preserves coproducts and is therefore an additive functor from $G\BC_{tr}$ to $\bC$.

\begin{ddd}\label{efuewhiffwefweffeff}
We will say that a $\bC$-valued  equivariant  coarse homology theory $E$ {\em admits transfers} if
there exists a functor \[E_{tr}\colon  G\BC_{tr}\to\bC\] such that $E_{tr}\circ \iota\simeq E$.
\end{ddd}

The condition that a coarse homology theory $E$ admits transfers is used in order to show a version of the coarse  Baum-Connes conjecture for $E$ and scalable spaces \cite[Sec.~10.3]{ass}.
Furthermore, the existence of transfers is an important ingredient 
 in \cite{trans} where we show that $G$-equivariant finite decomposition
 complexity of $X$ implies that a certain forget-control map 
$E(\beta_{X})$ is an equivalence.

In analogy with the universal equivariant coarse  homology theory,
we will construct  the universal
equivariant coarse homology theory with transfers
\[\Yo^{s}_{tr}\colon G\BC_{tr}\to G \Sp\cX_{tr} \ .\] 
 
Let $\bC$ be a stable, cocomplete $\infty$-category.  The next  proposition is true by design of~$\Yo^{s}_{tr}$.

\begin{prop}[\cref{rhiohff3f34f34f3f}]
Precomposition with $\Yo^{s}_{tr}$ induces  an equivalence from the $\infty$-category  
\[\Fun^{\colim}(G\Sp\cX_{tr},\bC)\] to  the   $\infty$-category of $\bC$-valued equivariant coarse homology theories with transfers.
\end{prop}

 In the present paper we consider the following examples of
   equivariant coarse homology theories: 
\begin{enumerate}
\item   equivariant coarse ordinary homology $H\cX^{G}$;
\item  equivariant coarse algebraic $K$-homology $K\bA\cX^{G}$ of an additive category $\bA$ with a strict action of $G$.
 \end{enumerate}
 Their construction is given in \cite[Sec.~7~\&~8]{equicoarse}. In this paper we are interested in the existence of transfers.

\begin{theorem}\label{riojgerogjo23rwefwefwef}
The equivariant coarse homology theories $H\cX^{G}$ and $K\bA\cX^{G}$ admit transfers.
\end{theorem}
The assertions of the theorem are shown in \cref{ergpoerpgergerg}.
The case of algebraic $K$-theory is actually quite involved 
and relies on the preparations in \cref{ergiuergergerge}.

In the final Section~\ref{sec_mackey} we show that a $\bC$-valued equivariant coarse homology theory $E$ gives rise to a $\bC$-valued Mackey functor which will be denoted by $EM$.

 If $V$ is a finite-dimensional orthogonal  representation of $G$, then we can express the delooping of $EM$ along the representation sphere $S(V)_{\infty}$ in terms of the 
 equivariant coarse homology $E_{V}$ obtained from $E$ by twisting with $V$, where $V$ is considered as a  $G$-bornological coarse  space. 
 More precisely we show the following.
  \begin{prop}[\cref{vwlkne2ofwevwevevw}]
We have a canonical equivalence of $\bC$-valued Mackey functors
\[S(V)_{\infty}\wedge EM\simeq E_{V}M\ .\] 
\end{prop}

In \cite{desc} we use transfers in order to prove injectivity results for assembly maps. In the present paper we demonstrate this method in the simple case of a finite group  $G$. Consider for example the family of solvable subgroups $\Sol$. Let $G\Orb$  be the orbit category and let $G_{\Sol}\Orb$ be  its subcategory of orbits with stabilizers in $\Sol$. We consider a cocomplete and complete stable $\infty$-category $\bC$ and a  functor
 $E \colon  G\Orb \to \bC$. The following theorem is a special case of \cref{ihwiuhvewvwevewv}.

\begin{theorem}\label{thm_intro_injective}
If
 $E$ extends to a $\bC$-valued Mackey functor, then
 the assembly map \[  \colim_{T\in G_{\Sol}\Orb}E(T )\to  E(*)\] is split injective.
\end{theorem}

\subsection*{Acknowledgements} U.B.~and A.E.~were supported by the SFB 1085 \emph{Higher Invariants} funded by the Deutsche Forschungsgemeinschaft DFG. A.E. was furthermore supported by the SPP 2026 \emph{Geometry at Infinity} also funded by DFG. C.W.~acknowledges support from the Max Planck Society. 

\section{Equivariant coarse motives with transfers}
 
 {For  an introduction to $G$-bornological coarse spaces and the associated motives we refer to \cite[Sec.~2--4]{buen} and to \cite[Sec.~2]{equicoarse}. In the present section we will discuss the new aspects related to transfers.}

 {In order to incorporate transfers for equivariant coarse homology theories, we introduce the $\infty$-category $G\BC_{tr}$ of $G$-bornological coarse spaces with transfers in \cref{olergregregegerg}. To this end, we introduce  in Section~\ref{rgioregergergregerg}   the notion of a bounded covering  which appears  in the definition of the morphisms in $G\BC_{tr}$. In Section~\ref{sec_coarse_spectra_tr} we  introduce the corresponding $\infty$-category of coarse motivic spectra with transfers, and in  Section~\ref{sec_homology_with_transfers}  we  discuss  equivariant coarse homology theories with transfers.}

\subsection{Bounded coverings and admissible squares}\label{rgioregergergregerg}

\newcommand{\Span}{\mathbf{Span}}

In order to incorporate transfers for equivariant coarse homology theories, we introduce the  {$\infty$-}category $G\BC_{tr}$ of $G$-bornological coarse spaces with transfers in \cref{olergregregegerg}. This category in particular contains for all $G$-sets $I$ transfer morphisms
\[\tr_{X,I}\colon X\to I_{min,min}\otimes X\ .\] The projection to the second factor from $I_{min,min}\otimes X$ to $X$ is a morphism of the underlying $G$-coarse spaces, but  {it is in general} not proper. The transfer is a kind of wrong-way map for  this projection.

 {In this section we will introduce} for $G$-bornological coarse spaces $W$ and $X$ the notion of a bounded covering  from $W$ to $X$ which generalizes the projection onto the second factor discussed above. By construction, the  {homotopy category of} $G\BC_{\tr}$ will have transfer maps $\tr_{w}\colon X\to W$ for all bounded coverings $w$ from $W$ to $X$, see the \cref{rgioergergrege}.

We start with recalling some basic definitions {from} coarse geometry.

\begin{ddd}\label{grejgio43t34t43}\mbox{}
\begin{enumerate}
\item A $G$-\emph{coarse space}  is a pair   $(W,\cC_{W})$  of a $G$-set $W$ and a coarse structure  {$\cC_{W}$} such that $\cC_{W}$ is $G$-invariant and the set of invariant entourages  {$\cC_{W}^G$} is cofinal in $\cC_{W}$.
\item If $(W,\cC_{W})$ and $ (W^{\prime},\cC_{W^{\prime}})$ are $G$-coarse spaces and $f\colon W\to W^{\prime}$ is an equivariant map between the underlying $G$-sets, then $f$ is called {\em controlled} if for every $U$ in $\cC_{W}$ we have $(f\times f)(U)\in \cC_{W^{\prime}}$.
\item By $G\Coarse$ we denote the category of $G$-coarse spaces and $G$-equivariant controlled maps. 
\end{enumerate}
The category $G\Coarse$ is complete and cocomplete by \cite[Prop.~2.18 and~2.21]{equicoarse}. We have a forgetful functor $G\BC\to G\Coarse$ which preserves coproducts and which sends the symmetric monoidal structure $\otimes$ on $G\BC$ to the product in $G\Coarse$.
\end{ddd}
 
Let $W$ be a set, $U$ be a subset of $W\times W$, and   $A$ be a subset of $W$.
 \begin{ddd}\label{ergioergergerg}
 The $U$-\emph{thickening} of $A$ is defined by
 \[ U[A]:=\{w\in W \mid \exists a\in A : (w,a)\in U\}\ .\qedhere\]
 \end{ddd}

Let $W$ be a  coarse space with coarse structure $\cC_{W}$. Then the union \[R_{W}:=\bigcup_{U\in \cC_{W}} U\] of all coarse entourages   of $W$  is an equivalence relation on $W$. Using this equivalence relation, we introduce the following notions.

Let $A$ and $B$ be subsets of $W$.

\begin{ddd}\label{rgiog34t34t34t3fe}
 \begin{enumerate}
   \item The \emph{coarse closure  {$[A]$}} of $A$ is the closure of $A$ with respect to the equivalence relation $R_{W}$. 
 \item If $A=[A]$, then $A$ is said to be {\em coarsely closed}.
 \item $A$ and $B$ are \emph{coarsely disjoint} if $[A]\cap [B]=\emptyset$.  \qedhere \end{enumerate}
  \end{ddd}

 If $A$ is a subset of $W$, then using \cref{ergioergergerg}, we  have
 \[ [A]=\bigcup_{U\in \cC_{W}} U[A]\ .\]

Let $W$ be a  coarse space with coarse structure $\cC_{W}$.
\begin{ddd} \label{ergopegg34t3t3} \mbox{} \begin{enumerate}
 \item The equivalence classes of $W$ with respect to the equivalence relation  $R_{W}$ are called the \emph{coarse components} of $W$.
 \item The $G$-set of coarse components of $W$ will be denoted by $\pi_{0}(W)$.  \qedhere\end{enumerate}
 \end{ddd}

In the following we discuss various ways to construct $G$-coarse spaces.
 
Let $W$ be a set  and $Q$ be a  subset of $\cP(W\times W)$.
\begin{ddd}  The \emph{coarse structure  $\cC\langle Q\rangle $  generated by $Q$} is   the smallest coarse structure on $W$ containing
the set $Q$.
\end{ddd}
If $W$ is a $G$-set and $Q$ consists of  $G$-invariant subsets, then $\cC\langle Q\rangle $ is a $G$-coarse structure. 

Let 
$U$ be a $G$-invariant entourage on a $G$-set $W$.
\begin{ddd} \label{rgijo34t34tegerg} We   let $W_{U}$ denote the $G$-coarse space
$(W,\cC\langle\{U\}\rangle)$.
\end{ddd}

Let $W$ be a $G$-set.
An \emph{equivariant partition} of $W$  is a  partition $(W_{i})_{i\in I}$ such that $I$ is a   $G$-set  and
$gW_{i}=W_{gi}$ for all $i$ in $I$ and $g$ in $G$.

Let $W$ be a $G$-set and $\cW:=(W_{i})_{i\in I}$ be an equivariant partition. Then we consider the invariant entourage
\begin{equation}\label{geh54lolj3p45g354}
U(\cW):=\bigsqcup_{i\in I} W_{i}\times W_{i}
\end{equation} on $W$. Note that we have a canonical equivariant bijection
\[ \pi_{0}(W_{U(\cW)})\cong I\ .\]

Assume now that  $W$ is a $G$-coarse space with coarse structure $\cC$ and with an equivariant partition
  $\cW:=(W_{i})_{i\in I}$.
\begin{ddd}\label{fwufhwifhfiuewfewfewfe}
We define  the  $G$-coarse    structure $\cC(\cW)$  on $W$ by
\[\cC(\cW):=\cC\langle\{U\cap U(\cW)\:|\: U\in \cC\}\rangle\ .\qedhere\]
\end{ddd}
 
Finally, let $W$ be a $G$-set, let $U$ be a $G$-coarse space with coarse structure $\cC_{U}$, and let $w \colon W\to U$ be an equivariant map of sets.
   \begin{ddd} \label{fiofjwefwefwef} The {\em induced coarse structure} $w^{-1}\cC_{U}$ 
 on $W$ is the maximal  coarse structure on $W$ such that the map $w$ is controlled.
   \end{ddd} 
   
   Note that $w^{-1}\cC_{U}$ is a $G$-coarse structure and  explicitly given  by 
   \[w^{-1}\cC_{U}=\cC\langle \{(w^{-1}\times w^{-1})(E)\:|\: E\in \cC_{U}\}\rangle\ .\]

We now turn to the definition of the notion of a bounded coarse covering.
Let $w\colon W\to U $ be a morphism of $G$-coarse spaces with coarse structures $\cC_{W}$ and $\cC_{U}$, respectively.
Let $\cW:=\pi_{0}(W)$ be the partition of $W$ into  coarse components.

\begin{ddd}\label{wefiowefewfetrz}
We say that $w$ is a \emph{bounded coarse covering} if the following conditions are satisfied:
\begin{enumerate}
\item  $(w^{-1}\cC_{U} )(\cW)=\cC_{W}$ (see \cref{fiofjwefwefwef} and \cref{fwufhwifhfiuewfewfewfe}).
\item  For every    {$W_0$ in $\pi_{0}(W)$ the map $w_{|W_0}\colon W_0\to w(W_0)$} is an isomorphism between  coarse components (see \cref{ergopegg34t3t3}).\qedhere
\end{enumerate}
\end{ddd}

Let $w\colon W\to U$ and $u\colon U\to V$ be bounded coarse coverings between $G$-coarse spaces.
\begin{lem}\label{regklergergreg}
The composition $u\circ w\colon W\to V$ is a bounded coarse covering.
\end{lem}

\begin{proof} We let $\cU$ be the partition of $U$ into  coarse components.
Then we have the following equalities:
\[ {((u\circ w)^{-1}\cC_{V})} (\cW)=(w^{-1}(u^{-1}\cC_{V} ))(\cW) =(w^{-1}( u^{-1}\cC_{V})(\cU))(\cW)=(w^{-1}\cC_{ U})(\cW)=\cC_{W}\ .\]
Here we use that the decomposition $w^{-1}\cU$ of $W$ is coarser than the decomposition $\cW$ for the second equality, and the assumption that $u$ and $w$ are bounded coarse coverings for the   third and the last equalities.

If $W_{0}$ is a coarse component in $W$, then $w$ maps it isomorphically to a   coarse component  $U_{0}$ of $U$, and $u$ maps $U_{0} $ isomorphically to a   coarse component in $V$. Hence $u\circ w$ maps $W_{0}$  isomorphically to a coarse component in $V$.
\end{proof}

We consider a cartesian diagram \[\xymatrix{W\ar[d]_{w}\ar[r]^-{f}&U\ar[d]^{u}\\V\ar[r]^-{g}&Z}\]
in $G\Coarse$.

\begin{lem}\label{oiwejofewwefewfewf}
If $u$ is a bounded coarse covering, then $w$ is a bounded coarse covering.
\end{lem}

\begin{proof}
Recall that the coarse structure $\cC_{W}$ of the space $W$ is generated by entourages of the form $w^{-1}(A)\cap f^{-1}(B)$ for entourages $A$ in $\cC_{V}$ and $B$ in $\cC_{U}$, and that the coarse structure $w^{-1}(\cC_V)(\cW)$ is generated by entourages of the form $w^{-1}(A)\cap U(\cW)$ for entourages $A$ in~$\cC_{V}$.  {Here $\cW:=\pi_{0}(W)$ is the partition of $W$ into coarse components.}

Let $\cU:=\pi_0(U)$.
	Given an entourage $A$ in $\cC_V$, define $B:=u^{-1}(g(A))\cap U(\cU)$, which is an entourage in $\cC_U$. Then we get $f^{-1}(B)=w^{-1}(A)\cap f^{-1}(U(\cU))$. Because $U(\cW)$ is contained in $f^{-1}(U(\cU))$, we have
	\[w^{-1}(A)\cap U(\cW)\subseteq w^{-1}(A)\cap f^{-1}(U(\cU)) =  {f^{-1}(B)}\]
	and hence $w^{-1}(\cC_V)(\cW)$ is contained in $\cC_W$. 
	
On the other hand, the inclusion
$\cC_{W}\subseteq w^{-1}(\cC_{V})(\cW)$ is clear.

{Let $W_{0}$ be a coarse component in $W$. We first  show that $w(W_{0})$ is a coarse component of $V$. There exists a coarse component $U_{0}$ in $U$ such that $f(W_{0})\subseteq U_{0}$.
We consider a point  $a$ in $[w(W_0)]$, and we must argue that $a\in   w(W_0)$.
  Since  $g(a)$ and $g(w(W_0))$ are in the same coarse component of $Z$,  we have $g(a)\in [u(U_{0})]$.  
  Since  $u_{|U_0}\colon U_0 \to u(U_0)$ is an isomorphism  of coarse components,   there exists
   $b$ in $U_0$ with $g(a) = u(b)$. 
The pair $(a,b)$ uniquely determines a point $c $ in $W  $. By the choice of $a$, there exists a point $c_0$ in $W_0$ such that $\{(a,w(c_0))\}$ is an entourage  of $V$    {and  \[(c,c_0) \in w^{-1}(\{(a,w(c_0))\})\ .\] 
Since $f(c_{0})\in U_{0}$ and $U_{0}$ is a coarse component, $\{(b,f(c_{0}) )\}$ is an entourage of $U$. Then \[(c,c_{0} )\in f^{-1}(\{(b,f(c_{0}))\})\ .\]
 Since the square is cartesian, the coarse structure $\cC_{W}$ of the space $W$ is generated by entourages of the form $w^{-1}(A)\cap f^{-1}(B)$ for entourages $A$ in $\cC_{V}$ and $B$ in $\cC_{U}$, and therefore $\{(c,c_0)\}$ is an entourage of $W$. Since $W_{0}$ is a coarse component and $c_{0}\in W_{0}$, we see that $ c\in W_{0}$. Hence $a=w(c)\in w(W_{0})$. This finishes the verification that $w(W_{0})$ is a coarse component.}}

 {We show that for every coarse component $W_{0}$ of $W$ the map  $w_{|W_{0}}\colon W_{0}\to w(W_{0})$ is an isomorphism of coarse components.
We first show that $w_{|W_{0}}$ is injective.  Consider two points 
  $w_0$ and $w_1$ in $W_{0}$ with $w(w_0) = w(w_1)$. Since  $u_{|[f(W_{0})]}\colon [f(W_{0})] \to u([f(W_{0})])$ is an isomorphism,  we get $f(w_0) = f(w_1)$. Since the square is cartesian, this implies $w_0 = w_1$}.
  
   We already know that $\cC_{W} = w^{-1}(\cC_{V})(\cW)$. Because $W_0$ is a coarse component of $W$, this implies $\cC_{W} \cap (W_0 \times W_0) = w^{-1}(\cC_{V})$ showing that $w_{|W_{0}}$ is an isomorphism of  coarse components. \end{proof}

  We consider a map between sets equipped with bornological structures.
  \begin{ddd}\label{ergierog34t34t34t3}
  \mbox{}
  \begin{enumerate}
  \item\label{4tgoi34t34t34t} The map
   is called \emph{bornological} if it sends bounded subsets to bounded subsets.
   \item The map
 is called \emph{proper} if preimages of bounded subsets are bounded.
\qedhere \end{enumerate}
   \end{ddd}

\begin{ddd}\label{wegioww4rwe}
Let $\wt{G\BC}$ be the category whose objects are $G$-bornological coarse spaces, and morphisms are morphisms between the underlying $G$-coarse spaces.
 \end{ddd}
 
 The forgetful functor $\wt{G\BC}\to G\Coarse$ is an equivalence of categories. Therefore $\wt{G\BC}$ has all small limits and colimits.

For spaces $X$ and $Y$ in $G\BC$ it makes sense to require that a morphism $ X\to  Y$ in $\wt{G\BC}$ is proper or bornological, or both,  as an additional property.

We consider two $G$-bornological coarse spaces {$X$ and $Y$} and a morphism $u\colon  X\to\ Y$ in $\wt{G\BC}$. 

\begin{ddd}\label{ergierog43t34t34t34t}
We say that $u$ is a {\em bounded covering} if the following conditions are satisfied:
\begin{enumerate}\item\label{frweoirhgirogregewefw} $u$ is a bounded coarse covering (see \cref{wefiowefewfetrz}).
\item \label{gropergergreg111} $u$ is a bornological map (see \cref{ergierog34t34t34t3}.\ref{4tgoi34t34t34t}).
\item \label{gropergergreg}
For every bounded subset $B$ of $X$ there exists a finite, coarsely disjoint partition
$(B_{a})_{a\in A}$ of $B$ such that $u_{|[B_{a}]}\colon [B_{a}]\to [u(B_{a})]$ is an isomorphism of  coarse spaces (see \cref{grejgio43t34t43} and \cref{rgiog34t34t34t3fe}). \qedhere
\end{enumerate}
\end{ddd}

Condition \ref{ergierog43t34t34t34t}.\ref{gropergergreg} gives that we have isomorphisms of coarse spaces 
\[ u_{|U[B_{a}]}\colon U[B_{a}]\to u(U[B_{a}]) \]
 for all coarse entourages $U$ of $X$,  see \cref{ergioergergerg}.
If $X$ has the property that a bounded set meets at most finitely many coarse components, then Condition 
 \ref{ergierog43t34t34t34t}.\ref{gropergergreg} is automatically satisfied. But it becomes relevant  if bounded sets can meet more than finitely many coarse components.
 
 Let $X,Y,W$ and $U$ be $G$-bornological coarse spaces and let $w\colon  X\to  Y$ and $u\colon  W\to  U$ be bounded coverings.
 Note that the coproduct in $G\BC$ is also  the coproduct  in $\wt{G\BC}$ and therefore we can form the morphism 
  $w\sqcup u \colon  {X\sqcup W}\to {Y\sqcup U}$ in $\wt{G\BC}$, where the coproduct of the spaces is understood in $G\BC$.
  Similarly, the underlying $G$-coarse space  of the tensor product in $G\BC$ is the product of the underlying $G$-coarse spaces. Hence we have a map $w\times u \colon  X\otimes W \to Y\otimes U$  in $\wt{G\BC}$.
  The following lemma follows directly from the definitions.
  \begin{lem}
		\label{lem:coproduct-boundedcovering}
		The maps  $w\sqcup u \colon  {X\sqcup W}\to {Y\sqcup U}$ and  $w\times u \colon   {X\otimes W}\to   {Y\otimes U}$ are bounded coverings.
	\end{lem}

\begin{proof}
 The case of $w\sqcup u$ is obvious.
 
 We consider the case of $w\times u$.
Let us first verify Condition \ref{ergierog43t34t34t34t}.\ref{frweoirhgirogregewefw}, i.e.,  that $w\times u$ is a bounded coarse covering. Indeed we have the following chain of equalities
\begin{eqnarray*}
((w\times u)^{-1}\cC_{Y\otimes U} )(\pi_{0}(X\otimes W))&=&
((w\times u)^{-1}\langle  \cC_{Y}\times \cC_{U}\rangle)(\pi_{0}(X\otimes W))\\
&=&
\langle  w^{-1}  (\cC_{Y})\times u^{-1}(\cC_{U})\rangle(\pi_{0}(X\otimes W))\\
&=& \langle  w^{-1}  (\cC_{Y})\times u^{-1}(\cC_{U})\rangle(\pi_{0}(X)\times \pi_{0}(W))\\
&=& \langle  w^{-1}  (\cC_{Y})(\pi_{0}(X))\times  u^{-1}(\cC_{U})( \pi_{0}(W))\rangle\\
&=& \langle  \cC_{X}\times   \cC_{W}\rangle\\
&=&\cC_{X\otimes W}\ .
\end{eqnarray*}
Moreover,
every coarse component $Z_0$ of $X\otimes W$ is  of the form $X_0 \times W_0$ for coarse components $X_0$ of $X$ and $W_0$ of $W$ and both $w_{|X_0}$ and $u_{|W_0}$ are isomorphisms between coarse components by assumption. Hence $(w\times u)_{|Z_0}$ is an isomorphism of coarse components.

 Conditions \ref{ergierog43t34t34t34t}.\ref{gropergergreg111} and \ref{ergierog43t34t34t34t}.\ref{gropergergreg} easily follow from the fact that the bornology on  $X \otimes W$ is  generated by $\cB_X \times \cB_W$. 
 \end{proof}

\begin{ex} \label{igjeroigegergregregreg}
Let $X$ be a $G$-coarse space and $I$ be a $G$-set. Then we can form the product $I_{min}\times X $  in $G$-coarse spaces, where $I_{min}$ is the $G$-coarse space with underlying $G$-set $I$ and  the minimal  coarse structure. The projection onto the second factor
$\pr_{2}\colon I_{min}\times X\to X$ is a bounded coarse covering.
 {If $X$ is a $G$-bornological coarse space, then $\pr_2$ is a bounded covering of $G$-bornological coarse spaces from $I_{min,min}\otimes X$ to $ X$, where $I_{min,min}$ carries the minimal coarse and bornological structures.}

More generally, assume that $X$ is a $G$-coarse space and $I\to I^{\prime}$ a map of $G$-sets.
Then the induced map $I_{min}\times X\to I_{min}^{\prime}\times X$ is a bounded coarse covering. If $X$ is a $G$-bornological coarse space, then  $I_{min,min}\otimes X\to I_{min,min}^{\prime}\otimes X$ is a bounded covering.
\end{ex}

\begin{ex}\label{rgioergergregegreg}
Let $X$ be a  {$G$-}bornological coarse space with bornology $\cB$ and assume that $\cB^{\prime}$ is a compatible $G$-bornological  structure  such that $\cB^{\prime}\subseteq \cB$. Then we consider the $G$-bornological coarse space $X^{\prime}$ obtained from $X$ by replacing
$\cB$ by $\cB^{\prime}$. Then the identity map of the underlying sets   is a bounded covering $  X^{\prime}\to  X$. 
Indeed, the identity is clearly a bounded coarse covering. The Condition~\ref{ergierog43t34t34t34t}.\ref{gropergergreg} is also satisfied (even for arbitrary subsets in place of $B$ and for the trivial partition).
 Finally, the  identity  is bornological since $\cB^{\prime}\subseteq \cB$.
 \end{ex}

We consider $G$-bornological coarse spaces $X$, $Y$ and $Z$,  and bounded coverings $u\colon  X\to  Y$ 
 and $v\colon   Y\to  Z$.
\begin{lem}
The composition $v\circ u\colon   X\to Z$ is  a bounded covering.
\end{lem}

\begin{proof}
By Lemma \ref{regklergergreg} we know that $v\circ u$ is a bounded coarse covering. Furthermore, as a composition of bornological maps it is bornological.

Let $B$ be a bounded subset {of $X$} and let $(B_{a})_{a\in A}$ be a finite, coarsely disjoint partition such that $u_{|B_{a}}\colon [B_{a}]\to [u(B_{a})]$ is an isomorphism of coarse spaces. For every $a$ in $A$ let $(C_{a,i})_{i\in I_{a}}$ be a finite, coarsely disjoint partition of $u(B_{a})$ such that
$v_{|[C_{a,i}]}\colon [C_{a,i}]\to [v(C_{a,i})]$ is an isomorphism of coarse spaces. Note that this partition exists since $u(B_{a})$ is bounded in $Y$.
 Then we set $B_{a,i}:=u^{-1}(C_{a,i})\cap B_{a}$ and observe that
$((B_{a,i})_{i\in I_{a}})_{a\in A}$ is a finite, coarsely disjoint partition of $B$ such that $(v\circ u)_{|[B_{a,i}]}\colon [B_{a,i}]\to [(v\circ u)(B_{a,i})]$ is an isomorphism of coarse spaces.
\end{proof}

We consider $G$-bornological coarse spaces $W,U,V$ and $Z$, 
and a diagram
\begin{equation}\label{wevlkwvewv}
\xymatrix{  W\ar[d]_{w}\ar[r]^-{f}&   U\ar[d]^{u}\\  V\ar[r]^-{g}&  Z}
\end{equation}
 in $\wt{G\BC}$.
\begin{ddd}\label{gergergeg}
The square \eqref{wevlkwvewv} is called {\em admissible} if the following conditions are satisfied:
\begin{enumerate}
\item\label{gioergergreg1} The square \eqref{wevlkwvewv}  is cartesian.
\item $g$ is proper and bornological. 
\item  \label{gioergergreg}$f$  is proper and bornological.  
\item $u$ is a bounded covering.\qedhere
 \end{enumerate}
 \end{ddd}
 Note that Condition \ref{gergergeg}.\ref{gioergergreg1} is equivalent to the condition that the underlying square of \eqref{wevlkwvewv} in $G\Coarse$ is cartesian.

\begin{lem}\label{gioegergergregerge}
If the square \eqref{wevlkwvewv} is admissible, then $w$ is a bounded covering.
\end{lem}

\begin{proof}
The map $w$ is a bounded coarse covering by \cref{oiwejofewwefewfewf}.

Moreover, $w$ is bornological. Indeed, let $B$ be a bounded subset of $W$. Then we have
\[w(B)\subseteq g^{-1}(u(f(B)))\ .\] Since $f$ and $u$ are bornological and $g$ is proper
we see that $g^{-1}(u(f(B)))$ and hence $w(B)$ are bounded.

We finally verify the Condition \ref{ergierog43t34t34t34t}.\ref{gropergergreg}.
Let $B$ be a bounded subset of $W$.  Then $f(B)$ is bounded in $U$  since $f$ is bornological. 
Let $(C_{a})_{a\in A}$ be a finite, coarsely disjoint partition of $f(B)$
such that $u_{|[C_{a}]}\colon [C_{a}]\to [u(C_{a})]$ is an isomorphism of coarse spaces for every~$a$ in~$A$. We define $B_{a}:=f^{-1}(C_{a})\cap B$. Then $( B_{a})_{a\in A}$ is a finite, coarsely disjoint 
partition of $B$. It suffices to show that for every $a$ in $A$ the map $w_{|[B_{a}]}\colon [B_{a}]\to [w(B_{a})]$ is injective since
$w$ is a bounded coarse covering and therefore   an isomorphism on  each coarse component of $W$.
Let $b,b^{\prime}$ be points in $[B_{a}]$ and assume that $w(b)=w(b^{\prime})$. 
Then $u(f(b))=u(f(b^{\prime}))$. Since $f(b), f(b^{\prime})\in [C_{a}]$ and $u_{|[C_{a}]}$ is injective, we conclude that that $f(b)=f(b^{\prime})$. 
 Since the square \eqref{wevlkwvewv} is a pull-back of sets, this implies $b=b^{\prime}$.
\end{proof}

We consider $G$-bornological coarse spaces $U,V,Z$
and a diagram \begin{equation}\label{wevlkwvewv1}
\xymatrix{&  U\ar[d]^{u}\\   V\ar[r]^-{g}& Z}
\end{equation}  in $\wt{G\BC}$
 such that $g$ is proper and bornological and $u$ is a bounded covering.
\begin{lem}\label{eiogergergergerg}
There exists an extension $(W,w,f)$ of \eqref{wevlkwvewv1} to an admissible square \eqref{wevlkwvewv}.

If $(W^{\prime},w^{\prime},f^{\prime})$ is a second admissible extension, then
there exists a unique isomorphism of $G$-bornological coarse spaces $\phi\colon W\to W^{\prime}$ such that 
\[\xymatrix {V\ar@{=}[d]&  W\ar[l]_{w}\ar[d]^{ \phi} \ar[r]^{f}& U\ar@{=}[d]\\
  V&  W^{\prime}\ar[l]_{w^{\prime}}\ar[r]^{f^{\prime}}&  U}\]
commutes.
\end{lem}

\begin{proof}
We choose an object $W$ representing the pull-back   $V \times_ {Z} U$
in $\wt{G\BC}$, and we can assume that  $W$ has the bornology $\cB_{W}:=f^{-1}\cB_{U}$. {This is an extension $(W,w,f)$ of \eqref{wevlkwvewv1} to an admissible square.}

Because  $W$ is a pull-back in $\wt{G\BC}$, it is unique up to unique isomorphism  in $\wt{G\BC}$.
This provides us the map $\phi$ which is an isomorphism in $\wt{G\BC}$. Since the maps $f\colon W\to U$ and $f^{\prime}\colon W^{\prime} \to U$ are proper and bornological,
the  map $\phi$ is an isomorphism of $G$-bornological coarse spaces.
\end{proof}

\subsection{The category \texorpdfstring{$G\BC_{\tr}$}{GBornCoarsetr}}
\label{olergregregegerg}
 
In this section we first introduce the category $\Ho(G\BC_{\tr})$ of $G$-bornological coarse spaces with transfers. It contains  the category $G\BC$ of $G$-bornological coarse spaces as a subcategory such that the inclusion
\begin{equation}\label{ccdwewrwerwrwererwerwerwer}
 \iota \colon G\BC\hookrightarrow \Ho(G\BC_{\tr})
\end{equation}
is a bijection on objects. We then define the quasicategory $G\BC_{\tr}$ which models the ordinary category $\Ho(G\BC_{\tr})$ as its homotopy category as indicated by the notation. Finally,  we discuss some basic properties of these categories.

Let $X$ and $Y$ be a $G$-bornological coarse spaces.
\begin{ddd}\label{GROIERG}
	A \emph{span $(W,w,f)$} from $X$ to $Y$  is a diagram  \[ \xymatrix{&W\ar@{->>}[dl]_{w}\ar[dr]^{f}&\\  X&& Y}\]
	in $\wt{G\BC}$ (see \cref{wegioww4rwe})
	  subject to the following conditions:
	\begin{enumerate}
		\item $ f$ is a morphism in $G\BC$ which is in addition bornological (see \cref{ergierog34t34t34t3}).
		\item $w \colon W\to X$ is a bounded covering (see \cref{ergierog43t34t34t34t}).
	\end{enumerate}

{We use double-headed arrows in order to indicate which map is a bounded covering.}

	An \emph{isomorphism} between spans $(W,w,f)$ and $(W^{\prime},w^{\prime},f^{\prime}) $ is defined to be an isomorphism of $G$-bornological coarse spaces $\phi \colon W\to W^{\prime}$ such that the diagram 
	\begin{equation}\label{revpokrvporekververv}
		\xymatrix{ X\ar@{=}[d]&W\ar@{->>}[l]_{w}\ar[d]^{\phi}_{\cong}\ar[r]^{f}&Y\ar@{=}[d]\\
			X&W^{\prime}\ar@{->>}[l]_{w^{\prime}}\ar[r]^{f^{\prime}}&Y}
	\end{equation}
	in $\wt{G\BC}$ commutes.
	
	We define $\Ho(G\BC_{\tr})$\footnote{Later we define a quasi-category $G\BC_{\tr}$ whose homotopy category is $\Ho(G\BC_{\tr})$ justifying this  notation, see \cref{egoiegergergergerbg}.} as the category whose objects are $G$-bornological coarse spaces and whose morphisms are isomorphism classes of spans.
	Morphisms in the category $\Ho(G\BC_{\tr})$ are called \emph{generalized morphisms} of $G$-bornological coarse spaces.
	
	The {\em composition} $(U,w\circ u, g\circ h)$ of the spans $(W,w,f)$ from $X$ to $Y$  and $(V,v,g)$ from $Y$ to $Z$ is determined by the choice of a span $(U,u,h)$ such that the square in the diagram
	\begin{equation}\label{fvelkvervrev}
		\xymatrix{&&U\ar@{->>}[dl]_{u}\ar[dr]^{h}&&\\
			&W\ar[dr]^{f}\ar@{->>}[dl]_{w}&& V\ar@{->>}[dl]_{v}\ar[dr]^{g}\\
			 X&&Y&&Z}
	\end{equation}
	is admissible {(Definition~\ref{gergergeg}).} 
\end{ddd}

Compositions in the category $\Ho(G\BC_{\tr})$ always exist and are well-defined by Lemmas~\ref{gioegergergregerge} and \ref{eiogergergergerg}.

\begin{ddd}\label{voiehjoriverververv}
	We define the embedding 
	\[ \iota \colon G\BC\to \Ho(G\BC_{\tr}) \]
	as follows:
	\begin{enumerate}
		\item It is given by the identity on objects.
		\item It sends the morphism $f \colon X\to Y$ to the generalized morphism represented by the span
		\[ \xymatrix{&\hat X\ar@{->>}[dl]_{\id}\ar[dr]^{\hat f}&\\X&&Y\ ,}\]
		where $\hat X$ is the $G$-bornological coarse space obtained from the space $X$ by replacing its bornology by the bornology $f^{-1}\cB_{Y}$, the right leg is induced by $f$, and the left leg is induced by the identity of underlying coarse spaces. \qedhere
	\end{enumerate}
\end{ddd}

Note that $\hat f$ in \cref{voiehjoriverververv} is proper and bornological by construction.  {The bornology $f^{-1}\cB_{Y}$ on $\hat X$ is compatible with the coarse structure of $\hat X$, because $f$ is controlled.} Since $f$ is proper, the left leg is bornological.
The left leg is a bounded covering by \cref{rgioergergregegreg}.
It is easy to see that the inclusion $\iota \colon G\BC\to \Ho(G\BC_{\tr})$ is a functor.

We will denote the generalized morphism represented by the span $(W,w,f)$ by $[W,w,f]$. For a $G$-bornological coarse space $X$ in $G\BC$ we will use the symbol $X$ also to denote the object $\iota(X)$ of $G\BC_{\tr}$. Furthermore, for a morphism $f$ in $G\BC$ we will keep the short notation $f$ for the generalized morphism $\iota(f)=[\hat X,\id,\hat f]$.

Let $W$ and $X$ be $G$-bornological coarse spaces and $w \colon W\to X$ be a bounded covering.

\begin{ddd}\label{rgioergergrege}
	The morphism 
	\[ \tr_{w}:=[W,w,\id_{W}] \colon X\to W \]
	in $\Ho(G\BC_{\tr})$ is called the {\em transfer} for $w$.
\end{ddd}

We will in particular need the following special case.
Let $X$ be a $G$-bornological coarse space and $I$ be a $G$-set. By \cref{igjeroigegergregregreg} the projection onto the second factor 
\[ u \colon I_{min,min}\otimes X\to X \]
is a bounded covering. 

\begin{ddd}\label{ergioer434t43t34t}
	The generalized morphism
	\[ \tr_{X,I}:=[I_{min,min}\otimes X,u,\id_{I_{min,min}\otimes X}] \colon X\to I_{min,min}\otimes X \]
	is called the {\em transfer} for $I$.
\end{ddd}

We define now a quasicategory $G\BC_{\tr}$ which models the ordinary category $\Ho(G\BC_{\tr})$ introduced in \cref{GROIERG} as its homotopy category.

Recall \cref{wegioww4rwe} of  the category $\widetilde{G\BC}$. 
We will describe $G\BC_{\tr}$ as a simplicial subset of $\Hom_{\Cat}(\Tw ,\widetilde{G\BC})$, where $\Tw\colon \Delta\to \Cat$ denotes the cosimplicial category with $\Tw[n]$ the twisted arrow category of the poset $[n]$. Our approach is similar to the construction of the effective Burnside category of a disjunctive triple in \cite{Barwick:2014aa}, but it is formally not a special case.

\begin{rem}\label{rioug9erger34t}
In this remark we recall the definition of $\Tw $, see also \cite[Sec.~2]{Gepner:2015aa} or \cite[Sec.~2]{Barwick:2014aa}, and provide an explicit description of $\Fun(\Tw ,\bC)$ for a  {small} category $\bC$.

First of all $\Tw $ is the 
 functor (compare \cite[Ex.~2.4]{Gepner:2015aa})
\[ \Tw \colon \Delta\to \Cat\ , \quad [n]\mapsto \Tw[n]\ ,\]
where $\Tw[n]$ is the poset of pairs of integers $(i,j)$ with  $0\le i\le j \le n$ such that $(i,j)\le (i^{\prime},j^{\prime})$ if and only if $i\le i^{\prime}\le j^{\prime}\le j$. If $\sigma \colon [n]\to [m]$ is a morphism in $\Delta$, then we define the morphism \[ \Tw(\sigma) \colon \Tw[n]\to \Tw[m]\ , \quad (i,j)\mapsto (\sigma(i),\sigma(j))\ .\]
For a category $\bC$ we now obtain the simplicial set
\[ \Hom_{\Cat}(\Tw ,\bC) \colon \Delta^{op}\to \Set\ .\]
In the following, we will use the following notation for the data of an $n$-simplex $X$ in $\Hom_{\Cat}(\Tw ,\bC)$. We write $X_{i,j}$ for the image under $X$ of the pair $(i,j)$ in $\Tw[n]$, and we use the shorthand $X_i$ instead of $X_{i,i}$. We will furthermore only depict the morphisms $X_{i,j}\to X_{i' ,j'}$ if $(i,j)$ and $(i',j')$ are adjacent, i.e., if $i=i'$ and $j' +1=j$ or $i' =i+1$ and $j=j'$.
Note that these morphisms $(i,j)\to (i^{\prime},j^{\prime})$ generate all morphisms in $\Tw[n]$.
\end{rem}
 
\begin{ddd}\label{rgeuouoifweffwefwe}
The simplicial set $G\BC_{\tr}$ is defined to be the subset of 
$\Hom_{\Cat}(\Tw ,\widetilde{G\BC})$ whose $n$-simplices $X$ satisfy the following:
\begin{enumerate}
 \item For every object $(i,j)$ in $\Tw[n]$ with $j\ge 1$ the morphism $X_{i,j} \to X_{i,j-1}$ is a bounded covering.
 \item For every object $(i,j)$ in $ \Tw[n]$ with $i\le n-1$ the morphism $X_{i,j} \to X_{i+1,j}$ is proper and bornological.
\item For every object $(i,j)$ with $1\le i\le j\le n-1$ the square
\[ \xymatrix{X_{i-1,j+1}\ar[r]\ar@{->>}[d]&X_{i,j+1}\ar@{->>}[d]\\X_{i-1,j}\ar[r]&X_{i,j}} \]
 is admissible.\qedhere
 \end{enumerate}
\end{ddd}

Here we use double-headed arrows in order to indicate which maps are bounded coverings.

In the following we describe the $3$-skeleton of $G\BC_{\tr}$ in terms of pictures.
These pictures are very helpful in order to see the verification of the horn-filling conditions in the proof of \cref{eriogergergerg}, but also for understanding the proof of \cref{ergioergergrege}.
 
\begin{enumerate}
 \item The $0$-simplices of $G\BC_{\tr}$ are the objects of $G\BC$. 
 \item $1$-simplices of $G\BC_{\tr}$ are spans (see \cref{GROIERG})
  \[ \xymatrix{&X_{0,1}\ar@{->>}[dl]\ar[dr]&\\X_{0}&&X_{1}}\]
  The two faces of this one-simplex are $X_0$ and $X_1$.
 \item $2$-simplices are diagrams
  
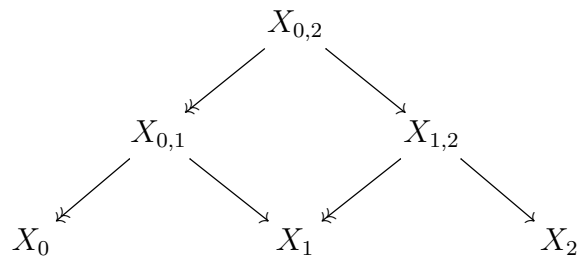
\[ \xymatrix{&&X_{0,2}\ar@{->>}[dl]\ar[dr]&&\\&X_{0,1}\ar@{->>}[dl]\ar[dr]&&X_{1,2}\ar@{->>}[dl]\ar[dr]&\\X_{0}&&X_{1}&&X_{2}}\]
  where the square is admissible.
  The three faces are
  
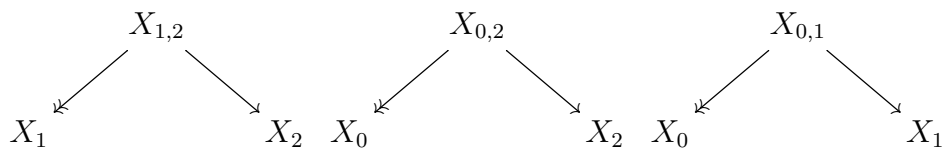
\[ 
  \xymatrix{&X_{1,2}\ar@{->>}[dl]\ar[dr]&\\X_{1}&&X_{2}}\ 
  \xymatrix{&X_{0,2}\ar@{->>}[dl]\ar[dr]&\\X_{0}&&X_{2}}\ 
  \xymatrix{&X_{0,1}\ar@{->>}[dl]\ar[dr]&\\X_{0}&&X_{1}} 
  \]
 \item $3$-simplices are diagrams 
  
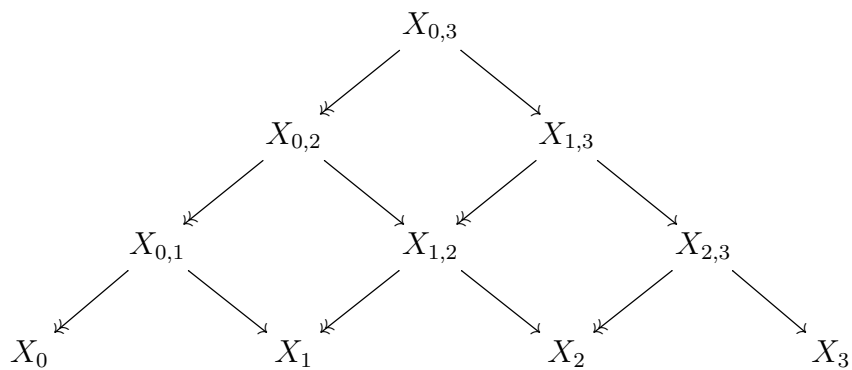
\[ \xymatrix{&&&X_{0,3}\ar@{->>}[dl]\ar[dr]&&&\\&&X_{0,2}\ar@{->>}[dl]\ar[dr]&&X_{1,3}\ar[dr]\ar@{->>}[dl]&&\\&X_{0,1}\ar@{->>}[dl]\ar[dr]&&X_{1,2}\ar@{->>}[dl]\ar[dr]&&X_{2,3}\ar@{->>}[dl]\ar[dr]&\\X_{0}&&X_{1}&&X_{2}&&X_{3}}\] 
  where again all squares are admissible.
  Its faces are
  \begin{align*}
  \mathclap{
  \xymatrix{&&X_{1,3}\ar@{->>}[dl]\ar[dr]&&\\&X_{1,2}\ar@{->>}[dl]\ar[dr]&&X_{2,3}\ar@{->>}[dl]\ar[dr]&\\X_{1}&&X_{2}&&X_{3}}\   \xymatrix{&&X_{0,3}\ar@{->>}[dl]\ar[dr]&&\\&X_{0,2}\ar@{->>}[dl]\ar[dr]&&X_{2,3}\ar@{->>}[dl]\ar[dr]&\\X_{0}&&X_{2}&&X_{3}} 
  }
  \end{align*}
  \begin{align*}
  \mathclap{
  \xymatrix{&&X_{0,3}\ar@{->>}[dl]\ar[dr]&&\\&X_{0,1}\ar@{->>}[dl]\ar[dr]&&X_{1,3}\ar@{->>}[dl]\ar[dr]&\\X_{0}&&X_{1}&&X_{3}} \ 
  \xymatrix{&&X_{0,2}\ar@{->>}[dl]\ar[dr]&&\\&X_{0,1}\ar@{->>}[dl]\ar[dr]&&X_{1,2}\ar@{->>}[dl]\ar[dr]&\\X_{0}&&X_{1}&&X_{2}}
  }
  \end{align*}
\end{enumerate}
 
\begin{lem}\label{lem:coskeletal}
The simplicial set  $G\BC_{\tr}$ is $2$-coskeletal.
\end{lem}

\begin{proof}
We observe that the data of an $n$-simplex is given by the collection of data of all $2$-simplices in the $n$-simplex.
Hence the restriction map
\[ \Hom_{\sSet} (\Delta^{n},G\BC_{\tr}) \to \Hom_{\sSet}(\Delta^{n}_{\le 2},G\BC_{\tr})\]
is an isomorphism for all $n\ge 3$, where $\Delta^{n}_{\le 2}$ denotes the $2$-skeleton of $\Delta^{n}$.
\end{proof}

\begin{lem}\label{eriogergergerg}
The simplicial set $G\BC_{\tr}$ is a quasi-category.
\end{lem}

\begin{proof}
We must check the inner horn filling condition.
\begin{enumerate}
 \item The image of $\Lambda^{2}_{1}$ in $\Delta^{2}$ has the form:
  
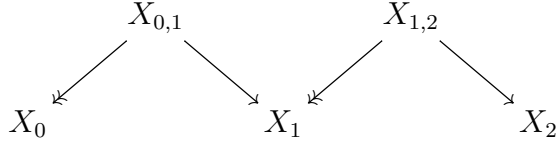
\[ \xymatrix{&X_{0,1}\ar@{->>}[dl]\ar[dr]&&X_{1,2}\ar@{->>}[dl]\ar[dr]&\\X_{0}&&X_{1}&&X_{2}} \]
  By Lemma \ref{eiogergergergerg} it has a filling.
 \item The image of $\Lambda^{3}_{2}$ in $\Delta^{3}$ is the bold part of the following diagram:  
  
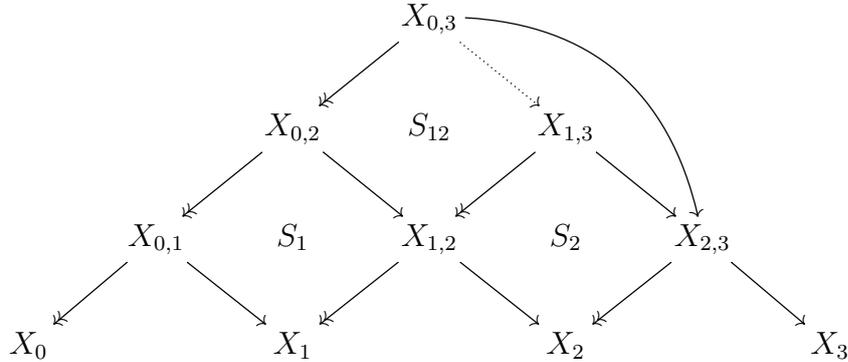
\[ \xymatrix{&&& X_{0,3}\ar@/^1cm/[ddrr]\ar@{->>}[dl]\ar@{..>}[dr]&&&\\&&X_{0,2}\ar@{->>}[dl]\ar[dr]&S_{12}&X_{1,3}\ar[dr]\ar@{->>}[dl]&&\\&X_{0,1}\ar@{->>}[dl]\ar[dr]&S_1&X_{1,2}\ar@{->>}[dl]\ar[dr]&S_2&X_{2,3}\ar@{->>}[dl]\ar[dr]&\\X_{0}&&X_{1}&&X_{2}&&X_{3}}\]
  We first get the dotted arrow using the cartesian property of the square $S_{2}$. 
  We further know that the squares $S_{1}$, $S_{2}+S_{12}$ and $S_{2}$ are admissible. We must show that $S_{12}$ is admissible.
  Since $S_{2}+S_{12}$ and $S_{2}$ are cartesian in $\wt{G\BC}$ we conclude that $S_{12}$ is cartesian in $\wt{G\BC}$. 
  Since the maps $X_{1,3}\to X_{2,3}$ and $X_{0,3}\to X_{2,3}$ are bornological and proper, also the map $X_{0,3}\to X_{1,3}$ is bornological and proper. This implies that $S_{12}$ is admissible.
  
  A similar argument applies to the inclusion of $\Lambda^{3}_{1}$ into $\Delta^{3}$.
 \item Since every inner horn $\Lambda^n_k$ for $n \ge 4$ already contains the full 2-skeleton, it is fillable by \cref{lem:coskeletal}.\qedhere 
  \end{enumerate}
\end{proof} 
 
 The following lemma justifies the choice of notation $\Ho(G\BC_{\tr})$ for the category introduced in \cref{GROIERG}.
 
\begin{lem}\label{egoiegergergergerbg}
The category $\Ho(G\BC_{\tr})$ is canonically equivalent to the homotopy category
of $G\BC_{\tr}$.
 \end{lem}
\begin{proof}
The equivalence is given by the  functor described as follows:
\begin{enumerate}
\item The functor is 
the obvious bijection on objects.
\item  The functor sends the class $[W,w,f]$ of spans from $X$ to $Y$ to the class of $(W,w,f)$ in the homotopy category of  $ G\BC_{\tr}$.
  \end{enumerate}

We first argue that the functor is well-defined  on morphisms.
If $\phi \colon (W,w,f)\to (W^{\prime},w^{\prime},f^{\prime})$ is an isomorphism between spans, then we can consider the diagram
\[ \xymatrix{&&W^{\prime} \ar@{->>}[dl]^{\phi}\ar[dr]^{f^{\prime}}&&\\&W\ar@{->>}[dl]^{w}\ar[dr]&&Y\ar@{->>}[dl]^{\id_{Y}}\ar[dr]^{\id_{Y}}&\\X &&Y&&Y}\  \]
which provides a homotopy between the morphisms $(W,w,f)$ and $(W^{\prime},w^{\prime},f^{\prime})$ in the left mapping space $\Hom^L_{ G\BC_{\tr}}(X,Y)$.

One easily checks the compatibility with composition so that we have a well-defined functor. It is furthermore obvious that 
the functor  is full.
  
On the other hand, homotopies of spans in the left mapping space $\Hom^L_{G\BC_{\tr}}(X,Y)$ are precisely of the above form.
Because $\phi$ is a pullback of an isomorphism, $\phi$ defines an isomorphism of spans. 
This shows that the functor 
 is also faithful.
\end{proof}

\begin{rem} 
A higher categorical refinement of $G\BC_{\tr}$ can also be obtained in the form of a bi-category $G\BC_{\tr}^{bi}$.
Because $\wt{G\BC}$ admits fibre products, we can form the bi-category $\Span(\wt{G\BC})$ of spans in $\wt{G\BC}$  \cite{bena}.
We obtain $G\BC_{\tr}^{bi}$ from $\Span(\wt{G\BC})$ by the following steps which all yield bi-categories.
\begin{enumerate}
\item In a first step we take a subcategory by requiring the left legs of the spans to be bounded coverings and the right legs to be proper and bornological. Compositions still exist by \cref{eiogergergergerg} in connection with \cref{regklergergreg},  \cref{oiwejofewwefewfewf}, \cref{gioegergergregerge}.

Identity morphisms belong to our category.
All relations involving $2$-isomorphisms are automatically implemented by morphisms between $G$-bornological coarse spaces. 
\item  The bi-category $G\BC^{bi}_{\tr}$ is defined to be the subcategory whose $2$-morphisms between spans are implemented by morphisms of $G$-bornological coarse spaces. \qedhere
\end{enumerate}
\end{rem}

According to \cite[Def. 6.1.6.13]{HA} an $\infty$-category is called semi-additive if it is pointed, and finite coproducts and products exist and are equivalent.
 
\begin{lem}\label{hfweuihfweifefefeef}
$\Ho(G\BC_{\tr})$ and $G\BC_{\tr}$ are  semi-additive.
\end{lem}

\begin{proof}
We show that the empty space $\emptyset$ is both initial and final in $G\BC_{\tr}$. Let $X$ be a $G$-bornological coarse space. 
We will use the simplicial set  of right morphisms $\Hom_{G\BC_{\tr}}^{R}(\emptyset,X)$ (see \cite[Sec.~1.2.2]{htt} for details). 
$\Hom_{G\BC_{\tr}}^{R}(\emptyset,X)$ is the one-point space. To see this note that e.g.\ the unique $2$-simplex in this simplicial set is given by 
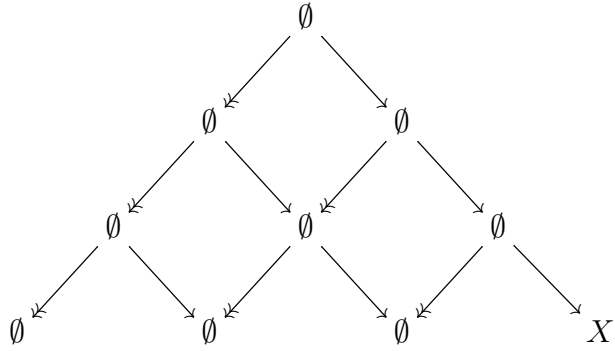
\[ \xymatrix{&&&\emptyset\ar@{->>}[dl]\ar[dr]&&&\\&&\emptyset\ar@{->>}[dl]\ar[dr]&&\emptyset\ar[dr]\ar@{->>}[dl]&&\\&\emptyset\ar@{->>}[dl]\ar[dr]&&\emptyset\ar@{->>}[dl]\ar[dr]&&\emptyset\ar@{->>}[dl]\ar[dr]&\\\emptyset&&\emptyset&&\emptyset&&X}\] 
This shows that $\emptyset$ is an initial object. 

In order to see that $\emptyset$ is also final we use the simplicial set $\Hom_{G\BC_{\tr}}^{L}(X, \emptyset )$ of left morphisms and again observe that it is a one-point space.

Hence also the homotopy category $\Ho(G\BC_{\tr})$ of $G\BC_{\tr}$
is pointed.
 
Since semi-additivity can be checked on the level of homotopy categories by \cite[Rem. 6.1.6.15]{HA}, it remains to check that $\Ho(G\BC_{\tr})$ is semi-additive.

We   show that $\Ho(G\BC_{\tr})$ admits finite products and coproducts, and that they are naturally isomorphic.

We first claim that the inclusion $\iota \colon \BC\to \Ho(\BC_{\tr})$ preserves finite coproducts.
Let $X$ and $Y$ be $G$-bornological spaces. Then we have a coproduct $X\sqcup Y$ in $G\BC$ together with canonical morphisms
\[ i \colon X\to X\sqcup Y\:\: \mbox{and}\:\: j \colon Y\to X\sqcup Y\ .\]
Let now $Z$ be a $G$-bornological coarse space and \[[W,w,f] \colon X\to Z\:\: \mbox{and}\:\: [V,v,g] \colon Y\to Z\] be generalized morphisms.
They extend uniquely to a generalized morphism
\[[W\sqcup V,w\sqcup v,f+g] \colon X\sqcup Y\to Z\ .\]
Note that $w\sqcup v$ is a bounded covering by Lemma~\ref{lem:coproduct-boundedcovering}.
Then 
\[ [W\sqcup V,w\sqcup v,f+g]\circ i=[W,w,f]\:\: \mbox{and}\:\:[W\sqcup V,w\sqcup v,f+g]\circ j=[V,v,g]\ .\]

We have generalized morphisms
\[ p:=[X,i,\id_{X}] \colon X\sqcup Y\to X\:\:\mbox{and} \:\: q:=[Y,j,\id_{Y}] \colon X\sqcup Y\to Y \ .\]
We claim that $p$ and $q$ exhibit $X\sqcup Y$ as the product of $X$ and $Y$ in $\Ho(G\BC_{\tr})$. Let
\[ [A,a,s] \colon Q\to X\:\: \mbox{and}\:\: [B,b,t] \colon Q\to Y\]
be generalized morphisms.
There is a unique generalized morphism
\[[A\sqcup B,a\sqcup b,s+t] \colon Q\to X\sqcup Y\ .\]
Then
\[ p\circ [A\sqcup B,a\sqcup b,s+t] =[A,a,s]\:\: \mbox{and}\:\: q\circ [A\sqcup B,a\sqcup b,s+t]  =[B,b,t]\ .\]
\end{proof}

Lemma \ref{hfweuihfweifefefeef} implies that the embedding $G\BC\to \Ho(G\BC_{\tr})$ does not preserve products.

Let $i$ be an element of $I$ which is fixed by $G$ and set $I^{\prime}:=I\setminus \{i\}$. Then we have the equality
\begin{equation}\label{fwefewlkml243rt4}
\tr_{X,I}=\tr_{X,I^{\prime}} + j_{i}
\end{equation}
in $\Hom_{\Ho(G\BC_{\tr})}(X, I_{min,min}\otimes X)$, where the embedding $j_{i} \colon X\to I_{min,min}\otimes X$ is induced by the inclusion $\{i\}\to I$.
We furthermore have a generalized morphism
\begin{equation}\label{rbmmnmmrjkkj}
p_{i}:=[  X, j_i, \id_{X} ] \colon I_{min,min}\otimes X\to X
\end{equation}
called the projection onto the $i$-th component of $I_{min,min}\otimes X$ such that $p_{i}\circ j_{i}=\id_{X}$.

\begin{ddd}\label{thorh5t4554t}
	We define the {\em canonical embedding}
	\begin{equation}\label{gr3ioegb433g}
	\iota \colon \Nerve(G\BC)\to G\BC_{\tr}\ .
	\end{equation}
	as the natural refinement of \cref{voiehjoriverververv}.
\end{ddd}

This canonical embedding sends, e.g., the $3$-simplex
\[ X\xrightarrow{f} Y\xrightarrow{g} Z\xrightarrow{h} U\]
in $\Nerve(G\BC)$ to the $3$-simplex
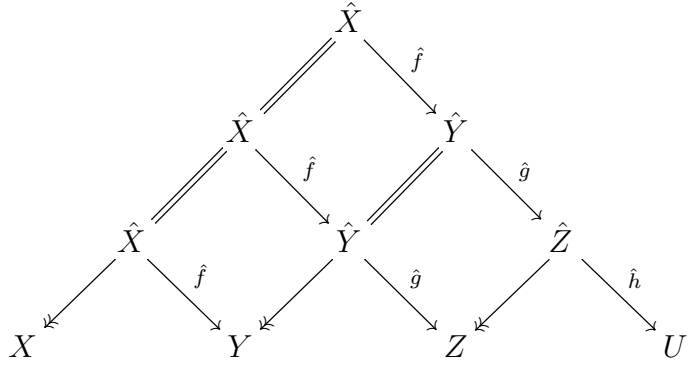
\[ \xymatrix{&&&\hat X \ar@{=}[dl]\ar[dr]^{\hat f}&&&\\&&\hat X \ar@{=}[dl]\ar[dr]^{\hat f}&&\hat Y\ar[dr]^{\hat g}\ar@{=}[dl]&&\\&\hat X\ar@{->>}[dl]\ar[dr]^{\hat  f}&&\hat Y\ar@{->>}[dl] \ar[dr]^{\hat g}&&\hat Z\ar@{->>}[dl] \  \ar[dr]^{\hat h}&\\X &&Y &&Z &&U}\]
in $G\BC_{\tr}$, where we use the notation introduced in \cref{voiehjoriverververv}.
 
\begin{ex}\label{hio3t45ththr}
Let $Q$ be a $G$-bornological coarse space. If
\[ w \colon W\to X \]
is a bounded covering between $G$-bornological coarse spaces, then
\[ w\times \id_{Q} \colon W\otimes Q\to X\otimes Q \]
 is again a bounded covering between $G$-bornological coarse spaces by \cref{lem:coproduct-boundedcovering}. 
Furthermore, if the diagram
\begin{equation}\label{wevlkwv222ewv}
\xymatrix{W\ar@{->>}[d]_{w}\ar[r]^{f}& U\ar@{->>}[d]^{u} \\
V\ar[r]^{g}&Z}
\end{equation}
is an admissible square of $G$-bornological coarse spaces, then the square
\[\xymatrix{
W\otimes Q\ar@{->>}[d]_{w\times \id_{Q}}\ar[rr]^-{f\times \id_{Q}} & & U\otimes Q\ar@{->>}[d]^{u\times \id_{Q}} \\
 V\otimes Q\ar[rr]^-{g\times \id_{Q}} & & Z\otimes Q
}\]
is admissible, too. We therefore get a functor 
\[ -\otimes Q\colon  G\BC_{\tr}\to  G\BC_{\tr}\ . \]

This construction actually produces a bifunctor
\begin{equation}\label{clkjcoecwecwecwcec}
G\BC_{\tr}\times G\BC\to G\BC_{\tr}\ .
\end{equation}
To illustrate this, we show what this functor does on 2-simplices. Given a 2-simplex
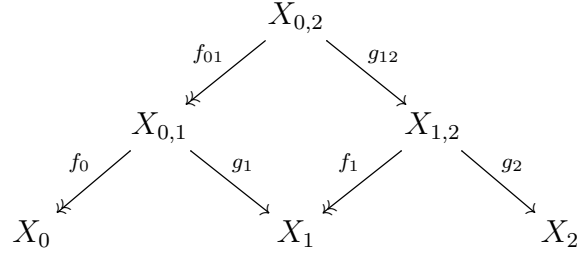
\[\xymatrix{&&X_{0,2}\ar@{->>}[dl]_{f_{01}}\ar[dr]^{g_{12}}&&\\&X_{0,1}\ar@{->>}[dl]_{f_0}\ar[dr]^{g_1}&&X_{1,2}\ar@{->>}[dl]_{f_1}\ar[dr]^{g_2}&\\X_{0}&&X_{1}&&X_{2}}\]
in $G\BC_{\tr}  $ and a composition $Q_0\xrightarrow{a_1} Q_1\xrightarrow{a_2} Q_2$ in $G\BC$ we obtain a new 2-simplex
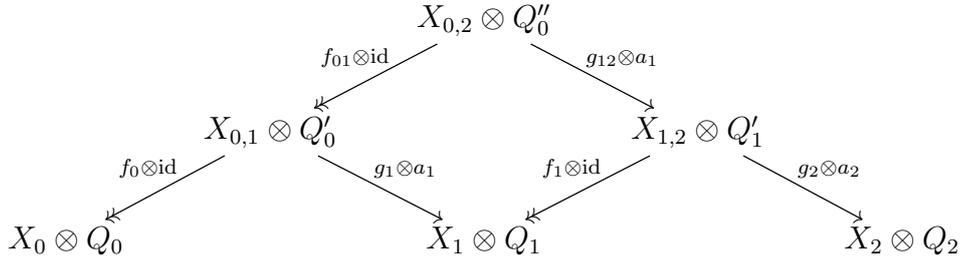
\[\xymatrix{
&&X_{0,2}\otimes Q_0'' \ar@{->>}[dl]_{f_{01}\otimes \id}\ar[dr]^{\ g_{12}\otimes a_1 }&&\\&X_{0,1}\otimes Q_0'\ar@{->>}[dl]_{f_{0}\otimes\id}\ar[dr]^{\ g_1\otimes a_1}&&X_{1,2}\otimes Q_1'\ar@{->>}[dl]_{f_{1}\otimes \id}\ar[dr]^{\ g_2\otimes a_2}&\\X_{0}\otimes Q_0&&X_{1}\otimes Q_1&&X_{2}\otimes Q_2
}\]
where $Q_0',Q_0''$ and denote $Q_0$ with the bornology changed to $a_{1}^{-1}(\cB_{Q_1})$ and $(a_2\circ a_1)^{-1}(\cB_{Q_2})$ respectively and $Q_1'$ denotes $Q_1$ with the bornology changed to $a_2^{-1}(\cB_{Q_2})$. That all arrows pointing to the left are bounded coverings follows from \cref{rgioergergregegreg} and \cref{lem:coproduct-boundedcovering}.
\end{ex}

\subsection{Coarse motivic spectra with transfers}
\label{sec_coarse_spectra_tr}

In this section we define the category $G\Sp\cX_{\tr}$ of coarse motivic spectra with transfers.
We closely follow  \cite[Sec.~3 and~4]{buen} and \cite[Sec.~4.1]{equicoarse}.

We start with the category \[\PSh(G\BC^{}_{\tr}):=\Fun(G\BC^{op}_{\tr},\Spc)\] of space-valued presheaves on $G\BC_{\tr}^{}$.

\begin{rem}
The  canonical embedding  \[\iota\colon \Nerve(G\BC)\to G\BC^{}_{\tr}\]  (see \cref{thorh5t4554t}) induces a restriction
\[\PSh(G\BC^{}_{\tr})\to \PSh(G\BC)\ .\] 
Note that this restriction does not preserve representables.
 \end{rem}

We let \[\yo_{\tr}\colon G\BC^{}_{\tr}\to \PSh(G\BC^{}_{\tr})\] denote  the Yoneda embedding.
In the following we will omit the canonical embedding $\iota$ defined in \cref{thorh5t4554t} from the notation.

For an equivariant  big family $\cY:=(Y_{i})_{i\in I}$ \cite[Def.~3.5]{equicoarse} on a $G$-bornological coarse space  $X$ we set
\[\yo_{\tr}(\cY):=\colim_{i\in I}\yo_{\tr}(Y_{i})\ .\]
 If $X$ is a $G$-bornological coarse space and $(Z,\cY)$ is an equivariant complementary pair \cite[Def.~3.7]{equicoarse} on $X$, then
we consider the map
\begin{equation}\label{fjfwekjfwejfewiufeuiwfhiewfe}
  \yo_{\tr}(\cY)\sqcup_{\yo_{\tr}(Z\cap \cY)}\yo_{\tr}(Z)\to\yo_{\tr}(X)\ .
\end{equation}

By \cite[Thm.~5.1.5.6]{htt} for  any small $\infty$-category $\bD$ the restriction along the Yoneda embedding
induces an equivalence \[\PSh(\bD)\simeq \Fun^{\lim}(\PSh(\bD)^{op} ,\Spc)\ .\]
Consequently,  if $E$ is an object of $\PSh(G\BC^{}_{\tr})$, then we can evaluate $E$  {on} presheaves (essentially via  right Kan extension).
For a big family $\cY$ on $X$, we abbreviate \[E(\cY) := E(\yo_{\tr}(\cY))\ .\]
Then the evaluation satisfies
\[E(\yo_{\tr}(X)) \simeq E(X)\:\:\mbox{ and } \:\:E(\cY) \simeq \lim_{i \in I} E(Y_i)\ .\]

\begin{ddd}\label{fweoifewfewfew}
We say that $E$ satisfies \emph{excision}  if:
\begin{enumerate}
\item $E(\emptyset) \simeq \emptyset$ .
\item\label{fwijfoeewfewfw}    $E$ is local with respect to  the morphisms \eqref{fjfwekjfwejfewiufeuiwfhiewfe}
for every $G$-bornological coarse space $X$ with an equivariant complementary pair $(Z,\cY)$.\qedhere
\end{enumerate}
\end{ddd}

\begin{rem}
Condition \ref{fweoifewfewfew}.\ref{fwijfoeewfewfw} is equivalent to the condition that for every $G$-bornological coarse space $X$ with  an equivariant complementary pair $(Z,\cY)$ the square
\[\xymatrix{
 E(X)\ar[r]\ar[d] & E(Z)\ar[d] \\
 E(\cY)\ar[r] & E(Z \cap Y)
}\]is cartesian.

Let us define
\[E(X,\cY):=\Fib(E(X)\to E(\cY))\ .\]
Then descent is also equivalent to the condition that the natural morphism
\[E(X,\cY)\to E(Z,Z\cap \cY)\] is an equivalence for every $G$-bornological coarse space $X$ with an equivariant complementary pair $(Z,\cY)$.
\end{rem}

 The presheaves which satisfy descent are called sheaves.

We denote the full subcategory of presheaves satisfying excision (called sheaves in the following) by
\[\Sh(G\BC^{}_{\tr})\subseteq \PSh(G\BC^{}_{\tr})\ .\]
Then we have a localization
\[L\colon \PSh(G\BC^{}_{\tr})\leftrightarrows \Sh(G\BC^{}_{\tr}):inclusion\ .\]

For the following definition 
recall the definition of a flasque $G$-bornological coarse space \cite[Def.~3.8]{equicoarse}.

Moreover, $\{0,1\}_{max,max}$ denotes the $G$-bornological coarse space given by the two-element set $\{0,1\}$ with trivial $G$-action and equipped with the maximal bornological coarse structure. The projection \begin{equation}\label{f24f2oijo23r23r23}
\{0,1\}_{max,max}\to *
\end{equation} is a morphism.

 Finally, if $X$ is a $G$-bornological coarse space with coarse structure $\cC_{X}$ and if $U$ in $\cC_{X} $ is   $G$-invariant, then $X_{U}$ denotes the  $G$-bornological coarse space
obtained from $X$ by replacing the coarse structure $\cC_{X}$ by the coarse structure $\cC\langle \{U\}\rangle$   (see Definition~\ref{rgijo34t34tegerg}). If $U^{\prime}$ in $\cC_{X}^{G}$ is such that $U\subseteq U^{\prime}$, then we have morphisms
$X_{U}\to X_{U^{\prime}}\to X$ of $G$-bornological coarse spaces, all induced by the identity of the underlying set.

Let $E$ be an object of  $\Sh(G\BC^{}_{\tr})$.  
\begin{ddd} 
\begin{enumerate} \item 
   $E $ is \emph{coarsely invariant} if it is local with respect to the morphism
 \[\yo_{\tr}(\{0,1\}_{max,max}\otimes X)\to \yo_{\tr}(X)\]
 induced by \eqref{f24f2oijo23r23r23}
 for all $G$-bornological coarse spaces $X$.
 \item $E$ \emph{vanishes on flasques} if   it is local for the morphisms \[\emptyset\to \yo_{\tr}(X)\]  for all flasque $G$-bornological coarse spaces $X$.   \item $E$ is \emph{$u$-continuous} if $E$ is local for the morphisms
 \[\colim_{U\in \cC_{X}^{G}}\yo_{\tr}(X_{U})\to \yo_{\tr}(X)\] for all $G$-bornological coarse spaces $X$.   \qedhere\end{enumerate} 
 \end{ddd}

 \begin{ddd} The category of $G$-equivariant motivic coarse spaces with transfers $G\Spc \cX_{\tr}$   is defined to be the full subcategory of  $\Sh(G\BC^{}_{\tr})$ which are coarsely invariant, vanish on flasques, and which are $u$-continuous.
 \end{ddd}

 We have a localization
 \[\cL_{\tr}\colon \Sh(G\BC^{}_{\tr})\leftrightarrows G\Spc\cX_{\tr}:inclusion\ .\]
 We furthermore have a functor \begin{equation}\label{verjl34g3g}
\Yo_{\tr}:=\cL_{\tr}\circ \yo_{\tr}\colon G\BC^{}_{\tr}\to G\Spc\cX_{\tr}\ .
\end{equation}

\begin{rem}\label{fijweofjojewfoefjeofewfewfewf}
 For a $G$-bornological coarse space $X$
 the representable presheaf $\yo_{\tr}(X)$ is a compact object.
 Moreover, the category $\PSh(G\BC^{}_{\tr})$ is compactly generated by representables.
 
  To make the construction of the category of motivic coarse spaces precise we 
 assume that there is a regular cardinal $\kappa$ which bounds the size of all coarse structures of spaces appearing in $G\BC_{\tr}$ (i.e., we consider a suitable subcategory  which is large enough to contain all spaces of interest),  and which also bounds the size of the index sets of big families involved in the descent condition. Then the locality conditions are generated by a small set of morphisms between $\kappa$-compact objects. It follows that $G\Spc\cX_{\tr}$ is $\kappa$-compactly generated and closed under $\kappa$-filtered colimits. For a bornological coarse space $X$ the object
 $\Yo_{\tr}(X)$   is $\kappa$-compact. See also \cite[Cor.~5.5.7.3]{htt}. 
\end{rem}

 By construction, $\Spc\cX_{\tr}$ is a presentable $\infty$-category. Let $\Prr^{L}$ be the large $\infty$-category of presentable $\infty$-categories and left-exact functors.
The inclusion $\Prr^{L}_{stab}\to \Prr^{L}$ of presentable stable $\infty$-categories in all presentable $\infty$-categories fits into an adjunction
\[\Stab\colon \Prr^{L}\leftrightarrows \Prr^{L}_{stab}:inclusion\ .\]

\begin{ddd}
We define the category $ G\Sp \cX_{\tr}$ of {\em coarse motivic spectra  with transfers}
  as the stabilization  $\Stab(G\Spc\cX_{\tr})$. \end{ddd}
By construction it fits into the  adjunction
\[\Sigma^{\infty}_{+}\colon G\Spc\cX_{\tr}\leftrightarrows  G\Sp \cX_{\tr}:\Omega^{\infty}\ .\]
We define the Yoneda functor  \begin{equation}\label{greg43t3rerer4t34t}
 \Yo^{s}_{\tr}:=\Sigma^{\infty}_{+}\circ \Yo_{\tr}\colon G\BC^{}_{\tr}\to  G\Sp \cX_{\tr}\ .
\end{equation}
   
 Recall that $G\BC_{\tr}^{}$ is semi-additive by Lemma \ref{hfweuihfweifefefeef} and that  $G\Sp \cX_{\tr}$ is additive since it is   a stable $\infty$-catefory.
 \begin{lem}\label{gfkiogergergregrge}
 The functor $\Yo^{s}_{\tr}$ is additive.
 \end{lem}
 \begin{proof}
 It suffices to show that $\Yo^{s}_{\tr}$ preserves zero objects  and coproducts.
 Both properties are consequences of excision. 

The zero object  in $G\BC_{\tr}^{}$ is given by the empty space $\emptyset$.
By excision we have $\Yo^{s}_{\tr}(\emptyset)\simeq 0$.

Let $X$ and $Y$ be two $G$-bornological coarse spaces. Their coproduct in $G\BC_{\tr}^{}$ is represented by the coproduct $X\sqcup Y$ in $G\BC$. We let $i\colon X\to X\sqcup Y$ and $j\colon Y\to X\sqcup Y$ denote the inclusions, and we let $(Y)$ denote the equivariant big family on $X\sqcup Y$ consisting just of $Y$. The pair $(X,(Y))$ is a complementary pair on $X\sqcup Y$. Since the subsets $X$ and $Y$ are disjoint, by excision the map
 \[\Yo^{s}_{\tr}(X)\oplus \Yo^{s}_{\tr}(Y)\xrightarrow{\Yo^s_{\tr}(i)+\Yo^s_{\tr}(j)}\Yo^{s}_{\tr}(X\sqcup Y)\]
 is an equivalence.
\end{proof}
 
 Let $X$ be a $G$-bornological coarse space,  $I$ be a $G$-set, and $i$ be a $G$-fixed element of $I$.
We set $I^{\prime}:=I\setminus \{i\}$.
Then $(\{i\}\times X, I^{\prime}\times X)$ is an invariant complementary pair on the space $I_{min,min}\otimes X$. By excision we have a decomposition 
\[  \Yo^{s}_{\tr}(I_{min,min}\otimes X)\simeq   \Yo^{s}_{\tr}(X)\oplus \Yo^{s}_{\tr}(I^{\prime}_{min,min}\otimes X)  \ .\] 
If we compose the   motivic transfer map $\Yo^{s}_{\tr}(\tr_{X,I})$ with the projections to the respective summands we get a decomposition
\[\Yo^{s}_{\tr}(\tr_{X,I})\simeq   a\oplus b\ ,\]
where
\[a\colon \Yo^{s}_{\tr}(X)\to \Yo^{s}_{\tr}(X)\ , \quad b\colon  \Yo^{s}_{\tr}(X)\to \Yo^{s}_{\tr}(I^{\prime}_{min,min}\otimes X)\ .\]
\begin{lem}\label{regopergergregrege}
We have equivalences
\[a\simeq  \id_{\Yo^{s}_{\tr}(X)}\ , \quad b\simeq  \Yo^{s}_{\tr}(\tr_{X,I^{\prime}})\ .\]
\end{lem}
\begin{proof}
Let $j_{i}\colon X\to  I_{min,min}\otimes X$ be the inclusion given by $x\mapsto (i,x)$.
In $G\BC_{\tr}$ we have the relation  \eqref{fwefewlkml243rt4}
\[j_{i}+\tr_{X,I^{\prime}}=\tr_{X,I} \ .\]
This implies by Lemma \ref{gfkiogergergregrge} that
\[\Yo^{s}_{\tr}(j_{i})+ \Yo^{s}_{\tr}(\tr_{X,I^{\prime}})\simeq  \Yo^{s}_{\tr}(\tr_{X,I})\ .\]
Using the projection \eqref{rbmmnmmrjkkj} we now have
\[a\simeq   \Yo^{s}_{\tr}(p_{i})\circ\Yo^{s}_{\tr}(\tr_{X,I})\simeq \Yo^{s}_{\tr}(p_{i}\circ  \tr_{X,I})\simeq  \id_{\Yo^{s}_{\tr}(X)}\]
and
\[ b\simeq  \Yo^{s}_{\tr}(\tr_{X,I})- \Yo^{s}_{\tr}(j_{i})\circ a\simeq   \Yo^{s}_{\tr}(\tr_{X,I})- \Yo^{s}_{\tr}(j_{i})\simeq   \Yo^{s}_{\tr}(\tr_{X,I^{\prime}})\ .\qedhere\] 
  \end{proof}
 
If $\cY:=(Y_{i})_{i\in I}$ is an equivariant  big family on a $G$-bornological coarse space $X$, {then} we set
\[\Yo^{s}_{\tr}(\cY):=\colim_{i\in I} \Yo^{s}_{\tr}(Y_{i})\ .\]

The following properties of the functor $\Yo^{s}_{\tr}$ are  shown by the same arguments as given for \cite[Cor.~4.12,   4.14 and 4.15]{equicoarse}.

\begin{lem}
\begin{enumerate}
\item If $X$ is a $G$-bornological coarse space and $A$   is a nice  invariant subset of $X$ \cite[Def.~3.3]{equicoarse}, then 
\[\Yo^{s}_{\tr}(A)\to \Yo^{s}_{\tr}(\{A\})\] is an equivalence.
\item If $(Y,Z)$ is an equivariant coarsely excisive pair on a $G$-bornological coarse space $X$, then we have a push-out: \[\xymatrix{\Yo^{s}_{\tr}(Z\cap Y)\ar[r]\ar[d]&\Yo^{s}_{\tr}(Z)\ar[d]\\\Yo^{s}_{\tr}(Y)\ar[r]&\Yo^{s}_{\tr}(X)}\]
\item If $I_{p}X$ is a coarse cylinder {\cite[Sec.~4.3]{buen}} over a $G$-bornological coarse space $X$, then the projection $I_{p}X\to X$ induces an equivalence
\[ \Yo^{s}_{\tr}(I_{p}X)\to \Yo^{s}_{\tr}(X)\ .\]
\end{enumerate}
\end{lem}

\begin{rem}
Let $A,B$ be objects in a stable $\infty$-category. Then we have an action
\[\nat\times \Map(A,B)\to \Map(A,B)\ ,  \quad (n,f)\mapsto nf\ .\] It sends a morphism
$f\colon A\to B$ to the composition
\[A \xrightarrow{\diag}\bigoplus_{i=1}^{n} A \xrightarrow{\bigoplus f} \bigoplus_{i=1}^{n} B \xrightarrow{+} B\ .\]
Here the diagonal map uses the interpretation of the sum as a product, while the last map is induced by the projections to the summands and interprets the sum as a coproduct.
\end{rem}

Let $X$ be a $G$-bornological coarse space and $I$ be a  set.  We consider $I$ as a $G$-set with the trivial $G$-action.
If $I$ is finite, then $I_{min,min}\to *$ is a morphism of $G$-bornological coarse spaces.
Hence we get a morphism $\rho\colon I_{min,min}\otimes X\to X$.
\begin{lem}
If $I$ is finite, then
\[\Yo^{s}_{\tr}(\rho)\circ \Yo^{s}_{\tr}(\tr_{X,I})\simeq |I|\cdot \id_{\Yo^{s}_{\tr}(X)}\ ,\]
\end{lem}
\begin{proof}
We have a commuting diagram
\[\xymatrix@C=4em{
 \Yo^{s}_{\tr}(X)\ar[r]^-{ \Yo^{s}_{\tr}(\tr_{X,I})}\ar@{=}[d] & \Yo^{s}_{\tr}(I_{min,min}\otimes X) \ar[d]^{\cong}\ar[r]^-{\Yo^{s}_{\tr}(\rho)} & \Yo^{s}_{\tr}(X)\ar@{=}[d] \\
  \Yo^{s}_{\tr}(X)\ar[r]^-{\diag_{\Yo^{s}_{\tr}(X)}} & \bigoplus\limits_{i\in I} \Yo^{s}_{\tr}(X)\ar[r]^-{+} & \Yo^{s}_{\tr}(X)} \\
\]
where the middle vertical isomorphism  is induced by excision. Lemma \ref{regopergergregrege}    ensures that the first square commutes. 
The second square commutes in view of Lemma  \ref{gfkiogergergregrge}.
\end{proof}

\subsection{Bounded and free unions}
 
Let $X $ be a $G$-bornological coarse space and $I$ be a $G$-set. 
\begin{ddd}\label{ddd_bounded_unions}
The \emph{bounded union} $\coprod_{i\in I}^{\bd}X$ in $G\BC$ is defined
as follows:
\begin{enumerate}
\item The underlying $G$-set of $\coprod_{i\in I}^{\bd}X$ is the product $I\times X$ with the diagonal $G$-action.
\item The bornology of $\coprod_{i\in I}^{\bd}X$ is given by the subsets $B$    satisfying the following two conditions:\begin{enumerate}
\item  The image of $B$ under the projection $I\times X\to I$ is finite.
\item The image of $B$ under the projection $I\times X\to X$ is bounded.
\end{enumerate}
\item The coarse structure of $\coprod_{i\in I}^{\bd}X$ is generated by the entourages
$ \diag_{I}\times U$ for all entourages $U$ of $X$.\qedhere
\end{enumerate}
\end{ddd}

\begin{rem}\label{eoifui24ri24r432r}
We can consider the $G$-set $I$ as the $G$-bornological coarse space   $I_{min,min}$ with the minimal bornology and the discrete coarse structure. Then we have an isomorphism of $G$-bornological coarse spaces
\[\coprod_{i\in I}^{\bd}X\cong I_{min,min}\otimes X\ ,\]
where $\otimes$ is the symmetric monoidal structure \cite[Sec.~4.3]{equicoarse} on $G\BC$.
\end{rem}

We say that a $G$-set $I$ has finite orbits if for every $i$ in $I$ the orbit $Gi$ is finite.

Assume that $I$ is a $G$-set with finite orbits. Let $X$ be a $G$-bornological coarse space.
\begin{ddd}\label{gierjgoergjt9t3t}
We define the \emph{free union}
$\coprod_{i\in I}^{\free}X$ in $G\BC$ as follows:
\begin{enumerate}\item\label{rgoejkrgo34t034t34t34t} The underlying $G$-bornological space of $\coprod_{i\in I}^{\free}X$ coincides with the one of
$\coprod_{i\in I}^{\bd}X$.
  \item\label{ergijio34t34t34t}
The coarse structure of $\coprod^{\free}_{i\in I} X$ is generated by the entourages
$\bigsqcup_{i\in I} U_{i}$ for all families $(U_{i})_{i\in I}$ of coarse entourages of $X$. \qedhere
\end{enumerate}
\end{ddd}

\begin{rem}
The restriction on the $G$-action on $I$ is necessary in order to ensure that the coarse structure described in Definition~\ref{gierjgoergjt9t3t}.\ref{ergijio34t34t34t} is a $G$-coarse structure.

If $I$ is more general, we could modify Point~\ref{ergijio34t34t34t} of Definition~\ref{gierjgoergjt9t3t} and instead
take the induced $G$-coarse structure. But then we may lose the compatibility with the bornology described in Point~\ref{rgoejkrgo34t034t34t34t} of Definition~\ref{gierjgoergjt9t3t}.
\end{rem}

\begin{rem}\label{gieugreg43t}
If $I$ is a $G$-set with finite orbits and $X$ is a $G$-bornological coarse space, then
we have a canonical morphism
\[\coprod_{i\in I}^{\bd}X\to \coprod_{i\in I}^{\free} X\]
  induced by the identity of the underlying set. 

In particular, if we assume that $I$ has the trivial $G$-action and $X$ is a $G$-bornological coarse space, then we have morphisms
\[\coprod_{i\in I} X \to \coprod_{i\in I}^{\bd}X \to \coprod_{i\in I}^{\free}X \]
all induced by the identity map of the underlying set. 
\end{rem}

\subsection{Equivariant coarse homology theories with transfers}
\label{sec_homology_with_transfers}

\newcommand{\Homology}{\mathbf{CoarseHomology}}

We recall the definition of an equivariant coarse
homology theory \cite[Def.~3.10]{equicoarse}. Let $\bC$ be a {cocomplete} stable $\infty$-category and
 $E\colon G\BC\to \bC$ be a functor. 
\begin{ddd}\label{roigwregergerger222}
 $E$ is an \emph{equivariant $\bC$-valued coarse homology theory}  if it satifies:
 \begin{enumerate}
\item $E$ is excisive for equivariant complementary pairs.
\item $E$ is coarsely invariant.
\item $E$ vanishes on flasque $G$-bornological coarse spaces.
\item $E$ is {$u$-}continuous.\qedhere
\end{enumerate}
\end{ddd}
We refer to \cite{equicoarse} for details on the notions appearing in the above definition.

Recall the embedding $\iota\colon \Nerve(G\BC)\to G\BC_{\tr}$ given in 
\cref{thorh5t4554t}.

\begin{ddd}\label{roigwregergerger}
An \emph{equivariant $\bC$-valued coarse homology theory with transfers} is a functor
\[E\colon G\BC_{\tr}^{}\to \bC\] such that $E\circ \iota$ is an equivariant $\bC$-valued coarse homology theory.
\end{ddd}
 The conditions  listed in Definition \ref{roigwregergerger222} determine the  full sub-$\infty$-category  \[G\Homology^{\bC}_{\tr}\subseteq \Fun(G\BC^{}_{\tr},\bC)\ .\]
 of {$\bC$-valued} equivariant coarse homology theories with transfer.

By construction of $G\Sp\cX_{\tr}$ we have the following proposition:
\begin{prop}\label{rhiohff3f34f34f3f}
The pre-composition with $\Yo^{s}_{\tr}$ (see \eqref{greg43t3rerer4t34t}) induces an equivalence
\[\Fun^{{\colim}}(G\Sp\cX_{\tr}, \bC)\xrightarrow{\simeq} G\Homology^{\bC}_{\tr}\ .\]
of the $\infty$-category of equivariant $\bC$-valued coarse homology theories with the $\infty$-category $\Fun^{{\colim}}(G\Sp\cX_{\tr}, \bC)$ of colimit-preserving functors from $G\Sp\cX_{\tr}$ to $\bC$.
\end{prop}
The argument is completely analogous to the one for \cite[Cor.~4.6]{buen}.

Let  $E\colon G\BC_{\tr} \to \bC$ be  an equivariant coarse homology theory with transfers.
\begin{kor}\label{ewfiowef234rr34t} The functor $E\colon G\BC^{}_{\tr} \to \bC$ is additive.
\end{kor}

\begin{proof}
This follows from Lemma \ref{gfkiogergergregrge}.
\end{proof}

Pull-back along the inclusion $\iota\colon \Nerve(G\BC)\to G\BC^{}_{\tr}$ sends  equivariant coarse homology theories with transfers to equivariant coarse homology theories in the sense considered in \cite{equicoarse}. Applied to $\Yo^{s}_{\tr}\circ \iota$ we get a colimit-preserving functor
\[\iota^{Mot}\colon G\Sp\cX \to G\Sp\cX_{\tr}\]
such that
\begin{equation*}
\Yo^{s}_{\tr}\circ \iota\simeq \iota^{Mot}\circ \Yo^{s}\ .
\end{equation*}

\begin{rem}\label{eifowffwefwefewf} 
For every $G$-set $I$ with finite $G$-orbits  and every $G$-bornological coarse space $X$ we have a version of the transfer
\begin{equation}\label{brtb4t45g4g45g4}
 {\tr_{X,I}^{\free}} \colon  X \xrightarrow{ \tr_{X,I} }   \coprod_{i\in I}^{\bd} X \to \coprod_{i\in I}^{\free} X
\end{equation}for the free union in $G\BC^{}_{\tr}$.
Furthermore, for every $G$-fixed point $j$ in $I$ we have the generalized morphism \begin{equation}\label{FEWFIOJIO23ROI}
 p^{\free}_{j}\colon  \coprod_{i\in I}^{\free} X \to X  \end{equation}represented by the span
\[\xymatrix{&X\ar@{->>}[dl]\ar[dr]^{\id_{X}}& \\ \coprod_{i\in I}^{\free} X &&X}\]
whose left leg is the inclusion of the $j$'th component.

If $E$ is now an equivariant coarse homology theory with transfers, then we have   induced morphisms
\[{E(\tr_{X,I}^{\free})} \colon  E(X)\to E  \big(\coprod_{i\in I}^{\free} X\big)\ , \quad E(p^{\free}_{j})\colon E \big( \coprod_{i\in I}^{\free} X\big)    \to E(X) \ .\]
Applying excision for the equivariant  coarsely excisive decomposition $( X_{j}, \coprod^{\free}_{i\in I\setminus\{j\}} X)$ of $\coprod_{i\in I}^{\free} X$  we get the right vertical arrow in the   diagram 
\begin{equation}\label{jrirhviuwevwevewv}
\xymatrix{
E(X)\ar[rrrr]^-{E(\tr^{\free}_{X,I}) }\ar@{=}[d]&&&& E  (\coprod_{i\in I}^{\free} X)\ar[d]^-{\simeq}\ar[rd]^-{E(p^{\free}_{j})}&\\E(X)\ar[rrrr]^-{\id_{E(X)}\oplus E(\tr^{\free}_{X,I\setminus\{j\}})}&&&& E(X)\oplus  E(\coprod^{\free}_{i\in I\setminus\{j\}} X)\ar[r]_-{\pr_{1}}&E(X)
}
\end{equation}
which commutes in view of \cref{regopergergregrege}.
\end{rem}

\begin{ex}\label{iojweoifewfewfwef}
Let $E$ be an equivariant $\bC$-valued coarse homology theory with transfers and $Q$ be any $G$-bornological coarse space. Then in view of the \cref{hio3t45ththr} and by \cite[Sec.~4.3]{equicoarse} the twist of $E$ by $Q$, which is defined as the composition
\[E(-\otimes Q)\colon G\BC^{}_{\tr}\xrightarrow{-\otimes Q} G\BC^{}_{\tr}\xrightarrow{E}\bC \ , \] is again an equivariant $\bC$-valued coarse homology theory with transfers.
For fixed $Q$, we thus get a colimit-preserving functor
\[ E(-\otimes Q) \colon G\Sp\cX_{\tr}\to \bC\ .\]

Using the bifunctor \eqref{clkjcoecwecwecwcec} we see that this construction is also functorial in $Q$ and satisfies the axioms of an equivariant coarse homology theory in this variable.  In order to see the last assertion, note that a functor $G\BC\to \Fun^{\colim}(G\Sp\cX_{\tr},\bC)$ is a coarse homology theory if and only if its evaluation  at $\Yo^{s}_{\tr}(X)$ for each object $X$ of $G\BC_{\tr}$ is a coarse homology theory. The objects of $G\BC_{\tr}$ are the objects of $G\BC$, and we already know that  twisting with a $G$-bornological coarse space preserves equivariant coarse homology theories by  \cite[Sec. 4.3]{equicoarse}.
Consequently, we get a bifunctor
\[ E(-\otimes-) \colon  G\Sp\cX_{\tr}\otimes G\Sp\cX\to \bC \]
which preserves colimits in each argument.
\end{ex}

We will show that if an equivariant coarse homology theory $E$ has transfers, then it has weak transfers \cite[Def.~2.4]{trans}.

We consider a family $(X_{i})_{i\in I}$ of $G$-bornological coarse spaces and a $G$-fixed point $j$ in $I$, and we set $I^{\prime}_{j}:=I\setminus  {\{j\}}$.
Then the pair of invariant subsets  $(X_{j}, \coprod^{\free}_{i\in I^{\prime}_{j}}X_{i})$ of  $\coprod^{\free}_{i\in I}X_{i}$ is an invariant coarsely excisive decomposition (see \cite[Def.~4.13]{equicoarse}).
If $E$ is an equivariant coarse homology theory, then
$E$ satisfies excision for invariant coarsely excisive decompositions   \cite[Cor.~4.14]{equicoarse}. Therefore    we can define a projection
\begin{equation}\label{brtboijogr4g45g5}
p^{ex}_{j}\colon E \big( \coprod^{\free}_{i\in I}X_{i} \big) \simeq E(X_{j})\oplus E \big( \coprod^{\free}_{i\in I^{\prime}_{j}}X_{i} \big) \to E(X_{j})\ ,
\end{equation} 
where the superscript $ex$ is a reminder for the fact that the morphism uses excision for~$E$.

Let $I$ be a set with the trivial $G$-action and let $E\colon G\BC\to \bC$ be an equivariant coarse homology theory.
Then we  define a functor
\[E^{I}\colon G\BC\to \bC\ , \quad X\mapsto E \big( \coprod^{\free}_{i\in I}X \big)\ .\]
For every   $j$  in $I$ the projection \eqref{brtboijogr4g45g5}    provides a  natural transformation of functors
\[p^{ex}_{j}\colon E^{I}\to E \ .\]  

Let $E$ be  an equivariant coarse homology theory.
\begin{ddd}\label{rgioggergergergerg}
$E$ has {\em weak transfers} for $I$ if there exists a natural transformation
\[\tr_{I}\colon E\to E^{I}\] such that
\begin{equation}\label{bvtrb4g4tb}
p^{ex}_{j}\circ \tr_{I}\simeq \id_{E}
\end{equation} for every $j$ in $I$.
\end{ddd}
  
\begin{lem}\label{eoifuwoif23r2r2r}
If $E$ admits transfers (see \cref{efuewhiffwefweffeff}), then $E$ has weak transfers.
\end{lem}

\begin{proof}
For every set $I$  and $G$-bornological coarse space $X$  we have  morphisms   $\tr_{X,I}^{\free}$ and $p_{j}^{\free}$, see \eqref{brtb4t45g4g45g4} and \eqref{FEWFIOJIO23ROI}, which satisfy the relation 
\begin{equation}\label{ewfewf32}
p_{j}^{\free}\circ \tr_{X,I}^{\free}=\id_{X}\ .
\end{equation} 
  If $E$ admits transfers, then by the commutativity of the right triangle in \eqref{jrirhviuwevwevewv} (where $E$ is replaced by the extension $E_{\tr}$ which exists by assumption) we have the equivalence
  \begin{equation}\label{vwojofijef23f}
E_{\tr}(p_{j}^{\free})\simeq p^{ex}_{j}
\end{equation} for every $j$ in $I$.  Here and below we implicitly identify the values of $E_{\tr}$ and $E$ on objects.

The morphism $\tr_{X,I}$ is  natural in $X$. We can therefore form the natural transformation
\[\tr^{\free}_{-,I}\colon \id_{G\BC^{}_{\tr}}\to \coprod^{\free}_{i\in I} -  \colon G\BC^{}_{\tr}\to G\BC^{}_{\tr}\]
of endofunctors of $G\BC^{}_{\tr}$.
We now define the natural transformation 
\[\tr_{I}:=E_{\tr}(\tr^{\free}_{-,I})\colon E\to E^{I}\ .\]
The relation \eqref{bvtrb4g4tb} is implied by \eqref{ewfewf32} and \eqref{vwojofijef23f}.
\end{proof}

Theorem \ref{riojgerogjo23rwefwefwef}  in combination with Lemma \ref{eoifuwoif23r2r2r} has the following corollary.
\begin{kor}\label{rgierog0uj9034tg3g}
\mbox{}
\begin{enumerate} \item Equivariant coarse ordinary homology $H\cX^{G}$ has weak transfers.
\item  Equivariant coarse algebraic $K$-homology $K\bA\cX^{G}$ with coefficients in  an additive category $\bA$ with a strict $G$-action  has weak transfers.
\end{enumerate}
\end{kor}
For $K\bA\cX^{G}$ an alternative and independent  argument is given in  \cite[Ex.~2.5]{trans}.

Let $E\colon G\BC_{\tr}\to \bC$ be an  equivariant  $\bC$-valued  coarse homology theory   with transfers.
\begin{ddd}\label{sdj2398sd}
$E$ is called \emph{strongly additive} if for every family $(X_{i})_{i\in I}$ of $G$-bornological coarse spaces  the morphism
\begin{equation}
\label{iuztrewq}
E\big( \coprod_{i\in I}^{\free}X_{i} \big) \to \prod_{j\in I} E(X_{j})\ .
\end{equation}
induced by  the family $(E(p^{\free}_{j}))_{j\in I}$, see \eqref{FEWFIOJIO23ROI}, is an equivalence.
 \end{ddd}

 \begin{rem}\label{fwefewifjioj32or}
An equivariant coarse homology theory with transfers $E$ is strongly additive if and only if the underlying equivariant coarse homology theory $E \circ \iota$ is strongly additive in the sense of \cite[Def.~3.12]{equicoarse}. This follows from the commutativity of the right triangle in \eqref{jrirhviuwevwevewv} which compares the projection $E(p_{j}^{\free})$ with the projection $p_{j}$ defined by excision  (the down-right composition in the triangle)  used in the reference.
 \end{rem}
 
 \begin{ex}
 Examples of strongly additive coarse homology theories with transfers are coarse algebraic $K$-homology and coarse ordinary homology, see \cref{sec:algK,sec:ordH}.
\end{ex}

\section{Examples}\label{ergpoerpgergerg}

In this section we show that equivariant coarse algebraic $K$-homology and equivariant coarse ordinary homology extend to equivariant coarse homology theories with transfers.

\subsection{Functors out of \texorpdfstring{$G\BC^{}_{\tr}$}{GBornCoarse with transfers}}\label{ergiuergergerge}
 
In order to construct coarse homology theories with transfers (see \cref{roigwregergerger}) we must construct functors out of the $\infty$-category $G\BC^{}_{\tr}$. Since this category is  given in   \cref{olergregregegerg}  explicitly as some  simplicial set, there are essentially two options. 
The simpler option is to start with the canonical functor
\[G\BC^{}_{\tr}\to \Ho(G\BC^{}_{\tr})
\]
and then to construct ordinary functors out of $\Ho(G\BC_{\tr})$. This option works in the case of the construction of equivariant ordinary coarse homology with transfers in \cref{sec:ordH}. The more complicated option is to describe directly a map of simplicial sets with domain 
 $G\BC^{}_{\tr}$. In the case of the construction of equivariant coarse algebraic  $K$-homology with coefficients in a $G$-equivariant additive category  in \cref{fwljfewkfjewfewfewfwf} the target of this map is the nerve of the strict $(2,1)$-category $\Add$ of additive categories.
 
The main goal of the present section is to   prepare the construction of coarse algebraic   $K$-homology with transfers by describing the data necessary to define a  functor  from 
$G\BC^{}_{\tr}$ to the nerve of some strict $(2,1)$-category $\bC$.

Applying the usual nerve functor $\Nerve\colon \Cat\to \sSet$ to the morphism categories  we get a category $\Nerve(\bC)$ which is enriched 
 in Kan complexes. We can now further apply the homotopy coherent nerve functor $\mathcal{N}$.
 In this way we get an $\infty$-category which, following \cite[Def.~A.12]{Gepner:2015aa}, will be denoted by $\Nerve_{2}(\bC)$. In the following we describe sufficient data  (justified by  \cref{ergioergergrege} below)  for a functor \begin{equation}\label{fuief3g43g3g3g}
\bV_{\tr}\colon G\BC^{}_{\tr}\to \Nerve_{2}(\bC)\ .
\end{equation}
 Suppose we are given the following data:
 \begin{enumerate}
 \item a functor $\bV\colon G\BC\to u(\bC)$, where $u(\bC)$ is the $1$-category obtained from $\bC$ by forgetting the rest of the $2$-category structure;
 \item for every bounded covering $w\colon W\to Z$ (see \cref{ergierog43t34t34t34t}) a $1$-morphism
\[w^{*}\colon \bV(Z)\to \bV(W)\ ;\]
 \item for every two composable bounded coverings $w\colon W\to Z$ and $v\colon V\to W$ a $2$-isomorphism  \[a_{v,w}\colon (w\circ v)^{*}\Rightarrow v^{*}\circ w^{*}\ ;\]
 \item\label{wefiowegergergergergerg} for every admissible square 
\[\xymatrix{&W\ar[dr]^{f}\ar@{->>}[dl]_{w}&\\V\ar[dr]_{g}&&U\ar@{->>}[dl]^{u}\\&Z&}\]
  of $G$-bornological coarse spaces (see \cref{gergergeg}) a $2$-morphism  \[b_{g,u}\colon f_{*}\circ w^{*}\Rightarrow u^{*}\circ g_{*}\ ,\]
 where we write $f_*$ and $g_*$ for $\bV(f)$ and $\bV(g)$ respectively.
\end{enumerate}
We assume that this data satisfies the following conditions:
 \begin{enumerate}
 \item\label{fuwif23fwfwef} 
 If the bounded covering $w\colon W\to Z$ is an isomorphism of the underlying $G$-coarse spaces, then we require that $w^{*}=(w^{-1})_{*}$.
 This is possible since  the inverse of a bornological bijection is proper and hence    $w^{-1}\colon Z\to W$ is a morphism of $G$-bornological coarse spaces.   \item\label{efiewoffewfwe2} If two composable bounded coverings $w\colon W\to Z$ and $v\colon V\to W$ are isomorphisms of the underlying $G$-coarse spaces, then $\alpha_{v,w}$ is the identity of $(v^{-1})_{*}\circ (w^{-1})_{*}=((w\circ v)^{-1})_{*}$. Note that this is possible  to require by Condition \ref{fuwif23fwfwef}.
 \item For every three composable bounded coverings $w\colon W\to Z$, $v\colon V\to W$, and  $u\colon U\to V$ the square \begin{equation}\label{gkengkergergreg}
 \xymatrix{(w\circ v\circ u)^*\ar@{=>}[r]^{a_{vu,w}}\ar@{=>}[d]_{a_{u,wv}}&(v\circ u)^*\circ w^*\ar@{=>}[d]^{a_{u,v}\circ w^*}\\
 u^*\circ(w\circ v)^*\ar@{=>}[r]^{u^*\circ a_{v,w}}&u^*\circ v^*\circ w^*}
\end{equation}
commutes.
\item\label{efiewoffewfwe21} In the case of an admissible square with morphisms $w,f,g,u$, if $u$ (and therefore also $w$) is an isomorphism of the underlying $G$-coarse spaces, then we require that $b_{g,u}$ is the identity of $f_{*}\circ (w^{-1})_{*} =(u^{-1})_{*}\circ g_{*}$.
\item\label{efiewoffewfwe211} In the case of an admissible square with morphisms $w,f,g,u$, if $f$ and $g$ are identities and therefore $w=u$, then we require that $b_{g,u}$ is the identity of {$w^*=u^*$.}
\item For every diagram
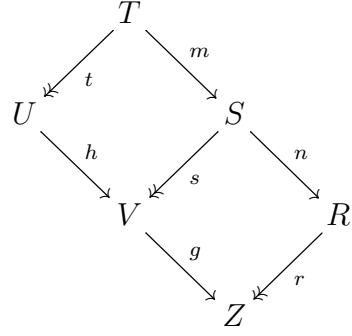
\[\xymatrix{&&&T\ar@{->>}[dl]^{t}\ar[dr]^{m}&&&\\&&U \ar[dr]^{h}&&S\ar[dr]^{n}\ar@{->>}[dl]^{s}&&\\& &&V \ar[dr]^{g}&&R\ar@{->>}[dl]^{r} &\\&&&&Z&&}\] 
consisting of two admissible squares we have the relation 
\begin{equation}\label{g34glkjkgl234g34}
b_{gh,r}=(b_{g,r}\circ h_*)(n_*\circ b_{h,s})\ .
\end{equation}
\item For every diagram  
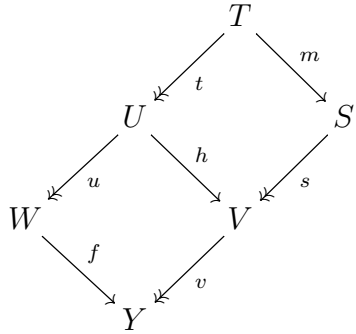
\[\xymatrix{&&&T\ar@{->>}[dl]^{t}\ar[dr]^{m}&&&\\&&U\ar@{->>}[dl]^{u}\ar[dr]^{h}&&S \ar@{->>}[dl]^{s}&&\\&W \ar[dr]^{f}&&V\ar@{->>}[dl]^{v} && &\\&&Y&&&&}\] 
consisting of  two admissible squares we have the relation \begin{equation}\label{fce32r32r23r5}
(a_{s,v}\circ f_*)b_{f,vs}= (s^*\circ b_{f,v})(b_{h,s}\circ u^*)(m_*\circ a_{t,u})\ .
\end{equation}
  \end{enumerate}

 \begin{lem}\label{ergioergergrege}
 The data as described above determines a functor
 \[ \bV_{\tr}\colon G\BC_{\tr}\to\Nerve_{2}(\bC)\ ,\]
  such that the diagram
 	\begin{equation}\label{btoijoi4g45g45g}\xymatrix{ G\BC_{\tr}\ar[r]^-{\bV_{\tr}}&\Nerve_2(\bC)\\
 		G\BC\ar[u]^\iota\ar[r]^-{\bV}&\Nerve(u(\bC))\ar[u]}
		\end{equation}
 	commutes.
 \end{lem}
\begin{proof}
It is known that the nerve  $\Nerve_{2}(\bC)$ for a strict $(2,1)$-category is 
 $3$-coskeletal \cite[Prop.~A.16]{Gepner:2015aa}. Therefore it suffices to provide the map $\bV_{\tr}$ on simplices of dimensions $0,1,2$ and~$3$. We need an explicit description of the $3$-skeleton of the nerve $\Nerve_{2}(\bC)$ (compare  \cite[Rem.~A.18]{Gepner:2015aa}).

For every $n$ in $\nat$ we consider the simplicially enriched category $\mathcal{C}[n]$ with objects $\{0,\dots,n\}$ and whose morphism space $\Map_{\cC[n]}(i,j)$ is the nerve of the poset of subsets of $[i,j]$ containing $i$ and $j$. Then by definition 
\[\Nerve_{2}(\bC)[n]= {\Hom_{\sCat}}(\mathcal{C}[n],\Nerve(\bC))\ .\] 
The following describes the $n$-simplices of $\Nerve_2(\bC)$ for $n\le 3$.
 \begin{enumerate}
 \item $\Nerve_{2}(\bC)[0] =\Ob(\bC)$.
 \item We have $\Nerve_{2}(\bC)[1]=\Fun(\mathcal{C}[1],\bC)$. Note that $\Map_{\mathcal{C}[1]}(0,1)=\{*\}$. Therefore a one-simplex  in $ \Nerve_{2}(\bC)$ is a morphism
 $X\to Y$ in $\bC$ and its faces are $X$ and $Y$.
 \item 
 A two-simplex in $\Nerve_{2}(\bC)$ is given by a   diagram
 \[\xymatrix{&Y\ar[dr]&\\X\ar[rr]\ar[ur]&\ar@{=>}[u]&Z}\]
\item 
The mapping spaces $\Hom_{\mathcal{C}[3]}(0,1)$, $\Hom_{\mathcal{C}[3]}(1,2)$ and $\Hom_{\mathcal{C}[3]}(2,3)$
are points. The mapping spaces $\Hom_{\mathcal{C}[3]}(0,2)$ and $\Hom_{\mathcal{C}[3]}(1,3)$ are isomorphic to
$\Delta^{1}$ and we call the one-simplexes $\alpha$ and $\beta$. The mapping space $\Hom_{\mathcal{C}[3]}(0,3)$ is the square 
\[\xymatrix{
&\{0,1,3\}\ar@{=>}[dr]^-{\ \beta \circ \{0,1\}} &\\
\{0,3\}\ar@{=>}[rr] \ar@{=>}[rd]^-{\gamma}  \ar@{=>}[ru]_-{\delta} &&\{0,1,2,3\}\\
&\{0,2,3\}\ar@{=>}[ur]_-{\ \{2,3\}\circ \alpha}&
}\]
Hence, in order to provide a $3$-simplex in $ \Nerve_{2}(\bC)$ we must provide the following data:
\begin{enumerate}
\item four objects $X_{0},X_{1},X_{2},X_{3}$;
\item six $1$-morphisms $f_{ij}\colon X_{i}\to X_{j}$ for $i<j$;
\item four $2$-morphisms  
\[\alpha\colon f_{02} \Rightarrow  f_{12}\circ f_{01},\ \beta\colon   f_{13} \Rightarrow f_{23}\circ f_{12},\ \gamma\colon f_{03}\Rightarrow   f_{23}\circ f_{02},\ \delta\colon f_{03}\Rightarrow    f_{13}\circ f_{01},\]
\end{enumerate}
satisfying the relation
\begin{equation}\label{fewfe23rrr32r32r32}
(\beta\circ f_{01})  \delta=(f_{23}\circ \alpha) \gamma\ .
\end{equation}

  \end{enumerate}
We can now construct $\bV_{\tr}$ using the data described above.
 \begin{enumerate}
 \item On zero simplices of $G\BC_{\tr}^{}$ we define
 $\bV_{\tr}(X):=\bV(X)$.
 \item On $1$-simplices of $G\BC_{\tr}^{}$ the functor $ \bV_{\tr}$   sends the span $(W,w,f)$
 to the morphism
\[f_{*}\circ w^{*}\colon \bV(X)\to \bV(Y)\ .\]
 Note that if $(W,w,f)$ is in the image of $\iota$, then {\eqref{btoijoi4g45g45g} commutes on the level of 1-simplices by Condition \ref{fuwif23fwfwef}.}
 \item  The functor $\bV_{\tr}$ sends a $2$-simplex
 \begin{equation}\label{fvelkvervrev2222}
 \xymatrix{&&  U\ar[dl]^{u}\ar[dr]^{h}&&\\
 	&  W\ar[dr]^{f}\ar[dl]^{w}&&  V\ar[dl]^{v}\ar[dr]^{g}\\
 	X&&  Y&&  Z}
 \end{equation} to the diagram
\[\xymatrix{&& \bV(Y)\ar[ddrr]^{ g_{*}\circ v^{*}}&&\\&& &&\\\bV(X)\ar[rrrr]_{ (g\circ h)_{*} \circ (w\circ  u)^{*}}\ar[uurr]^{f_{*}\circ w^{*}}&&\ar@{=>}[uu]&&\bV(Z)}\]
filled by the $2$-morphism
\[(g\circ h)_{*} \circ (w\circ u)^{*}\xRightarrow{(g\circ h)_*\circ a_{u,w}}(g\circ h)_{*}\circ u^{*}\circ w^{*}=g_{*}\circ (h_{*}\circ u^{*})\circ w^{*}\xRightarrow{g_*\circ b_{f,v}\circ w^*} g_{*}\circ v^{*}\circ f_{*}\circ w^{*}\ .\]
If  the $2$-simplex is in the image of $\iota$, then {\eqref{btoijoi4g45g45g} commutes on the level of 2-simplices by Conditions \ref{fuwif23fwfwef}, \ref{efiewoffewfwe2} and \ref{efiewoffewfwe21}.}

\item The functor $\bV_{\tr}$ sends a $3$-simplex
\[\xymatrix{&&&T\ar@{->>}[dl]^{t}\ar[dr]^{m}&&&\\&&U\ar@{->>}[dl]^{u}\ar[dr]^{h}&&S\ar[dr]^{n}\ar@{->>}[dl]^{s}&&\\&W\ar@{->>}[dl]^{w}\ar[dr]^{f}&&V\ar@{->>}[dl]^{v}\ar[dr]^{g}&&R\ar@{->>}[dl]^{r}\ar[dr]^{l}&\\X&&Y&&Z&&Q}\] 
 to the $3$-simplex of $\Nerve_{2}(\bC)$ given by the following data:
\begin{enumerate}
\item The objects of the $3$-simplex are $\bV(X)$, $\bV(Y)$, $\bV(Z)$, and $\bV(Q)$.
\item The $1$-morphisms are
\begin{enumerate}
\item $f_{01}:= f_{*}\circ w^{*}$
\item $f_{12}:= g_{*}\circ v^{*}$
\item $f_{23}:= l_{*}\circ r^{*}$
\item $f_{02}:= (g\circ h)_{*} \circ (w\circ  u)^{*}$
\item $f_{13}:= (l\circ  n)_{*} \circ (v\circ  s)^{*}$
\item $f_{03}:= (l\circ n\circ m)_{*}\circ (w\circ u\circ t)^{*}$
\end{enumerate}
\item The $2$-morphisms are
\begin{enumerate}
\item $\alpha:=   (g_*\circ b_{f,v}\circ w^*)((g\circ h)_*\circ a_{u,w})$
\item $\beta:=  (l_*\circ b_{g,r}\circ v^*)((l\circ n)_*\circ a_{s,v})$
\item $\gamma:= (l_*\circ b_{gh,r}\circ (w\circ u)^*)((l\circ n\circ m)_*\circ a_{t,wu})$
\item $\delta:=  ((l\circ n)_*\circ b_{f,vs}\circ w^*) ((l\circ n\circ m)_*\circ a_{ut,w})$
\end{enumerate}
\end{enumerate}
We must check the relation \eqref{fewfe23rrr32r32r32}:
\begin{align*}
(\beta & \circ f_{01})\delta =\\
& = \ (l_*b_{g,r}v^*f_*w^*)(l_*n_*\boldsymbol{a_{s,v}f_*}w^*)(l_*n_*\boldsymbol{b_{f,vs}}w^*)(l_*n_*m_*a_{ut,w})\\
& \stackrel{\mathclap{\eqref{fce32r32r23r5}}}{=} \ (l_*b_{g,r}v^*f_*w^*)(l_*n_*s^*b_{f,v}w^*)(l_*n_*b_{h,s}u^*w^*)(l_*n_*m_*\boldsymbol{a_{t,u}w^*})(l_*n_*m_*\boldsymbol{a_{ut,w}})\\
& \stackrel{\mathclap{\eqref{gkengkergergreg}}}{=} \ (l_*\boldsymbol{b_{g,r}v^*f_*}w^*)(l_*\boldsymbol{n_*s^*b_{f,v}}w^*)(l_*n_*b_{h,s}u^*w^*)(l_*n_*m_*t^*a_{u,w})(l_*n_*m_*a_{t,wu})\\
& \stackrel{\mathclap{!}}{=} \ (l_*r^*g_*b_{f,v}w^*)(l_*\boldsymbol{b_{g,r}h_*}u^*w^*)(l_*\boldsymbol{n_*b_{h,s}}u^*w^*)(l_*n_*m_*t^*a_{u,w})(l_*n_*m_*a_{t,wu})\\
& \stackrel{\mathclap{\eqref{g34glkjkgl234g34}}}{=} \ (l_*r^*g_*b_{f,v}w^*)(l_*\boldsymbol{b_{gh,r}u^*w^*})(l_*\boldsymbol{n_*m_*t^*a_{u,w}})(l_*n_*m_*a_{t,wu})\\
& \stackrel{\mathclap{!}}{=} \ (l_*r^*g_*b_{f,v}w^*)(l_*r^*g_*h_*a_{u,w})(l_*b_{gh,r}(wu)^*)(l_*n_*m_*a_{t,wu})\\
& = \ (f_{23}\circ \alpha)\gamma
\end{align*}
For better legibility we omitted the composition sign $\circ$ and marked boldface the part to which the respective relation is applied. The equations marked by $!$ hold in every $(2,1)$-category.

One again checks that the diagram~\eqref{btoijoi4g45g45g} commutes on the level of $3$-simplices because of Conditions~\ref{fuwif23fwfwef}, \ref{efiewoffewfwe2} and~\ref{efiewoffewfwe21}.
\end{enumerate}It is immediate from the definitions that our construction is compatible with the face maps.
To verify the compatibility with the degeneracy maps we use the Conditions~\ref{fuwif23fwfwef},
\ref{efiewoffewfwe2}, \ref{efiewoffewfwe21} and~\ref{efiewoffewfwe211} applied to identity maps  in the appropriate places.
\end{proof}

\subsection{Coarse algebraic \texorpdfstring{$K$}{K}-homology}\label{fwljfewkfjewfewfewfwf}
\label{sec:algK}

Let $\bA$ be  an additive category  with a strict $G$-action.
In this section we construct the extension of equivariant coarse algebraic $K$-homology $K\bA\cX^{G}\colon G\BC\to \Sp$ to an equivariant coarse homology theory with transfers $K\bA\cX^{G}_{\tr}$.
For the construction of the  functor $K\bA\cX^{G}$  (which will be recalled in detail below) and the verification of the axioms of an equivariant coarse  homology theory we refer to \cite[Sec.~8]{equicoarse}.

We first explain how the algebraic $K$-theory functor for additive categories can be extended to a functor defined on the $\infty$-category $\Nerve_{2}(\Add)$, see the  {beginning of \cref{ergiuergergerge}} for the notation $\Nerve_{2}$.
We start with a non-connective  algebraic $K$-theory functor
\[ K \colon \Add \to \Sp\ , \quad \bA\mapsto K(\bA) \]
for additive categories, see \cite{MR2079996}. More precisely, we consider $K$ as a functor between $\infty$-categories
\[ K \colon \Nerve(\Add)\to \Sp\ . \]
Let $W$ be the class of equivalences of additive categories in $\Add$. Since $K$ sends equivalences between additive categories to equivalences of spectra it has an essentially unique factorization over the localization $\Nerve(\Add)\to \Nerve(\Add)[W^{-1}]$. 
Because the natural inclusion $\Nerve(\Add)\to \Nerve_{2}(\Add)$ sends equivalences  {between} additive categories to equivalences in the $\infty$-category $\Nerve_{2}(\Add)$ it induces a functor $\Nerve(\Add)[W^{-1}]\to \Nerve_{2}(\Add)$. The latter is an equivalence  of $\infty$-categories \cite[Sec. 3.1]{addcats}.

Hence we get a commuting diagram in $\Cat_{\infty}$
\begin{equation}\label{g3oijoir5gteberv}
\xymatrix{\Nerve(\Add)\ar[rr]^{K}\ar[d]&&\Sp\\ \Nerve(\Add)[W^{-1}]\ar@{..>}[urr]\ar[d]^{\simeq}&&\\\Nerve_{2}(\Add)\ar[uurr]_{\bK}&&}
\end{equation}
It provides  an essentially unique extension of $K$ to a functor 
\begin{equation}\label{dqwdiuzhi2343r}
\bK\colon \Nerve_{2}(\Add)\to \Sp\ .
\end{equation}
Let $X$ be a $G$-bornological coarse space.
The spectrum $K\bA\cX^{G}(X)$ is the non-connective algebraic $K$-theory spectrum of the additive category $\bV_{\bA}^{G}(X)$ of equivariant $X$-controlled objects of $A$ and equivariant morphisms with controlled propagation \cite[Sec.~8.2]{equicoarse}. 
The functor $K\bA\cX^{G}$ is  defined as   the composition 
\[K\bA\cX^{G}:=K\circ \bV_{\bA}^{G}\colon G\BC \to \Add \to \Sp\ .\] For the verification that $K\bA\cX^{G}$ satisfies the axioms of a strongly additive equivariant coarse homology theory we refer to \cite[Thm.~8.9 and Prop.~8.19]{equicoarse}.

In order to construct the extension $K\bA\cX^{G}_{\tr}$ we use the method described in Section \ref{ergiuergergerge} to  construct an extension
\[ \bV_{\bA,\tr}^{G}\colon G\BC^{}_{\tr}\to \Nerve_{2}(\Add) \]
of the functor $\bV_{\bA}^{G}$, and compose it then with the functor $\bK$ in  \eqref{dqwdiuzhi2343r}.

We start with recalling the details of the definition of the $\Add$-valued functor $\bV_{\bA}^{G}$ from  \cite[Sec.~8.2]{equicoarse}.
Let $\bA$ be an additive category with a strict $G$-action and $X$ be a $G$-bornological coarse space. We consider the bornology  $\cB$ of $X$ as a poset with a $G$-action and hence as a category with a $G$-action.

If $A\colon \cB\to \bA$ is a functor and $g$ is an element of $G$, then  $g A\colon \cB\to \bA$ denotes  the functor which sends a bounded set $B$ in $\cB$ to the object $gA(g^{-1}(B))$ of $\bA$.
If $\rho\colon A\to  A^{\prime}$ is a natural transformation between two such functors, then we let  $g\rho \colon gA\to gA^{\prime}$  denote the canonically induced natural transformation.
\begin{ddd}\label{def:Xcontrolledobject}
	An \emph{equivariant $X$-controlled $\bA$-object} is a pair $(A,\rho)$ consisting of a functor
	$A \colon \cB \to \bA$
	and a family $\rho=(\rho(g))_{g\in G}$ of natural isomorphisms 
	$\rho(g)\colon A \to g A$
	 	satisfying the following conditions:
	\begin{enumerate}
		\item\label{def:Xcontrolledobject:it1} $A(\emptyset) \cong 0$.
		\item\label{def:Xcontrolledobject:it3} For all $B, B'$ in $\cB$, the commutative square
		\[\xymatrix{
			A(B \cap B')\ar[r]\ar[d] & A(B)\ar[d] \\
			A(B')\ar[r] & A(B \cup B')
		}\]
		is a pushout square.
		\item\label{def:Xcontrolledobject:it4} For all $B$ in $\cB$, there exists some finite subset $F$ of $B$ such that the inclusion $F\to B$ induces an isomorphism $A(F) \xrightarrow{\cong} A(B)$.
		\item \label{def:Xcontrolledobject:it5}For all pairs of elements  $g,g'$ of $G$ we have the relation
		$\rho(gg')= g\rho(g')\circ \rho(g)$.
		\qedhere
	\end{enumerate}
\end{ddd}

If $U$ is an invariant coarse entourage of $X$, i.e.,  an element of $\cC^{G}$, then we get a $G$-equivariant functor
\[U[-]\colon \cB\to \cB\] which sends a bounded
 subset $B$ of $X$ to its  $U$-thickening   \[U[B]:=\{x\in X\:|\: (\exists b\in B\:|\: (x,b)\in U)\}\ .\]
Note that $U[B]$ is again bounded by the compatibility of the coarse structure $\cC$ and the bornology $\cB$. For $g$ in $G$ we have the equality  $U[gB]=gU[B]$ by the $G$-invariance of $U$. Furthermore note that for $B^{\prime}$ in $\cB$ with $B\subseteq B^{\prime}$ we have $U[B]\subseteq U[B^{\prime}]$.

Let $(A,\rho), (A',\rho')$ be equivariant $X$-controlled $\bA$-objects and $U$ be  an invariant coarse entourage of $X$.
\begin{ddd}
	An \emph{equivariant $U$-controlled morphism} $\phi  \colon (A,\rho) \to (A',\rho')$ is a natural transformation
	\[ \phi \colon A(-) \to A'(U[-])\ ,\] 
	such that $\rho'(g)\circ \phi =(g \phi )\circ \rho(g)$ for all elements $g$ of $G$.
\end{ddd}
We let $\Mor_U((A,\rho),(A',\rho'))$ be the abelian group of equivariant $U$-controlled morphisms.

If $U^{\prime}$ is in $\cC^{G}$ and such that $U\subseteq U^{\prime}$, then for every $B$ in $\cB$ we have $U[B]\subseteq U^{\prime}[B]$. These inclusions induce a transformation between functors
$A^{\prime}(U[-])\to A^{\prime}(U^{\prime}[-])$ and therefore a map 
\[\Mor_U((A,\rho),(A',\rho'))\to \Mor_{U^{\prime}}((A,\rho),(A',\rho'))\]
by post-composition. Using these maps in the interpretation of the colimit
 we define the abelian group of equivariant controlled morphisms from $A$ to $A'$ by
\[ \Hom_{\bV^{G}_\bA(X)}((A,\rho),(A',\rho')) := \colim_{U \in \cC^{G}} \Mor_U((A,\rho),(A',\rho'))\ .\]

We now consider a pair  of morphisms in \[ \Hom_{\bV^{G}_\bA(X)}((A,\rho),(A',\rho'))\:\:\mbox{and}\:\:  \Hom_{\bV^{G}_\bA(X)}((A^{\prime},\rho^{\prime}),(A^{\prime\prime},\rho^{\prime\prime}))\ ,\] respectively, 
 which are represented by 
  \[\phi \colon A(-)\to A^{\prime}(U[-])\:\:\mbox{and}\:\: \phi ^{\prime}\colon A^{\prime}(-)\to A^{\prime\prime}(U^{\prime}[-])\ .\] 
We define the composition of the two morphisms  to be represented by the morphism
\[U[-]^{*}\phi ^{\prime}\circ \phi \colon  A\to A^{\prime\prime}((U^{\prime}\circ U)[-])\ , \]
where
$U[-]^{*}\phi^{\prime}\colon A^{\prime}(U[-])\to A^{\prime\prime}((U^{\prime}\circ U)[-])$ is defined in the canonical manner.

We denote now the resulting additive category of equivariant $X$-controlled $\bA$-objects and equivariant controlled morphisms by $\bV_\bA^G(X)$.

Let $f \colon (X,\cB,\cC) \to (X',\cB',\cC')$ be a morphism of $G$-bornological coarse spaces,
and let $(A,\rho)$ be an equivariant $X$-controlled $\bA$-object.
Since $f$ is proper, it induces a functor $f^{-1}\colon \cB^{\prime}\to \cB$, and we can define a functor $f_*A \colon \cB' \to \bA$ by
\[ f_*A := A\circ f^{-1} \ .\]
Furthermore, we define
\[f_*\rho(g) :=\rho(g)\circ f^{-1} \ .\]
Let $U$ be in $\cC^{G}$ and let $\phi \colon (A,\rho) \to (A',\rho')$ be an equivariant $U$-controlled morphism. Then  $V:= (f \times f)(U)$ belongs to $\cC^{\prime G}$ and  $U[f^{-1}(B^{\prime})] \subseteq f^{-1}(V[B^{\prime}])$ for all bounded subsets $B^{\prime}$ of $X^{\prime}$. Therefore, we obtain an induced $V$-controlled morphism
\[f_*\phi = \{  f_*A(B^{\prime}) \xrightarrow{\phi_{f^{-1}(B^{\prime})}} A(U[f^{-1}(B^{\prime})])\to  f_*A(V[B^{\prime}]) \}_{B^{\prime} \in \cB'}\ .\]
One checks that this construction defines an additive functor
\[ f_* \colon \bV^G_\bA(X) \to \bV^G_\bA(X')\ .\]
This completes the construction of the functor
\[ \bV^G_\bA \colon G\BC \to \Add\ .\]

We now start the construction of the functor $\bV_{\bA,\tr}^{G}$.

Let $w\colon W\to Z$ be a bounded covering. Given a controlled object $(A,\rho)$ in $\bV_\bA^G(Z)$ we define $w^*(A,\rho)=(w^*A,w^*\rho)$ as follows. Let $\cB_{W}$ denote the category    of bounded subsets of $W$ and let \begin{equation}\label{rgergkln3g4kl3tg34g34g}
 \cB'_W\subseteq \cB_W
\end{equation}
be the full subcategory   consisting of coarsely connected bounded subsets. 
Let
\[\hat w\colon \cB_W'\to \cB_Z\]
be the functor sending $B$ in $\cB_W'$ to $w(B)$ in $\cB_Z$. We define $w^*A:=\Lan(A\hat w)$ to be a left Kan extension of $A \hat w$ along the inclusion $i\colon \cB_W'\to \cB_W$ as indicated in the following diagram: 
\[\xymatrix@R=3em{\cB_{W}^{\prime}\ar[rr]^{A \hat w}\ar[d]_-{i}&\ar@{}[dl]^(0.1){}="a"^(0.8){}="b"\ar@{=>}"a";"b"_(0.35){\tau_{A,w}}&\bA\\\cB_{W}\ar[urr]_{\ \Lan(A\hat w)}&&}\]
The definition of $w^{*}A$ involves a choice. It is fixed uniquely up to unique isomorphism if we take into account 
the natural transformation \[\tau_{A,w} \colon A\hat w \to \Lan(A\hat w)i\]  which  is actually  a natural isomorphism since $i$ is fully faithful.
 If    $w$ is an isomorphism of the underlying coarse spaces, then $w^{-1}\colon Z\to B$ is proper, and we can choose the object $w^{*}A:=w^{-1}_{*}A$ and let
 $\tau_{A,w}$ be the identity. This ensures  Conditions
  \ref{fuwif23fwfwef} and 
 \ref{efiewoffewfwe2} formulated in \cref{ergiuergergerge}.
We suppress $\tau_{A,w}$ from notation unless we need to mention it explicitly.

For every $g$ in $ G$ we further define $w^*\rho(g)\colon w^*A\to gw^*A$ as the composition
\[\Lan(A\hat w)\xrightarrow{\Lan(\rho(g)\hat w)}\Lan(gA\hat w)\stackrel{\iota}{\to} g\Lan(A\hat w)\ ,\]
where the morphisms are uniquely determined by the universal property of left Kan extensions and the relations 
\[\tau_{gA,w}(\rho(g)\hat w)=(\Lan(\rho(g)\hat w)\circ i)  \tau_{A,w}\ , \quad (\iota\circ i)\tau_{gA,w}=g\tau_{A,w}\ .\]
The morphism $\iota$  is an isomorphism since $(g\Lan(A\hat w),g\tau_{A,w})$
 has the property of a left Kan extension of
$gA\hat w$ along $i$.

Note that $\bA$ admits finite sums but is in general not cocomplete.\footnote{Such a condition would actually lead, by an Eilenberg swindle, to a very uninteresting $K$-theory.} Therefore we must check that the Kan extensions actually exist and land in the desired functor category.
\begin{lem}
	The  Kan extensions involved in the construction of $A\hat w$ exists and $(w^*A,w^*\rho)$ is an object of $\bV_\bA^G(W)$.
\end{lem}
\begin{proof}
	By \cite[Cor.~X.3.4]{MR1712872}, the Kan extension exists if for every $B$ of $B_{W}$ the colimit
\[\colim_{(B' \subseteq B) \in \cB_W'/B} A(w(B'))\]
exists. Fix $B$ in $B_W$.
Since $w$ is a bounded covering (see \cref{ergierog43t34t34t34t}),  there exists a finite, coarsely disjoint partition $(B_{j})_{j\in J}$  of $B$ such that $w_{|[B_{j}]} \colon [B_j]\to [w(B_j)]$ is an isomorphism of  coarse spaces for every $j$ in $J$.
Since every element of $\cB_W'$ is coarsely connected, we have a decomposition of categories \[\cB_{W}^{\prime}/B\simeq \bigsqcup_{j\in J}\cB_{W}^{\prime} /B_{j}\ .\] 
For every $j$ in $J$ the inclusion of the discrete subcategory $\{( B_{j} \cap W_0 \subseteq B_{j}) \}_{W_0 \in \pi_0(W)}$ into the comma category $\cB_W'/B_{j}$ is cofinal. Hence we have to show that the sum
\begin{equation}
\label{eq:kansum}
\bigoplus_{j\in J}\bigoplus_{W_0 \in \pi_0(W)} A(w(W_0 \cap B_{j}))
\end{equation}
exists.
Since $w(B_{j})$ is a bounded subset of $Z$  by 
Property \ref{def:Xcontrolledobject}.\ref{def:Xcontrolledobject:it4} of $A$
it admits a finite subset $F_{j}$ such that
$A(F_{j})\xrightarrow{\cong}A(w( B_{j}))$.
We can choose a finite subset $P_{j}$ of $\pi_{0}(W)$ such that $F_{j}\cap w(B_{j}\cap W_{0})=\emptyset $ for all $W_{0}$ in $\pi_{0}(W)\setminus P_{j}$. 
	In \eqref{eq:kansum} we can therefore restrict the sum to the finite set $P_{j}$.
Since $\bA$ admits finite sums, this completes the proof of the existence of the Kan extension.

We use Properties \ref{def:Xcontrolledobject}.\ref{def:Xcontrolledobject:it3} and \ref{def:Xcontrolledobject}.\ref{def:Xcontrolledobject:it4} for $A$ in order to calculate the sum in \eqref{eq:kansum}, and hence the value of the Kan extension at $B$, explicitly. We obtain an isomorphism\begin{equation}\label{rijoijd32fd23f}
\Lan(A\hat w)(B) \cong  \bigoplus_{j\in J} A(w(B_{j})) \ .
\end{equation}
It is   now  straightforward  to check that $w^*A$ satisfies the conditions   \ref{def:Xcontrolledobject}.\ref{def:Xcontrolledobject:it1}, \ref{def:Xcontrolledobject}.\ref{def:Xcontrolledobject:it3} and  \ref{def:Xcontrolledobject}.\ref{def:Xcontrolledobject:it4} for a $W$-controlled $\bA$-object.  The relation \ref{def:Xcontrolledobject}.\ref{def:Xcontrolledobject:it5} can be checked by a similar reasoning as in the construction of $w^{*}\rho(g)$ using the universal property of left Kan extensions.
\end{proof}

The following observation is stated here for later use.
Let $W$ be a $G$-bornological coarse space and 
$i\colon \cB_W'\to \cB_W$ be the inclusion as in \eqref{rgergkln3g4kl3tg34g34g}.  Let $(A,\rho)$ be an object of $\bV_\bA^G(W)$. 
\begin{lem}
	\label{lem:lan}
	Then $A$ is canonically isomorphic to $\Lan(Ai)$, the left Kan extension of $A\circ i$ along $i$.
\end{lem}
\begin{proof} The argument is similar to the argument leading to  \eqref{rijoijd32fd23f} in the proof above.
\end{proof}

\begin{lem}
	\label{lem:restriction}
	If $w$ is an isomorphism on the underlying coarse spaces, then $\cB_W$ is a subset of $\cB_Z$ and in this case $w^*A$ is  {isomorphic} to the restriction of $A$ to $\cB_W$.
\end{lem}
\begin{proof}
	The first statement follows by \cref{ergierog43t34t34t34t} of a bounded covering. The second one from the pointwise formula 
	for Kan extensions and Properties   \ref{def:Xcontrolledobject}.\ref{def:Xcontrolledobject:it3} \&  \ref{def:Xcontrolledobject}.\ref{def:Xcontrolledobject:it4} of $A$.
\end{proof}

This finishes the construction of $w^{*}$ on objects.
We now define $w^*$ on morphism as follows. Let $U$ be an invariant entourage of $Z$ and let a natural transformation $\phi\colon A\to A'\circ U[-]$ be given. We set  $U_W:= (w\times w)^{-1}(U)\cap U(\pi_0(W))$, where $U(\pi_{0}(W))$ is defined as in \eqref{geh54lolj3p45g354}.  Then we consider the commutative diagram
\[\xymatrix{
	\cB_W\ar[rr]^-{U_W[-]} & & \cB_W\\
	\cB_W'\ar[rr]^-{U_W[-]}\ar[u]_{\subseteq}^i\ar[d]_{\hat w} & & \cB_W'\ar[u]_\subseteq^i\ar[d]_{\hat w}\\
	\cB_Z\ar[rr]^-{U[-]} & & \cB_Z
}\]
We consider the composition
\begin{align*}&&
 \Nat(A\hat w, A'U[-]\hat w)= \Nat(A\hat w, A'\hat wU_W[-])  \xrightarrow[(\tau_{A',w})_*]{\cong}  \Nat(A\hat w, \Lan(A'\hat w)iU_W[-])\\
&&=\Nat(A\hat w, \Lan(A'\hat w)U_W[-]i) 
\xrightarrow[(\tau_{A,w}^*)^{-1}]{\cong}  \Nat(\Lan(A\hat w), \Lan(A'\hat w)U_W[-])\ ,
\end{align*} and we define the morphism $ w^*\phi \colon w^*A \to (w^{*}A')U_{W}[-]$
to be the image of 
$\phi\circ \hat w $ under this map.
 {In other words, the morphism $w^*\phi$ is uniquely determined by the equation}
\begin{equation}\label{eq:transfer-on-morphisms} {(w^*\phi \circ i)\tau_{A,w} = \tau_{A',w} \circ (\phi\hat w)}\ .\end{equation}
Using this equation one checks easily that the construction of $w^{*}$ is compatible with the composition.

Given two bounded coverings $V\xrightarrow{v}W\xrightarrow{w}Z$, we now have to define a natural isomorphism \[a_{v,w}\colon  {(wv)^*A \to v^*w^*A}\ .\] Let $j\colon \cB_V'\to \cB_V$ be the inclusion analogous to the one in \eqref{rgergkln3g4kl3tg34g34g}. We observe that $\hat v$ has a canonical factorization $\hat v \colon \cB_V' \xrightarrow{\hat v'} \cB_W' \xrightarrow{i} \cB_W$ such that $\widehat{wv} = \hat w\hat v'$.
Since we have a natural isomorphism
\begin{equation}\label{qwdj1ke}
  {A\widehat{wv} = A\hat w\hat v' \xrightarrow[\tau_{A,w} \circ \hat v']{\cong} \Lan(A\hat w)i\hat v' = \Lan(A\hat w)\hat v \xrightarrow[\tau_{\Lan(A\hat w),v}]{\cong} \Lan(\Lan(A\hat w)\hat v)j} \ ,
\end{equation}
the functor
$\Lan(\Lan(A\hat w)\hat v)$ is a left Kan extension of $A\widehat{wv}$ along $j$ by \cref{lem:lan}. We define the natural isomorphism $a_{v,w}$ by
\[ {a_{v,w} \colon}(wv)^*A=\Lan(A\widehat{wv})\xrightarrow{\eqref{qwdj1ke}} \Lan(\Lan(A\hat w)\hat v)=v^*w^*A\ .\]
 In particular, $a_{v,w}$ is uniquely determined by the equality
\begin{equation}\label{eq:transfer-associator}
 {(a_{v,w} \circ j)\tau_{A,wv} = \tau_{\Lan(A\hat w),v}(\tau_{A,w} \circ \hat v')}\ .
\end{equation}
Since $(wv)^*\rho(g)$ and $v^*w^*\rho(g)$ are the natural equivalences of the left Kan extensions induced by $\rho(g)$, they agree under the above natural isomorphism.

Given three bounded coverings $U\xrightarrow{u}V\xrightarrow{v}W\xrightarrow{w}Z$, we have a commutative diagram
\[{\xymatrix{
 \cB_U'\ar[d]_{k}\ar[r]^{\hat u'}\ar[rd]^{\hat u}\ar@(ur,ul)[rr]^{\widehat{uv}'} & \cB_V'\ar[d]_{j}\ar[r]^{\hat v'}\ar[rd]^{\hat v} & \cB_W'\ar[d]_{i}\ar[rd]^{\hat w} & \\
 \cB_U & \cB_V & \cB_W & \cB_Z
}}\]
{We conclude that}
\begin{align*}
  (((a_{u,v} \circ w^*) a_{vu,w}) \circ k)\tau_{A,wvu}
  \ & = \ (a_{u,v} \circ w^* \circ k)\mathbf{(a_{vu,w} \circ k)\tau_{A,wvu}} \\
  & \stackrel{\mathclap{\eqref{eq:transfer-associator}}}{=} \ \mathbf{(a_{u,v} \circ w^* \circ k)\tau_{\Lan(A\hat w),vu}}(\tau_{A,w} \circ \widehat{vu}') \\
  & \stackrel{\mathclap{\eqref{eq:transfer-associator}}}{=} \ \tau_{\Lan(\Lan(A\hat w)\hat v),u} \mathbf{(\tau_{\Lan(A\hat w),v} \circ \hat u') (\tau_{A,w} \circ \widehat{vu}')} \\
  & \stackrel{\mathclap{\eqref{eq:transfer-associator}}}{=} \ \tau_{\Lan(\Lan(A\hat w)\hat v),u} (((a_{v,w} \circ j)\tau_{A,wv}) \circ \hat u') \\
  & = \ \mathbf{\tau_{\Lan(\Lan(A\hat w)\hat v),u}(a_{v,w} \circ \hat u)}(\tau_{A,wv} \circ \hat u') \\
  & \stackrel{\mathclap{\eqref{eq:transfer-on-morphisms}}}{=} \ (u^*a_{v,w} \circ k)\mathbf{\tau_{\Lan(A\widehat{wv}),u} (\tau_{A,wv} \circ \hat u')} \\
  & \stackrel{\mathclap{\eqref{eq:transfer-associator}}}{=} \ (u^*a_{v,w} \circ k)(a_{u,wv} \circ k)\tau_{A,wvu} \\
  & = \ (((u^*a_{v,w})a_{u,wv}) \circ k)\tau_{A,wvu}
\end{align*}
which proves that relation \eqref{gkengkergergreg} holds.
 
Given an admissible square \[\xymatrix{&W\ar[dr]^{f}\ar@{->>}[dl]^{w}&\\V\ar[dr]^{g}&&U\ar@{->>}[dl]^{u}\\&Z&}\]
we consider the diagram
\[\xymatrix{
	\cB_U\ar[r]^{f^{-1}}&\cB_W\\
	\cB_U'\ar[r]^{f^{-1}}\ar[u]^i_{\subseteq}\ar[d]_{\hat u}&\cB_W'\ar[u]_{\subseteq}\ar[d]_{\hat w}\\
	\cB_Z\ar[r]^{g^{-1}}&\cB_V
}\]
which is commutative since $w$ is bornological and admissible squares are pull-backs of the underlying coarse spaces.
 We define \[b'_{g,u} \colon \Lan((g_*A)\hat u) \to f_*\Lan(A\hat w)\] to be the natural isomorphism induced by the natural isomorphism
\[  { (g_*A)\hat u = Ag^{-1}\hat u = A\hat w f^{-1} \xrightarrow[\tau_{A,w} \circ f^{-1}]{\cong} \Lan(A\hat w)if^{-1} = f_*\Lan(A\hat w)i}\ . \]
 In particular, $b'_{g,u}$ is uniquely determined by the equation
\begin{equation}\label{eq:transfer-co-contra}
 { (b'_{g,u} \circ i)\tau_{g_*A,u} = \tau_{A,w} \circ f^{-1}}\ .
\end{equation}
We finally define $b_{g,u}$ as the inverse of $b'_{g,u}$.
As above this morphism is compatible with~$\rho$.

We check the relations \eqref{g34glkjkgl234g34} and \eqref{fce32r32r23r5}.
 Suppose that we have three admissible squares
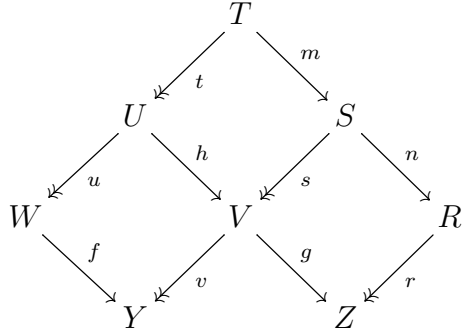
\[ \xymatrix{&&&T\ar@{->>}[dl]^{t}\ar[dr]^{m}&&&\\&&U\ar@{->>}[dl]^{u}\ar[dr]^{h}&&S\ar[dr]^{n}\ar@{->>}[dl]^{s}&&\\&W \ar[dr]^{f}&&V\ar@{->>}[dl]^{v}\ar[dr]^{g}&&R\ar@{->>}[dl]^{r} &\\&&Y&&Z&&} \] 
 We denote the inclusion $\cB_R' \to \cB_R$ (the analogue of \eqref{rgergkln3g4kl3tg34g34g}) by $i_R$. By repeated application of \eqref{eq:transfer-co-contra}, we then have
\begin{align*}
 (b'_{gh,r} \circ i_R)\tau_{g_*h_*A,t}
   &= \tau_{A,t} \circ (nm)^{-1} \\
 &=   (\tau_{A,t} \circ m^{-1})\circ n^{-1} \\
 &=   ((b'_{h,s} \circ i_R)\tau_{h_*A,s})\circ n^{-1}  \\
 &=  ((b'_{h,s} \circ i_R)\circ n^{-1})(\tau_{h_*A,s} \circ n^{-1} ) \\
 &= ((n_* \circ b'_{h,s}) \circ i_R)((b'_{g,r} \circ h_{*} \circ i_R)\tau_{g_*h_*A,t} \\
 &= (((n_* \circ b'_{h,s})(b'_{g,r} \circ h_{*})) \circ i_R)\tau_{g_*h_*A,t}\ .
\end{align*}
 This proves that relation \eqref{g34glkjkgl234g34} holds.

Finally, we compute
\begin{align*}
 ((b'_{h,s} \circ u^*)&(s^* \circ b_{f,v})\boldsymbol{(a_{s,v} \circ f_*)) \circ i_S)}\boldsymbol{\tau_{f_*A,vs}} \\
 &\stackrel{\eqref{eq:transfer-associator}}{=} ((b'_{h,s} \circ u^*) \circ i_S)\boldsymbol{((s^* \circ b'_{f,v}) \circ i_S)}\boldsymbol{\tau_{\Lan((f_*A)\hat v),s}}(\tau_{f_*A,v} \circ \hat s') \\
 &\stackrel{\eqref{eq:transfer-on-morphisms}}{=} ((b'_{h,s} \circ u^*) \circ i_S)\tau_{h_*\Lan(A\hat u),s}\boldsymbol{(b'_{f,v} \circ \hat s)}\boldsymbol{(\tau_{f_*A,v} \circ \hat s')} \\
 &\stackrel{\eqref{eq:transfer-co-contra}}{=} \boldsymbol{((b'_{h,s} \circ u^*) \circ i_S)}\boldsymbol{\tau_{h_*\Lan(A\hat u),s}}(\tau_{A,u} \circ h^{-1}\hat s') \\
 &\stackrel{\eqref{eq:transfer-co-contra}}{=} (\tau_{\Lan(A\hat u),t} \circ m^{-1})(\tau_{A,u} \circ \hat t'm^{-1}) \\
 &\stackrel{\hphantom{\eqref{eq:transfer-co-contra}}}{=} (\tau_{\Lan(A\hat u),t}(\tau_{A,u} \circ \hat t')) \circ m^{-1} \\
 &\stackrel{\eqref{eq:transfer-associator}}{=} ((a_{t,u} \circ i_T)\tau_{A,ut}) \circ m^{-1} \\
 &\stackrel{\hphantom{\eqref{eq:transfer-co-contra}}}{=} (m_* \circ a_{t,u} \circ i_S)(\tau_{A,ut} \circ m^{-1}) \\
 &\stackrel{\eqref{eq:transfer-co-contra}}{=} (m_* \circ a_{t,u} \circ i_S)(b'_{f,vs} \circ i_S)\tau_{f_*A,vs} \\
 &\stackrel{\hphantom{\eqref{eq:transfer-co-contra}}}{=} (((m_* \circ a_{t,u})b'_{f,vs}) \circ i_S)\tau_{f_*A,vs}
\end{align*}
 This implies immediately that relation \eqref{fce32r32r23r5} holds as well.

By \cref{ergioergergrege} the above data induces a functor
\[ \bV_{\bA,\tr}^{G}\colon G\BC_{\tr}^{}\to \Nerve_{2}(\Add)\ .\]

\begin{ddd}
	\label{def:algK}
	We  define the equivariant algebraic $K$-homology  with transfers
	\[K\cA\cX^{G}_{\tr}\colon G\BC_{\tr}\to \Sp\]
	as the composition
	\[K\cA\cX^{G}_{\tr}:=\bK\circ \bV_{\tr}\ .\qedhere\]
\end{ddd}

\begin{prop}
	The functor $K\cA\cX^{G}$ is equivalent to $K\cA\cX^{G}_{\tr}\circ \iota$.
\end{prop}
\begin{proof}
	This follows from the definition since the diagram
	\[\xymatrix{ G\BC_{\tr}\ar[r]^-{\bV_{\tr}}&\Nerve_2(\Add)\ar[r]^-\bK&\Sp\\
		\Nerve(G\BC)\ar[u]^\iota\ar[r]^-{\bV}&\Nerve(\Add)\ar[u]\ar[ur]_-K&	
}\]
commutes by \cref{ergioergergrege} and \eqref{g3oijoir5gteberv}.
\end{proof}

\subsection{Coarse ordinary homology}
\label{sec:ordH}

We first recall the construction of equivariant coarse ordinary homology \[H\cX^{G}\colon G\BC\to \Sp\] from \cite[Sec.~7]{equicoarse}.  
One starts with a functor
\[C\cX^{G}\colon G\BC\to \Ch\] which associates to a $G$-bornological coarse space $X$ the chain complex of $G$-invariant,  locally finite and controlled chains (the definitions will be recalled below).   We then use the Eilenberg--MacLane functor
\[\cE\cM\colon \Ch\to \Sp\] in order to define the equivariant coarse ordinary homology functor \[H\cX^{G}:=\cE\cM\circ C\cX^{G}\colon G\BC\to \Sp\ .\]

In order to define  equivariant coarse ordinary homology with transfers 
\[H\cX^{G}_{\tr}\colon G\BC^{}_{\tr}\to \Sp\] we will define a functor
\[C\cX^{G}_{\tr}\colon \Ho(G\BC_{\tr})\to  \Ch\] such that $C\cX^{G}_{\tr}\circ \iota=C\cX^{G}$.
It  then induces the desired extension $H\cX^{G}_{\tr}$ of $H\cX^{G}$ as the composition
 \begin{equation}\label{ververoij3o4gf34fg34fg}\hspace{-0.3cm}
G\BC_{\tr}^{}\to  \Ho(G\BC_{\tr}^{})
\xrightarrow{ C\cX^{G}_{\tr} }\Ch\xrightarrow{\cE\cM} \Sp\ ,
\end{equation} 
 where we omitted the nerve functor to consider ordinary categories as $\infty$-categories.

The construction of $H\cX^{G}_{\tr}$ turns out to be considerably less involved than in the construction of $K$-homology  $K\bA\cX^{G}_{\tr}$ given in Section \ref{fwljfewkfjewfewfewfwf} since  we can
stick to one-categorical considerations. We now explain the details.

Recall that the objects of $G\BC_{\tr}$ are $G$-bornological coarse spaces. Hence on objects we can define
 \[C\cX^{G}_{\tr}(X):=C\cX^{G}(X)\ .\]
To define $C\cX^{G}_{\tr}$  as a functor we must extend the functor
 $C\cX^{G}$ to generalized morphisms, see \cref{GROIERG}.

 We now recall the definition of  $C\cX^{G}(X)$. For an $n$ in $\nat$ the group 
 $C\cX_{n}^{G}(X)$  consists of   functions $c\colon X^{n+1}\to \Z$ which are $G$-invariant, and whose support is controlled and   locally finite.
 Here the group $G$ acts diagonally on the $(n+1)$-fold product $X^{n+1}$ of $X$ with itself. We say that a subset $S$ of $X^{n+1}$ is   controlled  if there exists an entourage $U$ of $X$ such that
 $(x_{0},\dots,x_{n})\in S $ implies that $(x_{i},x_{j})\in U$ for all $i,j$ in $\{0,\dots,n\}$. Finally, a subset $S$ of $X^{n+1}$ is 
 locally finite if for every bounded set $B$ of $X$ the set $\{s\in S\:|\: \mbox{$s$ meets $B$}\}$ is finite, where we say that
 $s=(x_{0},\dots,x_{n})$ meets $B$ if there exists $i$ in $\{0,\dots,n\}$ such that $x_{i}\in B$. 
 The differential \[\partial\colon C\cX_{n}^{G}(X)\to C\cX_{n-1}^{G}(X)\] is defined by $\partial:=\sum_{i=0}^{n} {(-1)^i}\partial_{i}$,
 where
 $\partial_{i}$ is the linear extension of the map $X^{n+1}\to X^{n}$ which omits the $i$'th entry.
 
 We consider now a generalized morphism
 $[W,w,f]$ from $X$ to $Y$, see the \cref{GROIERG}. 
 We consider the entourage $U(\pi_0(W))$ defined as in \eqref{geh54lolj3p45g354} for the partition $\pi_0(W)$ of $W$ into coarse components. We let $\chi^{n}_{\pi_0(W)}$ in $(\Z^{W^{n+1}})^{G}$ denote the $G$-invariant characteristic function of the set
 \[\{(w_{0},\dots,w_{n})\mid(\forall i,j\in \{0,\dots,n\}\mid (w_{i},w_{j})\in U(\pi_0(W))) \}\ ,\]
 i.e., the maximal $U(\pi_0(W))$-controlled subset of $W^{n+1}$.
 The map $w\colon W\to X$ induces a $G$-equivariant map $\hat w\colon W^{n+1}\to X^{n+1}$ which we can use to 
pull-back a $G$-invariant function $c$ on $X^{n+1}$ to a $G$-invariant function $\hat w^{*}c$ on $W^{n+1}$.
 Then we can define a map
 \[w^{*}\colon C\cX^{G}(X)\to (\Z^{W^{n+1}})^{G}\ , \quad w^{*}c:=\chi^{n}_{\pi_0(W)} \cdot \hat w^{*}c  \ .\]   
 We now show that $w^{*}c$ actually belongs to $C\cX_{n}^{G}(W)$.
  Let $B$ be a bounded subset of $W$. By \cref{ergierog43t34t34t34t} of a bounded covering there exists a finite partition $(B_{\alpha})_{\alpha\in I}$ of $B$ such that
  $w_{|[B_{\alpha}]}\colon [B_{\alpha}]\to [w(B_{\alpha})]$ is an isomorphism of coarse spaces. 
  Moreover, since $w$ is bornological and hence $w(B_{\alpha})$ is bounded  for every $\alpha$ in $I$, only
    finitely many points of the support of $w^{*}c$  meet $B_{\alpha}$.
  Since $I$ is finite only finitely many points of the support of $w^{*}c$ meet $B$.
   
  There exists an entourage $U$   of $X$ such that $c$ is $U$-controlled. Then it is straightforward to see that $w^{*}c$ is $w^{-1}U\cap U(\pi_0(W))$ controlled. Since $w^{-1}U\cap U(\pi_0(W))$ is an entourage of $W$ by the definition of a bounded coarse covering (\cref{wefiowefewfetrz}) we see that $w^{*}c$ is controlled.
  
 We have therefore defined a homomorphism
 \[w^{*}\colon C\cX^{G}_{n}(X)\to C\cX^{G}_{n}(W)\ .\]
 We now consider the compatibility of $w^{*}$ with the differential. 
 For notational simplicity we consider the case of $\partial_{n}$. We have 
 \[(\partial_{n}w^{*}c)(w_{0},\dots,w_{n-1})=\sum_{w_{n}\in W} \chi^{n}_{\pi_0(W)}(w_{0},w_{1},\dots,w_{n}) c(w(w_{0}),w(w_{1}),\dots,w(w_{n}))\ .\]  
 We fix $w_{0}$ in $W$ and 
let $W_{0}$ be the   coarse component of $w_{0}$. 
Because of the $\chi^{n}_{\pi_0(W)}$-factor a summand on the right-hand side is non-trivial only if the points
$w_{1},\dots,w_{n}$ all belong to   $W_{0}$.
Since $w$ is a bounded covering the restriction of $w$ to $W_{0}$ is a bijection
$w_{|W_{0}}\colon W_{0}\to w(W_{0})$ between  coarse components.
Since $c$ is controlled we see that    $c(w(w_{0}),\dots,w(w_{n-1}),x_{n})=0$  if  $x_{n}\not\in w(W_{0})$.
We therefore get the equality
\begin{eqnarray*}
(\partial_{n}w^{*}c)(w_{0},\dots,w_{n-1})&=&  \sum_{w_{n}\in W} \chi^{n}_{\pi_0(W)}(w_{0},w_{1},\dots,w_{n}) c(w(w_{0}),w(w_{1}),\dots,w(w_{n}))\\  &=&\sum_{x_{n}\in X} \chi^{n-1}_{\pi_0(W)}(w_{0},\dots,w_{n-1}) c(w(w_{0}),\dots,w(w_{n-1}),x_{n})\\&=& (w^{*}\partial_{n}c)(w_{0},\dots,w_{n})
 \ .\end{eqnarray*}
We thus have seen that $w^{*}$ induces a morphism of complexes.
We can now define
\[[W,w,f]_{*}\colon C\cX^{G}(X)\to C\cX^{G}(X)\ , \quad [W,w,f]_{*} :=f_{*}\circ w^{*}\ .\]
We must verify that $[W,w,f]_{*}$ is well-defined independently of the choice of the representative $(W,w,f)$ of the generalized morphism, and that this definition is compatible with the composition.  

Assume now that $\phi\colon W\to W^{\prime}$ induces an isomorphism between the spans $(W,w,f)$ and $(W^{\prime},w^{\prime},f^{\prime})$ from $X$ to $Y$. Then the commutative diagram  \eqref{revpokrvporekververv} induces a commutative diagram of chain complexes
\[\xymatrix{C\cX^{G}(X)\ar@{=}[d]\ar[r]^{w^{*}}&C\cX^{G}(W)\ar@/^0.3cm/[d]^{\phi_{*}}_(0.4){\cong}\ar[r]^{f_{*}}&C\cX^{G}(Y)\ar@{=}[d]\\
C\cX^{G}(X)\ar[r]^{w^{\prime,*}}&C\cX^{G}(W^{\prime})\ar@/^0.3cm/[u]_(0.4){\cong}^{\phi^{*}} \ar[r]^{f^{\prime}_{*}}&C\cX^{G}(Y)}
\]
where we use that $\phi_{*}$ is inverse to $\phi^{*}$. We conclude that $f_{*}w^{*}=f^{\prime}_{*}w^{\prime,*}$ and therefore that  $[W,w,f]$ is well-defined.

Let now $[V,v,g]$ be a generalized morphism from $Y$ to $Z$. Then we consider a representative $[U,(wu),(gh)]$ of the composition fitting into the diagram of $G$-coarse spaces
  \begin{equation}\label{fvelkvervrev1}
\xymatrix{&& U\ar[dl]_{u}\ar[dr]^{h}&&\\
& W\ar[dr]^{f}\ar[dl]_{w}&& V\ar[dl]_{v}\ar[dr]^{g}\\
X&&Y&&Z}
\end{equation}
where the square is admissible.
We get a diagram
\begin{equation}\label{fvelkvervrev11}
\xymatrix{&&C\cX^{G}(U) \ar[dr]^{ {h_*}}&&\\
&C\cX^{G}(W)\ar[ur]^{u^{*}}\ar[dr]^{ {f_*}} &&C\cX^{G}(V) \ar[dr]^{{g_*}}\\
C\cX^{G}(X)\ar[ur]^{w^{*}}&&C\cX^{G}(Y)\ar[ur]^{v^{*}}&&C\cX^{G}(Z)}
\end{equation}
of chain complexes.
The relation
\[[V,v,g]_{*}\circ [W,w,f]_{*}=[U,(wu),(gh)]_{*}\]
is now implied by the following two relations
\[u^{*}w^{*}=(wu)^{*}\ , \quad  {h_{*}}u^{*} = v^{*}f_{*}\]
which we will verify in the following to paragraphs.

Since $u$ is a morphism in $\wt{G\BC}$ we have the equality
\[u^{-1}U(\pi_0(W))\cap U(\pi_0(U))=U(\pi_0(U))\ .\] This implies  the relation $\chi^{n}_{\pi_0(U)}(\hat u^{*}\chi^{n}_{\pi_0(W)}) =\chi^{n}_{\pi_0(U)}$.
Therefore  for $c$ in $C\cX^{G}_{n}(X)$ we get the chain of equalities
\begin{align*}
\mathclap{
u^{*}w^{*}c=  \chi^{n}_{\pi_0(U)} (\hat u^{*} (\chi_{\pi_0(W)}^{n} \hat w^{*}c))=   \chi^{n}_{\pi_0(U)}(\hat u^{*} \chi_{\pi_0(W)}^{n})( \hat u^{*}\hat w^{*}c)=\chi^n_{\pi_0(U)}(\hat u^*\hat w^*c)= \chi^{n}_{\pi_0(U)} \widehat{(wu)}^{*}c=(wu)^{*} c\ .
}
\end{align*}
 
Let now $(v_{0},\dots,v_{n})$ be a point in $V^{n+1}$ and $c$ be in $C\cX_{n}^{G}(W)$.
Then we have the following chain of equalities:
\begin{eqnarray*}
(h_{*}u^{*}c)(v_{0},\dots,v_{n})&=&\sum_{(u_{0},\dots,u_{n})\in h^{-1}(v_{0},\dots,v_{n})} (u^{*}c)(u_{0},\dots,u_{n})\\
&=&\sum_{(u_{0},\dots,u_{n})\in h^{-1}(v_{0},\dots,v_{n})} \chi^n_{\pi_0(U)} (u_{0},\dots,u_{n})  c(u(u_{0}),\dots,u(u_{n}))\\
&\stackrel{!}{=}&\sum_{(u_{0},\dots,u_{n})\in h^{-1}(v_{0},\dots,v_{n})} \chi^n_{\pi_0(V)} (v_{0},\dots,v_{n})  c(u(u_{0}),\dots,u(u_{n}))\\
&\stackrel{!!}{=}&\sum_{(w_{0},\dots,w_{n})\in f^{-1}(v(v_{0}),\dots,v(v_{n}))} \chi^n_{\pi_0(V)} (v_{0},\dots,v_{n})  c(w_{0},\dots,w_{n})\\
&=&(v^{*}f_{*}c) (v_{0},\dots,v_{n})\ ,
\end{eqnarray*}
where for the equality marked ! we use the fact that (since $c$ is controlled and the square is admissible) if $(u_{0},\dots,u_{n})$  in $U^{n+1}$ is such that 
$c(u(u_{0}),\dots,u(u_{n}))\not=0$, then the conditions 
$\chi^{n}_{\pi_0(U)}(u_{0},\dots,u_{n})=1$ and $\chi^{n}_{\pi_0(V)}(h(u_{0}),\dots h(u_{n}))=1$ are equivalent.

For the equality marked !! we use that an admissible square is a pullback square and hence $u$ induces a bijection
\[\{(u_0,\ldots,u_n)\mid h(u_i)=v_i\}\to \{(w_0,\ldots, w_n)\mid f(w_i)=v(v_i)\}\ .\]
\begin{ddd}
We define
\[H\cX^{G}_{\tr} \colon G\BC_{\tr}\to \Sp\ .\]
as the composition \eqref{ververoij3o4gf34fg34fg}.
\end{ddd}

\begin{lem} $H\cX^{G}_{\tr}$ is an equivariant strongly additive coarse homology theory with transfers.
\end{lem}
\begin{proof}
By construction
$H\cX^{G}_{\tr}\circ \iota\simeq H\cX^{G}$ is a strongly additive equivariant coarse homology theory by \cite[Thm.~7.3 and Lem.~7.11]{equicoarse}.
\end{proof}

\section{Application: Mackey functors}
\label{sec_mackey}
In this final section we assume that $G$ is a finite group. In \cref{gegwfefwef} we show that any $G$-equivariant $\bC$-valued coarse homology theory with transfers gives rise to a $\bC$-valued Mackey functor. In the special case when $\bC$ is the category of spectra we obtain a spectral Mackey functor which is equivalent to the datum of a genuine $G$-equivariant spectrum, see   \cref{efuifeewfew}.  
Our main result  is   \cref{vwlkne2ofwevwevevw} which expresses the delooping along a representation sphere of  a Mackey functor obtained from an equivariant coarse homology theory with transfers    in terms of coarse geometry. 

Our main application of   transfers for equivariant coarse homology theories  is the descent argument leading to
injectivity results for assembly maps. We refer to  \cite{desc} for more details. 
In \cref{groiowwefwfewfwewf} we explain the main principle of the descent argument in the case of finite groups.
On the one hand we can avoid all the difficulties connected with infinite groups, but on the other hand, even for finite groups,  we obtain interesting consequences.

\subsection{Mackey functors from equivariant coarse homology theories with transfers}\label{gegwfefwef}
We let $G\Fin$ denote the category of finite $G$-sets and equivariant maps.
This category admits fibre products and we can form the bicategory 
$\Span(G\Fin)$ of spans in $G\Fin$. Its homotopy category is called the effective Burnside category of $G$. 
The $\infty$-categorical version of  the effective Burnside category   is the subcategory $A^{\eff}(G)$
of $\Fun(\Tw,G\Fin)$   (compare \cref{rioug9erger34t})  defined as follows.
\begin{ddd}\label{verpovevervev}
For every $n$ in $\nat$ the set
 of $n$-simplices of the $\infty$-category  $A^{\eff}(G)$ is the set of functors  $X$ in $\Fun(\Tw[n],G\Fin)$   such that  the squares 
\[\xymatrix{X_{i,j}\ar[r]\ar[d]&X_{i^{\prime},j}\ar[d]\\X_{i,j^{\prime}}\ar[r]&X_{i^{\prime},j^{\prime}}}\]
for all $0\le i\le i^{\prime}\le j^{\prime}\le j\le n$  are pull-backs.
\end{ddd}

Let $\bC$ be some $\infty$-category.

\begin{ddd}\label{geriogjergergrege}
We define the $\infty$-category $\Mack_\bC(G)$ of \emph{$\bC$-valued Mackey functors} to be the full subcategory of $\Fun(A^{\eff}(G)^{op},\bC)$ 
	of the coproduct preserving  (or equivalently, additive)  functors.
\end{ddd}

\begin{rem}\label{efuifeewfew}
The stable $\infty$-category $\Mack_\Sp(G)$ is called the  $\infty$-category of spectral Mackey functors, and it models the genuine stable homotopy category associated to the group~$G$ \cite{Guillou:2011aa,Barwick:2014aa}. 
Typical constructions in genuine equivariant stable homotopy theory are fixed points with respect to subgroups of $G$, deloopings along representation spheres, and 
geometric fixed points. In the present section we explain how these operations can be expressed in terms of coarse geometry provided the spectral Mackey functor is derived from an equivariant coarse homology theory with transfers, see \cref{rgiurg432t34t34t}.
\end{rem}

A $G$-set $S$ naturally gives rise to a $G$-bornological coarse space $S_{min,min}$ obtained by equipping $S$ with the minimal coarse and bornological structures. If $S\to T$ is a map between finite $G$-sets, then
$S_{min,min}\to T_{min,min}$ is controlled and proper.
We therefore have a functor  
  \begin{equation}\label{uiehweuihfiewfewfwef}
M\colon G\Fin\to G\BC\ , \quad  S\mapsto S_{min,min}\ .
\end{equation}
 
\begin{lem}\label{lem:prop}\mbox{}
	\begin{enumerate}
		\item \label{foiehfioefewfwefewf}The functor $M$ preserves finite coproducts.
		\item\label{fojwofewfewfewf} The functor $M$ intertwines the cartesian product on $G\Fin$
		 with the symmetric monoidal structure $\otimes$ on $G\BC$.		\item \label{fewuifhwiefewfewfw} The functor $M$ sends every morphism to a bounded covering.
		 \item  \label{efui2rf23r23r23r}The functor $M$ sends pull-back squares to admissible squares.
	\end{enumerate}
\end{lem}
\begin{proof} A finite coproduct  in $G\BC$ of $G$-sets with the minimal structures is the coproduct of the underlying  $G$-sets 
equipped with the minimal structures. This implies the 
 Assertion \ref{foiehfioefewfwefewf}. The finiteness assumption is necessary because an infinite coproduct in $G\BC$ of non-empty $G$-sets with the minimal structures would not have the minimal bornology anymore.
 
To see Assertion \ref{fojwofewfewfewf} note  that the
   $\otimes$-product of two finite $G$-sets with minimal structures in $G\BC$ is the product of the underlying sets with
 the minimal structures.   
 
 It has been observed in \cref{igjeroigegergregregreg} that a map between $G$-sets with minimal structures
 is a bounded covering. This implies 
 Assertion \ref{fewuifhwiefewfewfw}.
 
To see Assertion \ref{efui2rf23r23r23r} note that a cartesian square of finite $G$-sets becomes an admissible square (\cref{gergergeg}) if one equips the $G$-sets in the square with the minimal structures.
 \end{proof}

The following corollary is an immediate consequence of  \cref{lem:prop}.\ref{efui2rf23r23r23r}  and the fact (observed in the proof of  \cref{hfweuihfweifefefeef}) that the inclusion of $G\BC$  into $  G\BC_{tr}^{}$ preserves finite coproducts.
\begin{kor}\label{kor_M_coprod_preserving}
The functor $M$ naturally induces a coproduct preserving  functor \begin{equation}\label{dweoidowedwedwed}
M^{}\colon A^{\eff}(G)\to G\BC_{\tr}\ .
\end{equation}
\end{kor}

For a fixed $S$ in $G\Fin$ we have a functor \begin{equation}\label{4tgergergergregerg} P_{S}:=S\otimes (-)\colon A^{\eff}(G)\to A^{\eff}(G)\end{equation}  
given by  the cartesian product of objects, spans, etc., with $S$.  Recall that $G\Orb$ denotes the full subcategory of $G\Fin$
of transitive $G$-sets.
 
The effective Burnside category of $G$  has a canonical duality
\begin{equation}\label{vwhwiuevhuiwewec}
D\colon A^{\eff}(G)^{op}\to A^{\eff}(G)
\end{equation}
described in \cite[Sec.~2.18]{Barwick:2015ab}. 
It is the identity on objects. For the moment we only need to understand the functorial equivalence of mapping spaces
\begin{equation}\label{regergerg}
\Map_{A^{\eff}(G)}(S\otimes R,T)\simeq \Map_{A^{\eff}(G)}(R,D(S)\otimes T)
\end{equation}
for all $R,T$ in $A^{\eff}(G)$ and $S$ in $G\Orb$. Note that the right-hand side is defined  by considering $D(S)$ as an object of $G\Fin$.
The equivalence in \eqref{regergerg} is induced by the evaluation and coevaluation spans  
\[S\times S\xleftarrow{\diag(S)}S \to *\ , \quad *\leftarrow S\xrightarrow{\diag(S)} S\times S\ .\]
For details we refer to  \cite{Barwick:2015ab}.

Let now $\bC$ be stable and $E\colon G\BC_{\tr}\to \bC$ be an equivariant coarse homology theory with transfers.

\begin{ddd}\label{rgiurg432t34t34t}
We define the functor  \[EM:=E\circ M \circ D\colon A^{\eff}(G)^{op}\to  \bC \ \qedhere\]
\end{ddd}
	
The following corollary is an immediate consequence of  \cref{ewfiowef234rr34t}   {and \cref{kor_M_coprod_preserving}.}
\begin{kor}\label{kor_mackey}
The functor $EM$ preserves coproducts, i.e., it belongs to the subcategory $\Mack_\bC(G)$ of  $\Fun(A^{\eff}(G)^{op},\bC)$.
\end{kor}
The functor $EM$ is the Mackey functor associated to the $\bC$-valued  equivariant coarse homology theory with transfers $E$.
 
For a subgroup $H$ of $G$ we
define the functor
\begin{equation}\label{vrtvlk4nervev}
(-)^H\colon \Mack_\bC(G)\xrightarrow{\ev_{G/H}} \bC
\end{equation}
of evaluation at the $G$-set $G/H$.

\begin{rem}
The above notation is motivated by the notation for the $H$-fixed points  $F^{H}$ of a genuine $G$-equivariant spectrum $F$. Indeed, under the correspondence of spectral Mackey functors with genuine $G$-equivariant spectra (see \cref{efuifeewfew}) the operation \eqref{vrtvlk4nervev} corresponds to the operation of taking (categorical) $H$-fixed points. 
\end{rem}

 Let $E\colon G\BC_{\tr}\to \bC$ be an equivariant coarse homology theory with transfers. The following corollary is an immediate consequence of the definitions. 
 \begin{kor}\label{berbpojkpo34g3gereb} We have an equivalence
\[EM^{H}\simeq E((G/H)_{min,min})\ .\]
\end{kor}
	
\begin{rem}
The cartesian product of $G\Fin$ induces a symmetric monoidal structure $\otimes$ on $A^{\eff}(G)$.
The following constructions  could be written more naturally using this  symmetric monoidal structure.
Since in the present section we do not want to discuss this symmetric monoidal structure in detail we proceed in a more direct way.\end{rem}

 The functor $P_{S}$ defined in \eqref{4tgergergergregerg} preserves coproducts. Consequently, precomposition by  the functor $P_{S}$  preserves Mackey functors.
	Motivated by \cite[Cor.~4.5.1]{Barwick:2015ab}
we define the power functor  by the prescription
\begin{equation}\label{veroijeirojverv}
	G\Orb^{op}\times \Mack_\bC(G)\to \Mack_\bC(G)\ , \quad (S,F)\mapsto F^{S}:=P_{S}^{*}F\ .
	\end{equation}
	Here we implicitly use the functorial dependence   \[G\Orb\ni S\mapsto P_{S}\in \Fun(A^{\eff}(G),A^{\eff}(G))\ .\]
	Since Mackey functors are contravariant functors on $A^{\eff}(G)$ we eventually get the contravariant dependence on $S$ in \cref{veroijeirojverv}.

We now assume that $\bC$ is both presentable and stable.  By stability finite coproducts and products in $\bC$ coincide.
Using this fact we can describe the category $\Mack_{\bC}(G)$  as the full subcategory of sheaves in the $\infty$-category of $\bC$-valued presheaves $\Fun(A^{\eff}(G)^{op},\bC)$  on $A^{\eff}(G)$ with respect to the Grothendieck topology  given by finite disjoint decompositions into $G$-invariant subsets. It follows then that 
$\Mack_{\bC}(G)$ is presentable as well. 

Since limits in $\Fun(A^\eff(G)^{op},\bC)$ are defined objectwise, the functor of precomposition with $P_{S}$ preserves small limits.   
Since also the inclusion $\Mack_{\bC}(G)\to \Fun(A^\eff(G)^{op},\bC)$ detects and preserves limits, it follows that $(-)^{S}\colon \Mack_{\bC}(G)\to \Mack_{\bC}(G)$ preserves small limits. 
We therefore 
have an adjunction   
\[S\otimes (-)\colon  {\Mack_\bC(G)} \leftrightarrows {\Mack_\bC(G)}: (-)^{S}\]
which determines the tensor structure
\begin{equation}\label{veroijeirojverv111}
G\Orb\times \Mack_\bC(G)\to \Mack_\bC(G)\ , \quad (S,F)\mapsto S\otimes F\ .
\end{equation}  	 
Recall that by Elmendorf's theorem the $\infty$-category 
$\PSh( G\Orb)$
models the homotopy theory of $G$-spaces. 
We can left-Kan extend  the tensor structure~\eqref{veroijeirojverv111} (along the Yoneda embedding $G\Orb\to \PSh( G\Orb)$) to a functor 
\[\PSh( G\Orb) \times \Mack_\bC(G)\to \Mack_\bC(G)\ ,\quad  (X,F)\mapsto X\otimes F \]
preserving colimits in the first variable.
Similarly,  we can also right-Kan extend the power structure \eqref{veroijeirojverv}
  to a functor
\[\PSh( G\Orb)^{op} \times \Mack_\bC(G)\to \Mack_\bC(G)\ , \quad ( X,F)\mapsto F^{X}\]
preserving limits in the first variable.

	Let $E\colon G\BC^{}_{\tr}\to \bC$ be an equivariant coarse homology theory with transfers and recall \cref{rgiurg432t34t34t} of the $\bC$-valued Mackey functor $EM$ associated to $E$. Using \eqref{veroijeirojverv} we define the functor   
	\[ G\Orb^{op}\to \Mack_\bC(G) \ , \quad S\mapsto EM^{S}\ .\]
	For every equivariant coarse homology theory with transfers $E\colon G\BC^{}_{\tr}\to \bC$ and every $G$-bornological coarse space $V$
	we can form a new equivariant coarse homology theory with transfers $E_{V}\colon G\BC^{}_{\tr}\to \bC$ called the twist of $E$ by $V$,  see \cref{iojweoifewfewfwef}. 
	Recall the definition of  the functor $M^{} $ in \cref{dweoidowedwedwed}.  We can define  the functor
	\[G\Orb^{op}\to \Mack_\bC(G)\ , \quad S\mapsto E_{M^{}(D(S))}M\ .\]
To see the functorial dependence on $S$ note that by \cref{igjeroigegergregregreg}  for every map $S\to T$ in $G\Orb$ and $G$-bornological coarse space $V$ the induced map $S_{min,min}\otimes V\to T_{min,min}\otimes V$ is a bounded covering and therefore can serve as a left leg of a span  from $ T_{min,min}\otimes V$ to $S_{min,min}\otimes V$  whose right leg is the identity.
	
	\begin{prop}\label{rgiu9owewef}
		We have an equivalence  of functors $ EM^{(-)}\simeq E_{M^{}(D(-))}M$.
	\end{prop}
	\begin{proof}
		 The equivalence is implemented by the following chain of   equivalences which are natural in $S$: 
		\begin{eqnarray*} EM^{S}(-)&\simeq& (E\circ M^{}\circ D)^{S}(-)\\&\stackrel{\eqref{veroijeirojverv}}{\simeq}&  E\circ M^{}\circ  D\circ (S\otimes (-))\\
			&\stackrel{}{\simeq} &  E\circ M^{}\circ  (D(S) \otimes   D(-))\\&\stackrel{!}{\simeq}&E\circ (M^{} (D(S))\otimes   M^{}(D(-)))\\&\simeq& E_{M^{}(D(S))}M(-) \ ,\end{eqnarray*}
			where for the marked equivalence we use Lemma \ref{lem:prop}.\ref{fojwofewfewfewf}.
	\end{proof}

	The equivalence \eqref{regergerg} implies the natural equivalence of $\Spc$-valued Mackey functors
	\begin{equation}\label{verkvoervevwev}
y_{A^{\eff}(G)}(D(S)\otimes T)\simeq y_{A^{\eff}(G)}(  T)^{S}
\end{equation}
for every $T$ in $A^{\eff}(G)$ and $S$ in $G\Orb$. 
Using now that the functors $ D(S)\otimes -$ and $(-)^{S}$ on $\Mack_{\bC}(G)$ preserve colimits
	and that we can  write any  $\bC$-valued Mackey functor as a colimit of a diagram  of functors of the form
	$y_{A^{\eff}(G)}(T)\otimes C$ for~$T$ in $A^{\eff}(G)$ and $C$ in~$\bC$ (here $\otimes$ is the tensor structure of $\bC$ over $\Spc$)
	the equivalence \eqref{verkvoervevwev} extends to  the   Wirthmueller equivalence of functors
	\[D(-)\otimes F\simeq  F^{(-)}\colon G\Orb^{op}\to \Mack_\bC(G)\] for every $\bC$-valued Mackey functor $F$.
	If we combine the Wirthmueller equivalence
	with the equivalence shown in Proposition \ref{rgiu9owewef}, we get the following consequence.
	
	Let $\bC$ be a presentable stable $\infty$-category and
	 $E\colon G\BC^{}_{\tr}\to \bC$ be an equivariant coarse homology theory with transfers.
	\begin{kor}\label{weiufz932r32rr32r}
	 We have  a natural equivalence  of functors 
		\[(S\mapsto S\otimes EM\simeq   E_{S_{min,min}}M)\colon G\Orb\to  \Mack_\bC(G)\ .\]
	\end{kor}
	
	Let $X$ be a pointed $G$-space, i.e., an object of $ \Fun(G\Orb^{op},\Spc_{*})$,  and $F$ be a $\bC$-valued Mackey functor for a presentable and stable  $\infty$-category $\bC$.
	\begin{ddd}
		We define the $\bC$-valued Mackey functor  \[ X\wedge F:=\Cofib(*\otimes F\to X\otimes F)\ .\qedhere \]
		\end{ddd}
		 We have a canonical equivalence \begin{equation}\label{wfeflkmk2l3r23r32r}
X\wedge F\simeq   \Fib(X\otimes F\to *\otimes F)\ .
\end{equation}

		A $G$-topological space $A $ gives rise 
		to an object (also denoted by $A$)
		 of $\PSh(G\Orb)$
		 which sends the $G$-orbit $S$ to the space represented by the topological mapping space $\Map_{G\Top}(S_{disc},A )$. 
		 We use a similar notation convention for pointed $G$-topological spaces which yield
		objects of $ \Fun(G\Orb^{op},\Spc_{*})$.

		Let $V$ be a finite-dimensional Euclidean vector space  with an orthogonal representation of~$G$. We consider $V$ as an object of $G\BC_{\tr}$ with the structures induced by the metric. Let $S^{1}(V) $ be the 
		$G$-topological space given by the unit sphere in $V$. Furthermore, let $S(V) $ be the pointed $G$-topological space given by  the one-point compactification of $V$ by the point $\infty$. We will write $S(V)_{\infty}$ for the corresponding based space.

Let $\bC$ be a presentable stable $\infty$-category and $E\colon G\BC^{}_{\tr}\to \bC$ be an equivariant coarse homology theory with transfers,
	let $EM$ be  the $\bC$-valued Mackey functor  associated to $E$ (see \cref{rgiurg432t34t34t}), and let $V$ be a finite-dimensional Euclidean vector space  with an   orthogonal representation of $G$.	 	\begin{prop}
		\label{vwlkne2ofwevwevevw} 
		We have a canonical equivalence of $\bC$-valued Mackey functors
		\[S(V)_{\infty}\wedge EM\simeq E_{V}M\ .\]  	
	\end{prop}
	
	\begin{proof}
		  The cone $\cO(A)$  {(\cite[Sec.~9.4]{equicoarse})} of a compact metrizable $G$-space $A $ is a well-defined object of $G\BC$, and its   
		  underlying $G$-set is  the product of $G$-sets $[0,\infty)\times A  $. Its bornology is generated by the subsets
		  $[0,n]\times A$ for all $n$ in $\nat$. Finally, its coarse structure is the hybrid coarse structure associated to
		  the uniform structure for  some  choice of a metric $d$ on $A$ and the maximal coarse structure, and the exhaustion $([0,n]\times A)_{n\in \nat}$. The notation  $\cO(A)$ abbreviates the longer symbol  $\cO(A_{d,max,max})$ used e.g.~in \cite[Sec.~4]{desc}.  {The cone at infinity $\cO^{\infty}(A)$ is then defined as $\Yo^{s}(\cO(A),([0,n]\times A)_{n\in \nat})$, see \cite[Sec.~9.5]{equicoarse}.}
		  
		Let $S$ be  in $G\Orb$. We have an equivalence \cite[Prop. 9.35]{equicoarse} 
		\[\cO^\infty(S_{disc,max,max})\simeq\cO^{\infty}(S_{disc,min,max})\simeq \Sigma \Yo^{s}(S_{min,max})\ .\]
		Since $S$ is a finite set, the bornological coarse spaces $S_{min,max}$ and $S_{min,min}$ coincide.
		By \cref{weiufz932r32rr32r} we therefore have the equivalence 		\begin{equation}\label{ewfiohio234f23f23f32f32f}
		S\otimes EM\simeq E_{S_{min,min}}M\simeq \Sigma^{-1}E_{\cO^{\infty}(S_{disc,max,max})}M  
		\end{equation} 
		of $\bC$-valued Mackey functors which is natural in $S$. Note that the twist with an object of $G\Sp\cX$ is well-defined by \cref{iojweoifewfewfwef}. Let 
		\[\cO^{\infty}_{\hlg}\colon \PSh(G\Orb)\to G\Sp\cX\]
		be the left-Kan extension of the functor
		\[G\Orb\to G\Sp\cX\ , \quad S\mapsto \cO^{\infty}(S_{disc,max,max})\ .\]
		Note that this is consistent with the definition of $\cO^\infty_\hlg$ from \cite[Def.~8.16]{desc}, see also  \cite[Rem.~8.17]{desc} which explains the difference with the definitions given in \cite{equicoarse}. 

		By left Kan-extension along the Yoneda embedding $G\Orb\to  \PSh(G\Orb)$ both sides of  \eqref{ewfiohio234f23f23f32f32f} can be extended to
		colimit-preserving functors $ \PSh(G\Orb)\to \Mack_{\bC}(G)$. 
		Using also that $E$ preserves colimits, we get
	the equivalence of $\bC$-valued Mackey functors
		\begin{equation}\label{wefwheiofewfewfwf}
		X\otimes EM\simeq \Sigma^{-1} E_{\cO_{\hlg}^{\infty}(X)}M\ ,
		\end{equation}
		which is natural for $X$ in $ \PSh(G\Orb)$.

In the following we want to rewrite the cone sequence  {(\cite[Cor.~9.30]{equicoarse})}
  \begin{equation}\label{fk3jf34f3f}
 \Yo^{s}(S(V)_{max,max}) \to \Yo^{s}(\cO(S(V) ))\to \cO^{\infty}(S(V) ) \to \Sigma \Yo^{s}(S(V)_{max,max}) 
\end{equation}  
in $G\Sp\cX$ in simpler terms.

		 {Pulling back the $G$-bornological coarse structure of $V\oplus\R$ along the map   \begin{equation}\label{ejhewbcwewcwec}
[0,\infty)\times S^1(V\oplus \mathbb{R})\to V\oplus\mathbb{R}\ , \quad  (t,\xi)\mapsto t\xi
\end{equation} 
of $G$-sets} induces a $G$-bornological coarse structure on $ [0,\infty)\times S^1(V\oplus \mathbb{R})$ which we call the Euclidean cone structure. We let $\cO_{eu}(S^1(V\oplus \mathbb{R}))$ denote the $G$-set $ [0,\infty)\times S^1(V\oplus \mathbb{R})$ equipped with this structure.
		The identity of the underlying sets  induces a   morphism
		\[\cO_{eu}(S^1(V\oplus \mathbb{R}))\to  \cO(S^1(V\oplus \mathbb{R}))\]
		 of  $G$-bornological coarse spaces.
	By arguments which are analogous to the ones given in \cite[Sec.~8]{ass}
 one can show that this morphism induces an equivalence \begin{equation}\label{eqwkhkwefewfef}
\Yo^{s} 	(\cO_{eu}(S^1(V\oplus \mathbb{R})))	\stackrel{\simeq}{\to} \Yo^{s}( \cO(S^1(V\oplus \mathbb{R})))
\end{equation}
   in $G\Sp\cX$.	 
	  The map \eqref{ejhewbcwewcwec} has a right-inverse
	  \[V\oplus \R\to  [0,\infty)\times S^1(V\oplus \mathbb{R})\] which sends the origin
	 of $V\oplus \R$ to the  point $(0,(0,1))$.
	 One easily checks that these   maps   implement  an  equivalence of $G$-bornological coarse spaces between
$\cO_{eu}(S^{1}(V\oplus \R))$ and $V\oplus \R$. In particular we 
 get  the third equivalence in the chain
 \[\Yo^{s}(\cO(S(V)))\simeq \Yo^{s}(\cO(S^{1}(V\oplus \R)))\stackrel{\eqref{eqwkhkwefewfef}}{\simeq} \Yo^{s}(\cO_{eu}(S^{1}(V\oplus \R)))
 \simeq \Yo^{s}(V\oplus \R)\ .\]
 For the first equivalence we use the usual equivariant homeomorphism 
 $S(V) \cong S^{1}(V\oplus \R)$.

 Since $S(V)$ has a $G$-fixed point, the projection map 
 $ S(V)_{max,max}\to  *$ is an equivalence of 
 $G$-bornological coarse spaces. Therefore we have an equivalence
 \[\Yo^{s} (S(V)_{max,max})\stackrel{\simeq}{\to} \Yo^{s} (*)\ .\]
 We finally note that $S(V)$ is homotopy equivalent to a finite $G$-CW complex. 
 Since the functors $\cO^{\infty} $ and $\cO^{\infty}_{\hlg}$ behave as  $G\Sp\cX$-valued homology theories on the category of finite $G$-CW-complexes (see 
 \cite[Lem.~8.23]{desc}) and coincide on $G$-orbits, we have an equivalence
  \[\cO^{\infty} (S(V))\simeq \cO^{\infty}_{\hlg}(S(V) )\ .\] 
  The cone sequence \eqref{fk3jf34f3f}
   is therefore  equivalent to the fibre sequence
		  \begin{equation}\label{evwoihoigvrvw}
		\Yo^{s}(*) \to \Yo^{s}(V\oplus \mathbb{R})\to \cO_{\hlg}^{\infty}(S(V) ) \to \Sigma \Yo^{s}(*)
		\end{equation}
		in $G\Sp\cX$.
		Using the functoriality  of $Q\mapsto E_{Q}$ (\cref{iojweoifewfewfwef}) and  the obvious equivalence $E_{\Yo^{s}(*)}\simeq E$
		we therefore get the fibre sequence  of $\bC$-valued Mackey functors
	\[ EM\to E_{V\oplus\mathbb{R}}M\to E_{\cO_{\hlg}^{\infty}(S(V) )}M\to \Sigma EM\ . 
\]
		We now  apply  the equivalence \eqref{wefwheiofewfewfwf} to the third term and obtain the  sequence \begin{equation}\label{g4goijoggog}
EM\to E_{V\oplus\mathbb{R}}M\to \Sigma (S(V) \otimes EM)\to \Sigma EM\ .\end{equation} 
		We have a commuting diagram {
		\[\xymatrix{\Sigma (S(V) \otimes EM)\ar[d]\ar[r]_-\simeq^-{\eqref{wefwheiofewfewfwf}}&   E_{\cO_{\hlg}^{\infty}(S(V) )}M\ar[d]\ar[r]&\Sigma EM\ar@{=}[d]\\
			\Sigma(* \otimes EM)\ar[r]_-\simeq^-{\eqref{wefwheiofewfewfwf}}&E_{\cO_{\hlg}^{\infty}(* )}M\ar[r]_-\simeq&\Sigma EM }\]
		where the unnamed vertical maps are induced by the projection $S(V) \to * $. 
		It follows that the last map in \eqref{g4goijoggog} is induced by the projection $S(V) \to * $.}

		 In view of \eqref{wfeflkmk2l3r23r32r}	 the fibre sequence  \eqref{g4goijoggog} gives an equivalence
		\begin{equation}
		\label{eq:sdgglfsdljgf}
		E_{V\oplus \mathbb{R}}M\simeq \Sigma(S(V)_{\infty}\wedge EM)\ .
		\end{equation}
		  Applied to the special case where 
		  $V$ is the zero-dimensional representation we get the equivalence
		\begin{equation}\label{fwefioj23of223ff}
E_{\mathbb{R}}M\simeq \Sigma EM\ .
\end{equation}
		In the category $G\BC$ we have an equivalence
		$V\oplus \R\cong V\otimes \R$. 
		This implies the second equivalence in
		\[\Sigma (S(V)_{\infty}\wedge EM) \stackrel{\eqref{eq:sdgglfsdljgf}}{\simeq}E_{\R\oplus V}M\simeq  (E_{V})_{\R}M\simeq \Sigma E_{V}M\ ,\]
		where in the last equivalence  we apply \eqref{fwefioj23of223ff} to the twisted homology theory $E_{V}$ in place of $E$.
		Since $\Sigma$ is an equivalence 
		 the proposition follows.
	\end{proof}
	
	Finally we discuss the geometric fixed point functor for a subgroup $H$ of $G$.
	 	In genuine equivariant stable homotopy theory  the geometric fixed points  of  a genuine $G$-equivariant spectrum $F$ can be calculated by the formula
	\[\Phi^{H}(F)\simeq\colim_{V^{H}=\{0\}} [S(V)\wedge F]^{V}\ ,\]
	where the colimit runs over all finite-dimensional representations of $G$ with no non-trivial $H$-fixed vectors (\cite[Def.~II.2.10]{Nikolaus:2017aa}).
In the following we  {use} this formula as a definition:
\begin{ddd}\label{f2eoij2of23f23f}
We define the \emph{geometric fixed point functor}
\[\Phi^{H}\colon \Mack_{\bC}(G)\to \bC\]  by 
\[\Phi^{H}(F):=\colim_{V^{H}=\{0\}} [S(V)_{\infty}\wedge F]^{H}\ .\qedhere\]
\end{ddd}

	Let $E\colon G\BC^{Q}_{tr}\to \bC$ be an equivariant coarse homology theory with transfers,
	$EM$ be  the $\bC$-valued Mackey functor  associated to $E$, and $H$ be a subgroup of $G$.	 	
	\begin{prop}
		In $\bC$ we have an equivalence 
		\[\Phi^{H}(EM)\simeq \colim_{V^{H}=\{0\}} E((G/H)_{min,min}\otimes V)\ ,\] where the colimit runs over orthogonal representations $V$ of $G$ with no non-trivial  $H$-fixed vectors.
	\end{prop}
	\begin{proof} We   use \cref{berbpojkpo34g3gereb} and \cref{vwlkne2ofwevwevevw} in order to rewrite
		the formula from \cref{f2eoij2of23f23f} in the desired form.
	\end{proof}

\subsection{A descent principle}\label{groiowwefwfewfwewf}

In this section we explain how transfers can be applied to show injectivity of the assembly map in the case of finite groups. In view of \cref{hioehwvjoewvwevewvew} the main  result  itself  (see the \cref{wefiowefwfweewf1})  is not really new. Our main point is to give a selfcontained proof
using the descent principle, which avoids both the usage of the connection between spectral Mackey functors and genuine equivariant spectra and of results from genuine equivariant stable homotopy theory. Since we consider finite groups,
we can drop all arguments involving coarse geometry.   

Let $G$ be a finite group and let $\cF$ be a set of subgroups of $G$.   
\begin{ddd}
The set $\cF$ is called a family of subgroups if it is non-empty and closed under conjugation in $G$ and taking subgroups.
\end{ddd}

Let $\cF$  be a family of subgroups. Then we consider the full subcategory  $G_{\cF}\Orb$    of $G\Orb$ of transitive $G$-sets with stabilizers in $\cF$. For every  cocomplete  target category $\bC$ the inclusion $G_{\cF}\Orb\to G\Orb$ induces an adjunction 
\begin{equation}\label{eqcoiewcwc}
\Ind_{\cF}\colon \Fun( G_{\cF}\Orb,\bC)\leftrightarrows  \Fun(G\Orb,\bC):\Res_{\cF}\ .
\end{equation} 
We have a similar adjunction for contravariant functors.

We consider a functor $E\colon  G\Orb\to \bC$ with a cocomplete target $\bC$ and let $\cF$ be a family of subgroups of $G$. Note that $pt$ denotes the final object of $G\Orb$ given by the one-point $G$-set.
\begin{ddd}
The morphism
\begin{equation}\label{ewoij3wvwev}
\alpha_{\cF}\colon (\Ind_{\cF}\circ \Res_{\cF}(E))(pt)\to E(pt)
\end{equation}  given by the counit of the adjunction 
\eqref{eqcoiewcwc} is called the \emph{assembly map}.
\end{ddd}

Let $*_{\cF}$ denote the final object of $\PSh(G_{\cF}\Orb)$.
\begin{ddd}
The object $E_{\cF}G:=\Ind_{\cF}(*_{\cF})$ of $\PSh(G\Orb)$ is called the \emph{classifying space} of the family $*_{\cF}$.
\end{ddd}

One can check that
\begin{equation}\label{ewvfwjfwkefwf}
E_{\cF}G(S)\simeq \begin{cases} * & \text{if } S\in G_{\cF}\Orb\\\emptyset & \text{else} \end{cases}
\end{equation} 

The main result of this section is the next theorem.
Let $E\colon G\Orb\to \bC$ be a functor.
\begin{theorem}\label{ihwiuhvewvwevewv}
Assume: 
\begin{enumerate}
\item $\bC$ is stable, complete and cocomplete;
\item \label{wfiowfwefewfew} $E$ extends to a Mackey functor;
\item $E_{\cF}G$ is a compact object.
\end{enumerate}
Then the assembly map \eqref{ewoij3wvwev} is split injective.
\end{theorem}

Assumption \ref{wfiowfwefewfew} will be explained in \cref{efioewfwfwefewff} below.

\begin{rem}\label{regpoeregregrgr}
  Let $\bS$ be an $\infty$-category and $\bC$ a cocomplete $\infty$-category.
 Then pull-back along the Yoneda embedding $\yo \colon \bS\to \PSh(\bS)$ induces   an equivalence of $\infty$-categories 
 \[ \Fun^{\colim}(\PSh(\bS),\bC) \xrightarrow{\simeq}  \Fun(\bS,\bC) \ ,\]
 where the superscript indicates the full subcategory of colimit preserving functors.
  For a functor $F \colon \bS\to \bC$ we let  $ \wt F\colon \PSh(\bS)\to \bC$   denote the  
  essentially uniquely determined colimit-preserving functor corresponding to $F$ under this equivalence.
   Note that $\wt F$ (together with the identification   of its restriction with $F$) is a left Kan extension of $F$ along the Yoneda embedding.  
  
 Similarly, for a complete target $\bC$ we have an equivalence 
   \[ \Fun^{\lim}(\PSh(\bS)^{op},\bC) \xrightarrow{\simeq}  \Fun(\bS^{op},\bC) \ .\] Again,
 for a functor $F\colon \bS^{op}\to \bC$ we let $\widetilde F\colon \PSh(\bS)^{op}\to \bC$ denote the essentially uniquely determined limit-preserving functor corresponding to $F$ under the above equivalence. Note that~$\wt F$ (together with the identification  of its restriction  with $F$) is a right Kan extension of the functor~$F$ along the Yoneda embedding.
 If we consider $\wt F$ as a contravariant functor from $\PSh(\bS)$ to~$\bC$, then it sends colimits to limits.
 
 Let $\bS$ and $\bT$ be $\infty$-categories
 and assume now that we   have a bifunctor
 \[F\colon \bS^{op}\times \bT\to \bC\] with a complete and cocomplete target $\bC$.
 Then we can define  a functor
 \[\widetilde F\colon \PSh(\bS)^{op}\times \PSh(\bT)\to \bC\]
 by first right Kan extending $F$ in the first variable, and then left Kan extending the result in the second variable. We consider $F$ and $\wt F$ as contravariant functors in the first variable. The functor $\wt F$
 is essentially uniquely determined by  the property that it restricts to $F$ along the  product of Yoneda embeddings
 $\bS\times \bT\to \PSh(\bS)\times \PSh(\bT)$ and satisfies 
 \begin{equation} \label{ewfwefevvwvw1}\wt F(\colim_{I}X,\colim_{J} {\yo(Y)})\simeq \colim_{J} \lim_{I} {\wt F}(X,{\yo(Y)}) \end{equation}
 for all diagrams $X\colon  I\to {\PSh}(\bS) $ and
 {$Y\colon J\to \bT$.}

 Similary, switching the order of the left and right Kan extensions, we obtain a functor (contravariant in its first variable)
  \[ \widetilde F^{\prime}\colon \PSh(\bS)\times \PSh(\bT)\to \bC\ .\]
  Again, the functor $\wt F^{\prime}$
 is essentially uniquely determined by  the property that it restricts to $F$ along the  product of Yoneda embeddings
   and satisfies 
 \begin{equation} \label{ewfwefevvwvw}\wt F^{\prime}(\colim_{I} {\yo(X)},\colim_{J} Y)\simeq  \lim_{I} \colim_{J}{\wt F^{\prime}}({\yo(X)},Y) \end{equation}
 for all diagrams {$X\colon  I\to  \bS $} and
 $Y\colon J\to\PSh(\bT)$.

Finally note that the natural comparison morphism \[\colim_{J}\lim_{I}\to \lim_{I}\colim_{J}\] provides a comparison morphism
\[
c\colon \wt F\to \wt F^{\prime}\ .\qedhere
\]
\end{rem}

Let $E\colon G\Orb\to \bC$ be a functor. In the following statement $*$ denotes the final object of $\PSh(G\Orb)$. Recall the notation introduced in \cref{regpoeregregrgr}.
\begin{lem}\label{cweclkwevcwecw}
The assembly map \eqref{ewoij3wvwev}  is equivalent to the morphism
\[\wt E(E_{\cF}G)\to \wt E(*)\]
induced by the  morphism $E_{\cF}G\to *$.
\end{lem}

\begin{proof}
We have an equivalence $*\simeq \yo(pt)$. Moreover, 
$E_{\cF}G=\Ind_{\cF}(*_{\cF})$ can be expressed in terms of a left Kan extension, and the pointwise formula gives
\[E_{\cF}G\simeq \colim_{S\in G_{\cF}\Orb/pt} \yo (S)\ .\]
Since $\wt E$ preserves colimits  the morphism $\wt E(E_{\cF}G)\to \wt E(*)$ is equivalent to the morphism
\[\colim_{S\in G_{\cF}\Orb/pt}  E(S)\to E(pt)\]
induced by the morphisms $S\to pt$ in $G\Orb$. But this is now exactly the formula for the assembly map if one expresses 
$\Ind_{\cF}\circ \Res_{\cF}(E)$ as a left Kan extension of $ \Res_{\cF}(E)$ and again applies the pointwise formula.
\end{proof}

 Recall \cref{verpovevervev} of the $\infty$-category $A^{\eff}(G)$  modeling the effective Burnside category of $G$. 
We have a functor \begin{equation}\label{vervlkrmlvrvervvev} m\colon    G\Fin   \times G \Fin^{op} \to A^\eff(G)\end{equation}
which is characterized by the property that it
sends  a pair $( \psi\colon Q\to R\ ,\phi \colon T\to S)$ of morphisms in $G\Fin \times G\Fin  $ to the morphism
\[\xymatrix{&Q\otimes T\ar[dl]_{ \id\otimes   \phi }\ar[dr]^{\psi \otimes \id   }&\\
	Q\otimes S&&R\otimes T
}\]
in $A^\eff(G)$.  Note that we consider $m$ as a contravariant functor in the second argument.

 We now consider a functor  $ M \colon A^\eff(G)^{op}\to\bC$.
\begin{ddd}\label{jfi3rhfiuofewfewf-fin} We define the functor
	\[ F:=  M \circ m^{op} \colon G\Fin^{op} \times G\Fin \to \bC\ .\qedhere\] 
	 \end{ddd}
Note that the functor $F$ depends on $M$, but this is not reflected in the notation. 

The inclusion \begin{equation}\label{g45gkjnjgrgrgrg3}
r\colon G \Orb \to G\Fin
\end{equation} induces an adjunction
\begin{equation}\label{efkjbnkjffrefw-fin}
r_{!}\colon   \PSh(G\Orb) \leftrightarrows  \PSh(G\Fin):r^{*}\ .
\end{equation}

 We consider a functor $M\colon A^{\eff}(G)^{op}\to \bC$ for a complete and cocomplete target $\bC$, let $F$ be as in \cref{jfi3rhfiuofewfewf-fin}, and use the notation introduced in \cref{regpoeregregrgr}.
The counit \begin{equation}\label{g5k3n4kf34f34f34f1-fin2}
    r_{!}r^{*}\to \id \end{equation} of the adjunction \eqref{efkjbnkjffrefw-fin} induces   transformations 
\begin{equation}\label{rejvbnkjf4f43f-fin}
u\colon \wt F(-,-)\to \wt F(r_{!}r^{*}(-) ,-)  \ , \quad v\colon \wt F(-,r_{!}r^{*}(-))\to \wt F(-,-)
\end{equation}
Recall \cref{geriogjergergrege} of a Mackey functor.
 \begin{lem}
	\label{lem:u-equiv-fin}
	If $ M$ is a Mackey functor, then transformations $u$ and $v$ in \eqref{rejvbnkjf4f43f-fin} are equivalences.
\end{lem}
\begin{proof}  We show that  $v$ is an equivalence.	The proof for $u$ is similar.
 We first show that
 \begin{equation}\label{f2foi23jof2f32f32f32}
\wt F(\yo(T),r_{!}r^{*}\yo(S))\to \wt F(\yo(T),\yo(S))
\end{equation}
is an equivalence for all $S,T\in G\Fin$.

We observe that     \[  r_! r^*(\yo(S))\to \yo(S)\]    is   equivalent to  the morphism
	\[ \coprod_{R\in G\backslash S}\yo(r(R)) \to  \yo(S)\ ,\] induced by the family of inclusions $(r(R)\to S)_{R\in G\backslash S}$ (see  \cite[Lem.~5.10]{desc} for more details). 
  Using  
the fact that $\wt F$ preserves  colimits  in its second argument, we conclude that  the   morphism 
\eqref{f2foi23jof2f32f32f32} 
   is equivalent to the morphism \begin{equation}\label{f34f3t6t45t4t45t54t4}
\coprod_{R\in G\backslash S}\wt F(\yo(T),\yo(r(R)))\to \wt F(\yo(T),\yo(S))\ .
\end{equation}
	 In view of  the defining  relation between $\wt F$ and $F$	(see \cref{regpoeregregrgr}),
	 the  morphism \eqref{f34f3t6t45t4t45t54t4} is in turn equivalent to
	 the morphism
	\begin{equation}\label{g35g34g34g3}
\coprod_{R\in G\backslash S} F(T,r(R))  \to F(T,S)\ .
\end{equation}
	By \cref{jfi3rhfiuofewfewf-fin},  the morphism  \eqref{g35g34g34g3} is equivalent to the  morphism
	 \begin{equation}\label{vr4lklkmlkemcevev}
	 \coprod_{R\in G\backslash S} M(T\times r(R) ) \to M(T\times S)
	 \end{equation} 
		obtained from the transfers along the inclusions of the orbits of $S$. 
		We now use that these transfers also induce an equivalence 
	  \[   \coprod_{R\in G\backslash S}T\times r(R) \simeq  T\times S   \] in $A^{\eff}(G)^{op}$ and  that $M$ is a Mackey functor, i.e., coproduct preserving. This implies
	 that \eqref{vr4lklkmlkemcevev} and hence \eqref{f2foi23jof2f32f32f32}   is an equivalence.

Finally, using \eqref{ewfwefevvwvw1} and the fact that $r_{!}r^{*}$ preserves colimits, we can extend the equivalence
 \eqref{f2foi23jof2f32f32f32}  to all objects of $\PSh(G\Fin)^{op}\times  \PSh(G\Fin)$.	
	\end{proof}

We consider a functor $M\colon A^{\eff}(G)^{op}\to \bC$ for a complete and cocomplete target $\bC$ and
we let  $S$ be an object $S$ of $G\Fin$.  
Let $F$ be as in \cref{jfi3rhfiuofewfewf-fin} and recall  the notation introduced in \cref{regpoeregregrgr}.

\begin{lem}
	\label{lem:swap-fin}
	There is an equivalence  
	\[s\colon \wt F(-,\yo(S)) \simeq \wt F(-\times  \yo (S),*) \]	in
	 $ \Fun(\PSh(G\Fin)^{op},\bC)$. \end{lem}
\begin{proof}
	By definition of $m$ (see \eqref{vervlkrmlvrvervvev}), we have an equivalence \[m(-,S) \simeq m(-\times S,\pt)\]  of functors   $G\Fin \to A^\eff(G)$.
	Composing with $M$ and using the \cref{jfi3rhfiuofewfewf-fin}, we get an equivalence
	\begin{equation}\label{brtbrtvrvrv}
F(-,S)\simeq F(-\times S,pt)
\end{equation}
	of functors $G\Fin^{op}\to \bC$.
 We abbreviate
	\[F_{1}:= F(-,S)\ , \quad F_{2}:= F(-\times  S,pt)\ .\]
	By \eqref{brtbrtvrvrv} we have
	 	 	 an equivalence  
\begin{equation}\label{ferfoijiojf2f2f} \wt{F_1} \simeq  \wt{F_2}  \end{equation}
	 of  contravariant functors from  $\PSh(G\Fin)$ to $ \bC$ which send colimits to limits.
We now observe that, by the definitions given in \cref{regpoeregregrgr},
 \begin{equation}\label{vjkwnjewkvewvwvw}\wt{F_1}(-)\simeq \wt F(-,\yo(S))\ .  \end{equation}
Furthermore, since
$\yo $ preserves products,   for $T$ in $G\Fin $ we have an equivalence 
\begin{equation}\label{g3ggff2f2f}\wt{F_2}(\yo (T))\simeq F(T\times S,pt)\simeq \wt F(\yo (T\times S),*)\simeq \wt F(\yo (T)\times \yo (S),*)\ .\end{equation}
  We now use the general fact that for
  $X$ in $\PSh(G\Fin)$ the functor \[-\times X\colon \PSh(G\Fin)\to \PSh(G\Fin)\] of taking the product with $X$ preserves colimits.\footnote{It is a general property of $\infty$-topoi that colimits are universal, i.e., preserved by fibre products. We note that $\PSh(G\Fin)$ is an $\infty$-topos.}
	This implies that the equivalence \eqref{g3ggff2f2f} extends to an equivalence
\begin{equation}\label{wececwec}
\wt{F_2}(-)\simeq \wt F(-\times \yo(S),*)
\end{equation}	 
	of  contravariant functors  from $\PSh(G\Fin)$ to $\bC$ sending colimits to limits.
	Combining now \eqref{wececwec}, \eqref{vjkwnjewkvewvwvw} and \eqref{ferfoijiojf2f2f} we get the equivalence asserted in the lemma. 
\end{proof}

Let $M\colon A^\eff(G)^{op}\to\bC$ be a functor, $F$ be as in \cref{jfi3rhfiuofewfewf-fin}, and recall   the notation introduced in \cref{regpoeregregrgr} and \eqref{g45gkjnjgrgrgrg3}.
We consider an object  $A$  in $\PSh(G\Fin)$ and  a transitive $G$-set $R$ in $G\Orb$. Let
\begin{equation}\label{VWEVEWVWEV}
p_{R}\colon \wt F(*,\yo(r(R))) \to \wt F(A,\yo(r(R)))
\end{equation}
 be the map induced by $A\to *$ (note that $\wt F$ is contravariant in the first variable).
\begin{prop}\label{guoiuoi34ug34g34g3-fin}
	\label{thm:desc-orbits-fin}
	Assume:
	\begin{enumerate}\item \label{rfjgwiofewf}$R\in G_{\cF}\Orb$;
		\item $M$ is a Mackey functor;
		\item $r^*A$ in $\PSh(G\Orb)$ is equivalent to $E_\cF G$.
	\end{enumerate} Then \eqref{VWEVEWVWEV} is an equivalence.
\end{prop}
\begin{proof}
	 {Recall the equivalences $u$ and $s$ from \cref{lem:u-equiv-fin} and \cref{lem:swap-fin}} and consider the following commutative diagram:
	\[\xymatrix{
		\wt F(*,  \yo(r(R)) )\ar[d]_{s}^\simeq\ar[r]^-{p_{R}}&\wt F(A,\yo(r(R))) \ar[d]_{s}^\simeq\\
		\wt F (\yo ( r(R)),*)\ar[r]\ar[d]_{u}^\simeq&\wt F(A\times \yo (  r(R)),*)\ar[d]_{u}^\simeq\\
		\wt F(  r_!   r^*(\yo (  r(R))),*)\ar[r]\ar[d]_{!}^{\simeq}&\wt F (  r_!  r^*(A\times  \yo (  r(R))),*)\ar[d]_{!}^\simeq\\
		\wt F(  r_!(\yo (R)))\ar[r]\ar@{=}[d] &\wt F(    r_!(  r^*A\times \yo (R)))\ar[d]^{\simeq}\\
		\wt F(  r_!(\yo (R)))\ar[r]^-{!!}_-{\simeq}& \wt F(  r_!(E_{\cF}G\times \yo(R)))	}\]   
	For the equivalences marked by $!$ we use the canonical equivalence
	$  r^{*}\yo (  r(R))\simeq \yo (R)$ and that $  r^{*}$   preserves limits.
	
	Let $S$ be in $G\Orb$. By Assumption \ref{rfjgwiofewf}, the relation $\yo(R)(S)\not=\emptyset$ implies that $S \in G_{\cF}\Orb$. Hence by  \eqref{ewvfwjfwkefwf}
	\[E_{\cF}G\times \yo(R)\simeq \yo(R)\]
	and the map marked by $!!$ is an equivalence as claimed.
	  \end{proof}

In the situation of \cref{thm:desc-orbits-fin}, we can consider the map \begin{equation}\label{eberberblkberb} p_{A}\colon \wt F(*,A)\to \wt F(A,A)\end{equation}
 induced by $A\to *$. 
\begin{kor}
	\label{cor:assembly}
	Assume:
	\begin{enumerate}
		\item $M$ is a Mackey functor;
		\item $r^*A$ in $\PSh(G\Orb)$ is equivalent to $E_\cF G$.
	\end{enumerate} Then \eqref{eberberblkberb} is an equivalence.
\end{kor}
\begin{proof}
	Since $r^*A$ is equivalent to $E_\cF G$, $A$ is a colimit of objects of the form $\yo(S)$ with $S$ in $G_\cF\Fin$. Since $\wt F$ preserves colimits in its second argument, it suffices to show that
	\begin{equation}\label{g33jkgn35g3iogioj3g}
p_{\yo(S)}\colon \wt F(*,\yo(S))\to \wt F(A,\yo(S))
\end{equation}
	is an equivalence for all $S$ in $G_{\cF}\Fin$. By 
	\cref{lem:u-equiv-fin}, in \eqref{g33jkgn35g3iogioj3g} we can replace $\yo(S)$ by $r_{!}r^{*}\yo(S)$. We have
	\[r_{!}r^{*}\yo(S)\simeq r_{!}\Big(\coprod_{R\in G\backslash S} \yo(R)\Big)\simeq \coprod_{R\in G\backslash S}\yo(r(R))\ .\]
	Since $\tilde F$ preserves colimits in its second argument 
	 the map \eqref{g33jkgn35g3iogioj3g} is equivalent to the map
	\[\coprod_{R\in G\backslash S}\wt F(*,\yo(r(R)))\to \coprod_{R\in G\backslash S}\wt F(A,\yo(r(R)))\ .\]
Since $R\in G\backslash S$ implies $R\in G_{\cF}\Orb$, this map is an equivalence by \cref{thm:desc-orbits-fin}.
\end{proof}

We now consider the comparison morphism
\[c\colon \wt F\to \wt F'\]
introduced in \cref{regpoeregregrgr}.

Let $M\colon A^\eff(G)^{op}\to\bC$ be a functor for a complete and cocomplete target $\bC$, $F$ be as in \cref{jfi3rhfiuofewfewf-fin}, and recall  the notation introduced in \cref{regpoeregregrgr}.
Let $A$ and $B$ be in $\PSh(G\Fin)$.
\begin{lem}\label{lem:c-map}
Assume:
\begin{enumerate}
\item $\bC$ is stable;
\item  $A$ or $B$ is compact.
\end{enumerate}
Then the map
	\[c\colon \wt F(A,B)\to \wt F'(A,B)\]
	is an equivalence.
\end{lem}

\begin{proof}
Any compact presheaf is a retract of a finite colimit of representable presheaves.
Since a retract of an equivalence is an equivalence, it suffices to show the assertion under the assumption that $A$  or $B$ is a finite
colimit of representables.
To this end  we use the equivalences \eqref{ewfwefevvwvw1} and \eqref{ewfwefevvwvw} and the  fact that   in a stable $\infty$-category finite colimits commute with all limits and finite limits commute with all colimits. 
 \end{proof}

Let $M\colon A^\eff(G)^{op}\to\bC$ be a functor for a complete and cocomplete target $\bC$, let $F$ be as in \cref{jfi3rhfiuofewfewf-fin}, and recall the notation introduced in the \cref{regpoeregregrgr}. We consider an object~$A$  in $\PSh(G\Fin) $. Let
\[p'_{A}\colon \wt F'(A,A)\to \wt F'(A,*)\]
be the map induced by $A\to *$

Analogously to \cref{cor:assembly}, we obtain the following statement.
\begin{kor}
	\label{cor:coassembly}
	Assume:
	\begin{enumerate}
		\item $M$ is a Mackey functor;
		\item $r^*A$ in $\PSh(G\Orb)$ is equivalent to $E_\cF G$.
	\end{enumerate} Then $p'_A$ is an equivalence.
\end{kor}
One can even formally deduce this statement  from  \cref{cor:assembly} by going over to opposite categories in the appropriate way.

There is a canonical morphism \begin{equation}\label{f34kjnjk3v3rvr3}
i\colon G\Orb\to A^{\eff}(G)^{op}
\end{equation} 
which is the obvious inclusion on objects and sends the morphism $f\colon S\to T$ to
the span
\[\xymatrix{&S\ar[dl]_{f}\ar[dr]^{\id_{S}}&\\T&&S}\]

Let $M\colon A^\eff(G)^{op}\to\bC$ be a functor for a complete and cocomplete target $\bC$, let $F$ be as in \cref{jfi3rhfiuofewfewf-fin}, and recall  the notation introduced in \cref{regpoeregregrgr}. We set $E:=i^{*}M$.

\begin{lem} 
	\label{lem:assemblymap}
	The assembly map \eqref{ewoij3wvwev}
	  is equivalent to the morphism 
	\[\alpha\colon \wt F(*, r_{!}E_{\cF}G)\to  \wt F(*, *)\]
	induced by the projection $r_{!}E_{\cF}G\to *$.
	 	 \end{lem}
\begin{proof}
By \cref{cweclkwevcwecw} the assembly map is equivalent to the morphism
\[ \wt E(E_{\cF}G)\to \wt E(*) \] induced by the projection $ E_{\cF}G\to *$. 
The relations
$i(-)\simeq m(*,r(-))$  and   $E\simeq i^{*}M$ now imply that
\[E(-)\simeq F(pt,r(-))\ .\]
We therefore get an equivalence
\[\wt E(-)\simeq \wt F(*,r_{!}(-))\] of colimit-preserving functors from $\PSh(G\Orb)$ to $\bC$.
The assertion is now obvious. 
\end{proof}

Let   $E\colon G\Orb\to \bC$ be a functor and $i$ be as in \eqref{f34kjnjk3v3rvr3}.
\begin{ddd}\label{efioewfwfwefewff}
We say that $E$ \emph{extends to a Mackey functor} if there exists a Mackey functor
$M\colon A^{\eff}(G)^{op}\to \bC$ such that
$i^{*}M\simeq E$.
\end{ddd}

\begin{proof}[Proof of \cref{ihwiuhvewvwevewv}]
Let $M\colon A^{\eff}(G)^{op}\to \bC$ be a Mackey functor 
such that $E\simeq i^{*}M$.  Let $F$ be as in \cref{jfi3rhfiuofewfewf-fin} and recall the notation introduced in \cref{regpoeregregrgr}. 

We define 	the object   $A:=r_!E_\cF G$ of $\PSh(G\Fin)$.  Because the functor $r\colon G\Orb\to G\Fin$ is fully faithful,  we have an equivalence  $r^{*}r_{!} \simeq \id$. 
In particular, we get the equivalence 
  $r^*A\simeq r^*r_! E_\cF G\simeq E_\cF G$.    Since $r_!$ is left adjoint to $r^*$ and $r^*$ preserves colimits, $r_!$ preserves compacts. Therefore, $A$ is a compact object.
	
	We now consider   the diagram
	\[\xymatrix{
		\wt F(*,A)\ar[r]^\alpha \ar[d]^{\simeq}_{p_A}&\wt F(*,*)\ar [dr]_{\simeq}^{c}&\\
		\wt F(A,A)\ar [dr]^{\simeq}_{c}&&\wt F'(*,*)\ar [d]\\
		&\wt F'(A,A)\ar [r]_{\simeq}^{p'_A}&\wt F'(A,*)}	
	\]
	The maps labeled with $p_A$, $p'_A$ and $c$ are equivalences by \cref{cor:assembly}, \cref{cor:coassembly} and \cref{lem:c-map}.
	 The diagram yields a left-inverse of $\alpha$. \cref{ihwiuhvewvwevewv} now follows from \cref{lem:assemblymap}.
	 \end{proof}

To apply the \cref{ihwiuhvewvwevewv} we have to verify the assumption on  the compactness of $E_{\cF}G$.
 Following \cite{oliver}, we introduce the following condition on the family $\cF$.
\begin{ddd}
We call $\cF$ \emph{separating} if for every two subgroups $H$ and $K$ of $G$ such that $H$ is normal in $K$ and $K/H$ is prime cyclic, either both $K$ and $H$ belong to $\cF$, or both are not contained in $\cF$.
\end{ddd}

 \begin{ex}\label{wefiowefwfweewf}
 The family $\Sol$ of solvable subgroups  of $G$ is separating.
 \end{ex}

\begin{theorem}[{\cite[Thm.~4]{oliver}}]
If $\cF$ is a separating family, then there exists a finite $G$-CW-complex of the homotopy type of $E^{top}_{\cF}G$.
\end{theorem}

Note that Oliver's theorem actually states that there exists a disc with a $G$-action
with the correct  homotopy types of fixed point spaces. 
By a theorem of Illman \cite{illman} one can then find a finite $G$-CW-complex
in the same $G$-homotopy type.

\begin{kor}\label{rgiow32rer32r23r23r}
If $\cF$ is a separating family, then
$E_{\cF}G$ is   compact.
\end{kor}
\begin{proof}
We let $G\Top[W^{-1}]$ denote the $\infty$-category obtained from the category of  topological spaces by inverting $G$-weak homotopy equivalences, i.e.,  $G$-maps which induce weak equivalences on the fixed points spaces for all subgroups of $G$. By Elmendorf's theorem we have  the   equivalence  
\[ G\Top[W^{-1}]\simeq  \PSh(G\Orb)\ .\]
Under this equivalence a $G$-CW-complex of the homotopy type of $E^{top}_{\cF}G$
goes to a presheaf equivalent to $E_{\cF}G$. We now note that
a finite $G$-CW-complex represents a compact object in $G\Top[W^{-1}]$
and therefore in  $\PSh(G\Orb) $.
\end{proof}

Therefore, \cref{ihwiuhvewvwevewv} has the following corollary.
 
 Let $G$ be a finite group, let $\cF$ be a family of subgroups of $G$, and let
 $E\colon G\Orb\to \bC$ be a functor.
\begin{kor}\label{wefiowefwfweewf1}
 Assume: 
\begin{enumerate}
\item $\bC$ is stable, complete and cocomplete;
\item  $E$ extends to a Mackey functor;
\item $\cF$ is separating.
\end{enumerate}
Then the assembly map $\Ind_{\cF}\circ \Res_{\cF}(E)(pt)\to E(pt)$ is split injective.
  \end{kor}
By \cref{wefiowefwfweewf} this corollary applies to the family $\cF=\Sol$.

\begin{rem}\label{hioehwvjoewvwevewvew} 
There  should be  a simple proof of  
\cref{wefiowefwfweewf1} in the case of $\bC=\Sp$  based on known facts in equivariant stable homotopy theory.
Let $\beins$ be the tensor unit of the symmetric monoidal category of spectral Mackey functors. We use that
$\pi_{0}(\beins(*))$ is   the Burnside ring $A(G)$ of $G$. 

Following \cite{MR0394711} every finite $G$-CW-complex $X$ represents an element $[X]$ in $A(G)$ with $[X\sqcup Y]=[X]+[Y]$ and $[X\times Y]=[X][Y]$. 
The element $[X]$ only depends on the $G$-homotopy type of $X$.

Let now $E^{top}_{\cF}G$ be a $G$-CW-complex of the homotopy type of the classifying space of the family~$\cF$.
Then   we have a $G$-homotopy equivalence $E_{\cF}^{top}G\times E^{top}_{\cF}G\simeq  E^{top}_{\cF}G$. If
$E^{top}_{\cF}G$ is in addition finite, then $[E^{top}_{\cF}G] $ is a projection in $A(G)$. 

Hence, if there exists a finite $G$-CW-complex in the homotopy type of the classifying space of the family $\cF$,
then we  get a decomposition
$\beins\simeq \beins^{\cF}\oplus \beins_{\cF}$, where $\beins_{\cF}$ is the image of the projection $[E^{top}_{\cF}G]$.
The decomposition of the tensor unit naturally induces a   decomposition of   every  spectral Mackey functor $M\simeq M^{\cF}\oplus M_{\cF}$. 
 In order to relate this with \cref{wefiowefwfweewf1}
   one now has  to check   that the inclusion $M_{\cF}(pt)\to M(pt)$ is equivalent to the assembly map. \end{rem}

\bibliographystyle{amsalpha}
\bibliography{born-transbas}

\providecommand{\bysame}{\leavevmode\hbox to3em{\hrulefill}\thinspace}
\providecommand{\MR}{\relax\ifhmode\unskip\space\fi MR }
\providecommand{\MRhref}[2]{%
  \href{http://www.ams.org/mathscinet-getitem?mr=#1}{#2}
}
\providecommand{\href}[2]{#2}
\begin{thebibliography}{BEKW17b}

\bibitem[Bar17]{Barwick:2014aa}
C.~Barwick, \emph{Spectral {M}ackey functors and equivariant algebraic
  {$K$}-theory ({I})}, Adv. Math. \textbf{304} (2017), 646--727.

\bibitem[BE16]{buen}
U.~Bunke and A.~Engel, \emph{Homotopy theory with bornological coarse spaces},
  \href{https://arxiv.org/abs/1607.03657v3}{arXiv:1607.03657v3}, 2016.

\bibitem[BE17]{ass}
\bysame, \emph{Coarse assembly maps},
  \href{https://arxiv.org/abs/1706.02164}{arXiv:1706.02164}, 2017.

\bibitem[BEKWa]{addcats}
U.~Bunke, A.~Engel, D.~Kasprowski, and Ch. Winges, \emph{Homotopy theory with
  marked additive categories}, in preparation.

\bibitem[BEKWb]{desc}
\bysame, \emph{{Injectivity results for coarse homology theories}}, in
  preparation.

\bibitem[BEKW17a]{trans}
\bysame, \emph{Coarse homology theories and finite decomposition complexity},
  \href{https://arxiv.org/abs/1712.06932}{arXiv:1712.06932}, 2017.

\bibitem[BEKW17b]{equicoarse}
\bysame, \emph{{Equivariant coarse homotopy theory and coarse algebraic
  $K$-homology}}, \href{https://arxiv.org/abs/1710.04935}{arXiv:1710.04935},
  2017.

\bibitem[Ben67]{bena}
J.~Benabou, \emph{Introduction to bicategories}, Lecture Notes in Mathematics
  47, Springer-Verlag, 1967, pp.~1--77.

\bibitem[BGS15]{Barwick:2015ab}
C.~Barwick, S.~Glasman, and J.~Shah, \emph{{Spectral Mackey functors and
  equivariant algebraic K-theory (II)}},
  \href{https://arxiv.org/abs/1505.03098}{arXiv:1505.03098}, 2015.

\bibitem[GHN17]{Gepner:2015aa}
D.~Gepner, R.~Haugseng, and Th. Nikolaus, \emph{Lax colimits and free
  fibrations in {$\infty$}-categories}, Doc. Math. \textbf{22} (2017),
  1225--1266.

\bibitem[GM11]{Guillou:2011aa}
B.~Guillou and J.P. May, \emph{Models of {G}-spectra as presheaves of spectra},
  \href{https://arxiv.org/abs/1110.3571}{arXiv:1110.3571}, 2011.

\bibitem[Ill78]{illman}
S.~Illman, \emph{Smooth equivariant triangulations of {$G$}-manifolds for {$G$}
  a finite group}, Math. Ann. \textbf{233} (1978), no.~3, 199--220.

\bibitem[Lur09]{htt}
J.~Lurie, \emph{Higher topos theory}, Annals of Mathematics Studies, vol. 170,
  Princeton University Press, 2009.

\bibitem[Lur14]{HA}
\bysame, \emph{Higher algebra}, available at
  \href{http://www.math.harvard.edu/~lurie/}{www.math.harvard.edu/lurie}, 2014.

\bibitem[ML98]{MR1712872}
S.~Mac~Lane, \emph{Categories for the working mathematician}, second ed.,
  Graduate Texts in Mathematics, vol.~5, Springer-Verlag, 1998.

\bibitem[NS17]{Nikolaus:2017aa}
Th. Nikolaus and P.~Scholze, \emph{On topological cyclic homology},
  \href{https://arxiv.org/abs/1707.01799}{arXiv:1707.01799}, 2017.

\bibitem[Oli76]{oliver}
R.~Oliver, \emph{Smooth compact lie group actions on disks}, Math. Z.
  \textbf{149} (1976), 79--96.

\bibitem[Sch04]{MR2079996}
M.~Schlichting, \emph{{Delooping the $K$-theory of exact categories.}},
  {Topology} \textbf{43} (2004), no.~5, 1089--1103.

\bibitem[tD75]{MR0394711}
Tammo tom Dieck, \emph{The {B}urnside ring of a compact {L}ie group. {I}},
  Math. Ann. \textbf{215} (1975), 235--250. \MR{0394711}

\end{thebibliography}
\end{document}